\documentclass[11pt]{amsart}
\usepackage{lmodern}
\usepackage{amsmath, amsthm, amssymb, amsfonts}
\usepackage[normalem]{ulem}
\usepackage{hyperref}

\usepackage{mathrsfs}

\usepackage{verbatim} 
\usepackage{longtable}

\usepackage[style=alphabetic,backend=biber]{biblatex}
\usepackage{import}

\usepackage{mathtools}

\usepackage{tikz}
\usetikzlibrary{decorations.pathmorphing}
\tikzset{snake it/.style={decorate, decoration=snake}}

\usepackage{caption}

\usepackage{tikz-cd}
\usetikzlibrary{arrows}

\theoremstyle{plain}
\newtheorem{thm}{Theorem}[section]
\newtheorem{cor}[thm]{Corollary}
\newtheorem{lem}[thm]{Lemma}
\newtheorem{prop}[thm]{Proposition}

\theoremstyle{definition}
\newtheorem{defn}[thm]{Definition}
\newtheorem{example}[thm]{Example}

\theoremstyle{remark}
\newtheorem{rmk}[thm]{Remark}

\newcommand{\BA}{{\mathbb{A}}}

\newcommand{\BC}{{\mathbb{C}}}

\newcommand{\BG}{{\mathbb{G}}}

\newcommand{\BN}{{\mathbb{N}}}

\newcommand{\BP}{{\mathbb{P}}}
\newcommand{\BQ}{{\mathbb{Q}}}

\newcommand{\BZ}{{\mathbb{Z}}}

\newcommand{\CC}{{\mathbb{C}}}

\newcommand{\mc}{\mathcal}
\newcommand{\mf}{\mathfrak}
\newcommand{\ms}{\mathscr}

\newcommand{\pg}{\mf{p}^{\mf{G}}}
\newcommand{\qg}{\mf{q}^{\mf{G}}}
\newcommand{\ptau}{\prescript{\mf p}{}{\tau}}
\newcommand{\pH}{\prescript{\mf p}{}{\mc H}}

\DeclareMathOperator{\CH}{CH}
\DeclareMathOperator{\Tor}{Tor}
\DeclareMathOperator{\td}{td}
\DeclareMathOperator{\Td}{Td}
\DeclareMathOperator{\ttd}{\widetilde{td}}
\DeclareMathOperator{\tch}{\widetilde{ch}}
\DeclareMathOperator{\tCh}{\widetilde{Ch}}
\DeclareMathOperator{\Ch}{Ch}
\DeclareMathOperator{\fix}{fix}
\DeclareMathOperator{\mov}{mov}
\DeclareMathOperator{\Aut}{Aut}
\DeclareMathOperator{\Mod}{Mod}
\DeclareMathOperator{\Corr}{Corr}
\DeclareMathOperator{\Prym}{Prym}
\DeclareMathOperator{\oPrym}{\overline{\Prym}}
\DeclareMathOperator{\oJ}{\overline{J}}
\DeclareMathOperator{\J}{J}
\DeclareMathOperator{\PMod}{PMod}
\DeclareMathOperator{\codim}{codim}
\DeclareMathOperator{\SL}{SL}
\DeclareMathOperator{\PGL}{PGL}
\DeclareMathOperator{\GL}{GL}
\DeclareMathOperator{\Nm}{Nm}

\DeclareMathOperator{\Gal}{Gal}
\DeclareMathOperator{\sm}{sm}
\DeclareMathOperator{\Supp}{Supp}

\newcommand{\ch}{{\mathrm{ch}}}

\DeclareFontFamily{OT1}{rsfs}{}
\DeclareFontShape{OT1}{rsfs}{n}{it}{<-> rsfs10}{}
\DeclareMathAlphabet{\curly}{OT1}{rsfs}{n}{it}

\newcommand\Hom{\operatorname{Hom}}

\newcommand\Spec{\operatorname{Spec}}

\newcommand{\Coh}{\mathrm{Coh}}
\newcommand{\Pic}{\mathop{\rm Pic}\nolimits}

\usepackage{tikz}
\usepackage{lmodern}
\usetikzlibrary{decorations.pathmorphing}

\bibliography{main}

\begin{document}
\title{Fourier duality on compactified Prym fibrations}
\date{\today}

\author[A. Larsen]{Anne Larsen}
\address{Massachusetts Institute of Technology}
\email{annelars@mit.edu}

\begin{abstract}
We study Fourier-Mukai duality for a class of compactified Prym fibrations including moduli spaces of $\SL$ Higgs bundles over the elliptic locus. This leads to a shadow of the Hausel-Thaddeus conjecture, proof of the Corti-Hanamura motivic decomposition conjecture for these fibrations, and multiplicativity of the perverse filtration.

Our approach can be described as a generalization of the Maulik-Shen-Yin \cite{MSY} package for compactified Jacobian fibrations to the case when the dual abelian fibration is a stack. This forces us to move beyond the case of full supports treated in \cite{MSY}. Technical tools include a new pullback identity for the stacky Grothendieck-Riemann-Roch tau functor introduced by To\"en \cite{Toen}, results on descent of (Arinkin-)Poincar\'e sheaves, and a comparison of Poincar\'e sheaves of the Prym varieties of families of smooth and nodal curves along the lines of \cite{FHHO}.
\end{abstract}

\maketitle

\setcounter{tocdepth}{1} 

\tableofcontents
\setcounter{section}{-1}

\section{Introduction}

\subsection{Overview}
Fix a smooth complex projective curve $C$ of genus $g \ge 2$. The \textit{perverse filtration} on the cohomology of the moduli space $M_{\GL_n}(C)$ of semistable (rank $n$, degree $d$) Higgs bundles on $C$ has been an object of intense study over the past fifteen years, motivated especially by the $P=W$ conjecture of \cite{PW}, as well as connections with enumerative geometry \cite{CDP}. One consequence of the $P=W$ conjecture, which was also historically the last step in its proof, is the statement that the perverse filtration is multiplicative with respect to cup product.

Initial proofs of the multiplicativity of the perverse filtration, in \cite{MSPW} and \cite{HMMS}, seemed to rely heavily on the specific topology of the Hitchin fibration. However, recent work of Maulik-Shen-Yin \cite{MSY} proves that a similar result holds for all \textit{dualizable abelian fibrations}, a notion of compactified abelian schemes admitting a (derived-equivalent) dual compactified abelian scheme. (The main examples to date are compactified Jacobians of families of integral, locally planar curves; this includes in particular an open subset of $M_{\GL_n}(C)$, which is the compactified Jacobian of the family of \textit{spectral curves} on the smooth surface $|T^*C|$. But folklore now expects similar results to hold in greater generality.)
%Folklore now suggests that that such a result should hold for an even more general class of abelian fibrations (see \cite{MSY-survey} for further discussion).

This work is a generalization of the \cite{MSY} framework to the case when the dual abelian fibration of $A$ exists not as a variety, but as a stack. As noted, e.g., in \cite{HT}, this situation naturally arises when one considers moduli spaces $M_{\SL_n}(C)$ of semistable, fixed degree $\SL_n$-Higgs bundles on a curve $C$ (which form a compactified Prym fibration of the family of spectral curves over $C$). In this case, the correct dual is the stack quotient $M_{\PGL_n}(C) := [M_{\SL_n}(C)/\Gamma]$, where the finite abelian group of $n$-torsion line bundles $\Gamma := \J^0(C)[n]$ acts on the moduli space of $\SL_n$-Higgs bundles by tensor product.

Using additionally the hyperk\"ahler structures on these moduli spaces, Hausel and Thaddeus \cite{HT} suggested that $M_{\SL_n}(C)$ and $M_{\PGL_n}(C)$ give an example of (noncompact) SYZ mirror partners. As a result, they conjectured, and proved in the case $n \le 3$, an agreement of (stringy) Hodge numbers. (The conjectured agreement of Hodge numbers is trivial in the case of $M_{\GL_n}(C)$, which is self-dual. Thus the pair $M_{\SL_n}(C)$ and $M_{\PGL_n}(C)$ provides the simplest nontrivial example for which to test such a mirror symmetry conjecture.) The Hausel-Thaddeus conjecture was verified for all $n$ in \cite{GWZ}, using $p$-adic integration, and in \cite{MS}, using the Ng\^o endoscopic correspondence of \cite{Yun}. In this paper, Fourier-Mukai transform provides another perspective on the relationship between the cohomology of the moduli space $M_{\SL_n}(C)$ and the stringy cohomology (i.e., the cohomology of the inertia stack) of its dual $M_{\PGL_n}(C)$.

\subsection{Main theorem}
Given a family of integral curves with planar singularities $X \to B$ over a smooth base $B$, equipped with a flat degree $n$ morphism $X \to C \times B$ over $B$, we define the compactified Prym fibration $\pi: \oPrym(X/B) \to B$ to be the kernel of the relative norm map $\Nm_{X/C}: \oJ^0(X/B) \to \J^0(C) \times B$. The dual fibration is defined to be the stack quotient $\oPrym(X/B)^\vee := [\oPrym(X/B)/\Gamma]$, for $\Gamma := \J^0(C)[n]$. The inertia stack of the dual is then the disjoint union of fixed loci
$$I\oPrym(X/B)^\vee = \coprod_{\gamma \in \Gamma} [\oPrym(X/B)^{\gamma}/\Gamma].$$

We call such a compactified Prym fibration \textit{good} if it satisfies the list of conditions in Definition \ref{defn: good fibration}: some mild assumptions on flatness and smoothness, two support conditions satisfied by weak abelian fibrations (in the sense of Ng\^o), and a technical hypothesis that the generic curve of each support be at worst nodal.  This definition is in particular satisfied by $\SL_n$ Higgs-type moduli spaces (with $d=0$) over the elliptic locus, see Proposition \ref{prop: SL Higgs is good}.

\begin{thm} \label{thm: main intro}
    Let $\pi: M \to B$ be a good compactified Prym fibration.\\
    (a) There is an isomorphism of complex mixed Hodge structures up to Tate twists $H^*(M; \CC) \cong H^*(IM^\vee; \BC)$.\\
    (b) There is a motivic decomposition of $R\pi_* \BQ_M$ into its shifted perverse components.\\
    (c) The perverse filtration $P_\bullet H^*(M; \BC)$ is multiplicative with respect to cup product.
\end{thm}

The first statement is a simple consequence of the existence of a derived equivalence between $M$ and $M^\vee$, an enhancement of the nonrelative case treated in \cite{Popa} rather than a consequence of the \cite{MSY} machine; however, we include it as a shadow of the Hausel-Thaddeus conjecture. The second part (see \ref{thm: main sec 2}(i) for a more precise statement) is a Prym version of \cite[Theorem 0.3]{MSY}, and proves the Corti-Hanamura motivic decomposition conjecture \cite{Corti-Hanamura} for $\pi: M \to B$. The third statement expands the known list of abelian fibrations whose perverse filtrations are multiplicative. (Conditions under which one should expect multiplicativity or, stronger, the existence of a multiplicative splitting are the subject of ongoing research; see the survey \cite{MSY-survey} for more discussion and references.)

\begin{rmk}
    The Hitchin fibration for $M_{\SL_n}(C)$ is not a good compactified Prym fibration as defined above; it is a compactified Prym fibration for a family of curves including non-integral (both non-irreducible and non-reduced) members. In general it appears difficult to extend Arinkin's Poincar\'e sheaf, which plays a key role in the proof, to a family of curves with non-reduced members. In \cite[\S 5.4]{MSY}, for the case of $M_{\GL_n}(C)$, this difficulty is circumvented using Springer theory and suitable moduli spaces of parabolic Higgs bundles.
    This argument does not immediately generalize to $\SL_n$, but in upcoming work we will present a different proof of the multiplicativity of the perverse filtration on $H^*(M_{\SL_n}(C))$. It is also worth pointing out that in contrast to the case of $\GL_n$, the multiplicativity statement is not enough to complete the proof of the $P=W$ conjecture for $\SL_n$, which remains open in general.
    (Indeed, as far as the author is aware, there is currently no literature on the weight filtration on the non-$\Gamma$-fixed classes of the $\SL_n$ character variety when $n$ is composite, other than the $E$-polynomial computation in \cite{Mereb}.)
\end{rmk}

\begin{rmk} \label{rmk: LSV}
    Our use of the term \textit{compactified Prym fibration} for the situation of a family of (possibly singular) curves with a (not necessarily \'etale) degree $n$ morphism to a fixed smooth base curve should be contrasted with the compactified Prym fibrations discussed, e.g., in \cite{LSV}, for an \'etale double cover of a family of (possibly singular, non-fixed) base curves. The recent preprint \cite{Huishi} shows that an \'etale double cover-type compactified Prym is a dualizable abelian fibration, i.e., that the results of \cite{MSY} apply directly. (Although the two definitions of Prym variety are not entirely comparable except in the case of an \'etale double cover of a smooth base curve, an important difference is that the \cite{LSV} definition includes only the connected component of the identity, whereas ours takes the whole kernel of the norm map. Taking only the identity component has the effect of restricting the discussion to abelian schemes, rather than abelian scheme extensions of finite groups, so that the \cite{MSY} setup suffices for the purposes of \cite{Huishi}.)
\end{rmk}

\subsection{Approach}
The starting point of this paper, and before it of \cite{MSY}, is the observation that if $A \to S$ is an abelian scheme over a smooth base $S$, with dual $A^\vee := \Pic^0(A/S)$, then cohomological Fourier-Mukai transform by the Chern characters of the \textit{Poincar\'e line bundle} (the normalized universal line bundle) on $A \times_S A^\vee$ recovers the perverse (Leray) filtration on $H^*(A)$. The original proof in \cite{Deninger-Murre} relied heavily on the abelian group structure of the fibers of $A \to S$.

Given a family of integral curves $X \to B$ as described before, Arinkin \cite{Arinkin} constructed a \textit{Poincar\'e sheaf} on the compactified Jacobians $\oJ^0(X/B) \times_B \oJ^0(X/B)$ as an extension of the Poincar\'e line bundle over the locus $U \subset B$ for which $X_U \to U$ is smooth. The insight of \cite{MSY} was that cohomological correspondences defined in terms of the Chern characters of this sheaf exactly as in the smooth case would continue to recover the perverse filtration, given the (somewhat restrictive) condition of \textit{full supports}, which in the end allows one to check compatibility with the perverse filtration on the open subset $U$.

We follow the same approach, defining cohomological correspondences in terms of (suitable stacky) Chern characters of the (suitably descended) Arinkin-Poincar\'e sheaf and checking that they recover the perverse filtration as in \cite{MSY}. The main difficulty is that we can no longer rely on the full support condition in this setting; this is related to the fact that the various connected components of the inertia stack $IM^\vee$ can map onto proper subsets of the base $B$. Thus, instead of reducing to the statement of \cite{Deninger-Murre} for abelian schemes, we are forced to consider compactified Prym fibrations of families of nodal curves over $C$. To check whether our projectors recover the perverse filtration in this setting, we use a comparison of the Poincar\'e sheaves on the compactified Prym variety of a family of nodal curves and of a family of resolutions, along the lines of \cite{FHHO}. This reduces the problem to the Prym variety of a family of smooth curves, which is an extension of a finite group by an abelian scheme; finally, to reduce from this to abelian schemes, we use the duality for commutative abelian group stacks described in \cite{Brochard}.

One technical challenge takes up a significant fraction of the paper: although a singular, stacky Grothendieck-Riemann-Roch $\tau$ functor, taking values in the Chow groups of the inertia stack and generalizing the Fulton ``Todd-twisted Chern character" functor of \cite[\S 18]{Fulton}, was defined by To\"en in \cite{Toen}, we were unable to find in the literature a full list of the properties needed to carry out our argument. (In particular, there was a question of compatibility of $\tau$ with pullback by a representable regular embedding such that the induced embedding of inertia stacks is no longer regular.) This problem is treated extensively in \S \ref{sec: GRR}, and we hope that the main theorem of that section (Theorem \ref{thm: tau properties}) may be of some independent interest.

\subsection{Plan of the paper}
In \S \ref{sec: GRR}, we start by developing the theory of the stacky Grothendieck-Riemann-Roch $\tau$ functor on global quasiprojective quotients by finite abelian groups, along the lines of \cite{Toen} and \cite{Edidin-RRforDM}. The main goal is to establish compatibility of this functor with some representable regular embeddings of singular stacks. However, we also (re)prove various other identities necessary for generalizing the arguments of \cite{MSY} to the stacky case (see Theorem \ref{thm: tau properties}). A reader willing to take these for granted can skip this section, returning to \S \ref{subsec: stacky tau background and def} for definitions as needed.

In \S \ref{sec: main theorem}, we recall the necessary background and state a more precise form of the main theorem. Definitions relating to relative correspondences and the perverse filtration can be found in \S \ref{subsec: stacky MSY background}. In \S \ref{subsec: compactified Prym} we discuss compactified Prym fibrations, and \S \ref{subsec: Poincare} recalls the construction of the Arinkin Poincar\'e sheaf for Pryms. With this background, in \S \ref{subsec: main theorem} we define the projectors used to prove Theorem \ref{thm: main intro}(b) and (c), which are restated in a sharpened form in Theorem \ref{thm: main sec 2}. We prove Theorem \ref{thm: main intro}(a) in \S \ref{subsec: proof of (a)}.

In \S \ref{sec: homological realization}, we show the compatibility of the projectors defined in \S \ref{sec: main theorem} with the perverse filtration. In \S \ref{subsec: hom realization smooth} we generalize the result of \cite{Deninger-Murre} from abelian schemes to extensions of a finite abelian group by an abelian scheme. In \S \ref{subsec: FHHO Pryms in families} we compare the Poincar\'e sheaf of a family of nodal curves to that of a family of normalizations, proving a relative, Prym version of the identity of \cite[Theorem 5.3]{FHHO}. With these preliminaries, in \S \ref{subsec: proof of realization} we prove that the projectors defined in \S \ref{sec: main theorem} indeed recover the perverse filtration.

Finally, in \S \ref{sec: projectors and mult} we turn to the proofs of Theorem \ref{thm: main intro}(b) and (c). In \S \ref{subsec: FV} we prove an analogue of the Fourier vanishing result of \cite[\S 3]{MSY}, using the Prym version of the needed Arinkin-style convolution kernel dimension bounds established in \S \ref{subsec: supp K}. We complete the proofs of Theorem \ref{thm: main intro}(b) and (c) in \S \ref{subsec: pf of projectors} and \ref{subsec: pf of multiplicativity}, respectively.

\textbf{Acknowledgements:} I thank my advisor, Davesh Maulik, for suggesting the problem and for years of helpful conversations. Thanks also to Mark de Cataldo for pointing out a mistake in an initial version of Proposition \ref{prop: SL Higgs is good}, and to David Zhiyuan Bai for a discussion of seesaw lemmas. This work was partly supported by an NSF Graduate Research Fellowship under grant no. 2141064 and a Simons Investigator Grant. AI was not used in the preparation of this paper.

\section{Stacky Grothendieck-Riemann-Roch} \label{sec: GRR}
Grothendieck-Riemann-Roch theorems for stacks have been intensively studied (see, e.g., \cite{Kock}, \cite{Toen}, \cite{EG-RRequivChow} and subsequent papers, \cite{Roy} and subsequent, \cite{Khan}).
However, the author is not aware of a reference that describes the compatibility of a generalization of the ``Todd-twisted Chern character" $\tau$ functor of \cite[\S 18]{Fulton} with pullback by a regular embedding of singular Deligne-Mumford stacks, which is crucial for an argument along the lines of \cite{MSY}. (Most references seem not to treat compatibility with pullback at all, and \cite[Th\'eor\`eme 4.10]{Toen} only addresses representable \'etale pullback, though using the properties established there it is not hard to give a pullback formula for a morphism of smooth stacks, see the proof of Theorem \ref{thm: tau properties}(d).)

In \S \ref{subsec: stacky tau background and def} we define the stacky Chern class, Todd class, and $\tau$ functors for global quotients of quasiprojective varieties by finite abelian groups; our definitions are somewhat more concrete, though equivalent to, those of \cite{Toen}. Then \S \ref{subsec: stacky tau identities} is dedicated to establishing the various needed properties of $\tau$ on such quotients, summarized in Theorem \ref{thm: tau properties}. (Various results also hold in the generality of DM stacks, but for simplicity of exposition, we restrict ourselves to this case.) Some readers may prefer to skip this section, taking these foundational results for granted.

\subsection{Background and definitions} \label{subsec: stacky tau background and def}
In this section, we give a construction of the stacky $\tau$ functor, in our somewhat restricted context:
\begin{defn} \label{defn: nice quotient}
    By \textit{nice quotient} we will mean a quotient stack $[X/G]$, where $X$ is a quasiprojective complex variety equipped with an action of a finite abelian group $G$. Given a stack $\mc X$, we use \textit{nice presentation} to mean a global quotient presentation $\mc X = [X/G]$ of this form.
\end{defn}
Throughout this section, we will use $X, Y, \ldots$ to denote varieties; $G, H, \ldots$ for finite abelian groups; and $\mc X, \mc Y, \ldots$ for the resulting quotient stacks.

\subsubsection{Equivariant $K$-theory} \label{subsubsec: equivar K-theory}
Let $K_0(\mc X)$, respectively $K^0(\mc X)$, be the Grothendieck group of the exact category of $G$-equivariant coherent, respectively locally free, sheaves on $X$ (or equivalently, of coherent, respectively locally free, sheaves on the quotient stack $\mc X$). These have the usual functorialities of $K$-theory:
\begin{enumerate}
    \item Tensor product equips $K^0(\mc X)$ with the structure of a commutative ring, and $K_0(\mc X)$ with the structure of a $K^0(\mc X)$-module.
    \item There is a natural map $K^0(\mc X) \to K_0(\mc X)$. If $\mc X$ is smooth, this map is an isomorphism \cite[Corollary 5.8]{Thomason-foundations}.
    \item Given a $G$-equivariant morphism $f: X \to Y$, we obtain a pullback morphism $f^*: K^0(\mc Y) \to K^0(\mc X)$. If $f$ is $\Tor$-finite, there is also a pullback morphism $f^*: K_0(\mc Y) \to K_0(\mc X)$. More generally, if $f: \mc X \to \mc Y$ is any morphism of stacks (not necessarily representable), we have a pullback $f^*: K^0(\mc Y) \to K^0(\mc X)$.
    \item In particular, taking $Y = \Spec \BC$, we obtain a morphism $K^0(BG) \to K^0(\mc X)$, giving $K_0(\mc X)$ the structure of a $K^0(BG)$-module. Note that $K^0(BG) \cong \BZ[G^\vee]$ is the representation ring of $G$, or alternatively, since $G$ is finite abelian, the group ring of $G^\vee := \Hom(G, \BC^*)$.
    \item Given a proper $G$-equivariant morphism $f: X \to Y$, we obtain a morphism $f_*: K_0(\mc X) \to K_0(\mc Y)$. 
    \item In the situation of (5), given $\mc F \in K_0(\mc X)$ and $\mc E \in K^0(\mc Y)$, the canonical (and therefore $G$-equivariant) isomorphism of sheaves of \cite[Tag 08EU]{stacks-project}  gives a projection formula in equivariant $K$-theory $f_*(f^* \mc E \cdot \mc F) = \mc E \cdot f_* \mc F$.
    \item More generally, given any proper morphism of stacks $f: \mc X \to \mc Y$, not necessarily representable, we have a pushforward $f_*: K_0(\mc X) \to K_0(\mc Y)$ defined in terms of higher pushforwards of coherent sheaves on $\mc X$. Again, the projection formula holds \cite[Proposition 3.7]{Khan}.
    \item Given a $G$-equivariant closed embedding $X \hookrightarrow Y$, we define the relative $K$-group $K^0_{\mc X}(\mc Y)$ to be the Grothendieck group of $G$-equivariant finite complexes of locally free sheaves on $Y$ with cohomology supported on $X$. Given a $G$-equivariant morphism $f: Y \to Z$ and $A \subset Z$ with $f^{-1}(A) \subset X$ there is a pullback
    $f^*: K^0_{\mc A}(\mc Z) \to K^0_{\mc X}(\mc Y).$
    Furthermore, given another closed embedding $B \hookrightarrow Y$, tensor product gives
    $$\cap: K^0_{\mc X}(\mc Y) \times K_0(\mc B) \to K_0(\mc X \cap \mc B).$$
    \item Linearly extending the function $[\mc E] \mapsto [\wedge^k \mc E]$ of $G$-equivariant vector bundles $\mc E$, we obtain a well-defined map $\lambda^k: K^0(\mc X) \to K^0(\mc X)$, and the collection of maps $\lambda^k$ satisfies the properties of a $\lambda$-ring \cite[Lemma 2.4]{Kock}. We also define an operation $\lambda_{-1} := \sum_{i \ge 0} (-1)^i \lambda^i$ on classes $v \in K^0(\mc X)$ such that $\lambda^i(v) = 0$ for $i$ sufficiently large. By definition of a $\lambda$-ring, we have $\lambda_{-1}(\mc E + \mc F) = \lambda_{-1}(\mc E) \cdot \lambda_{-1}(\mc F)$ for all $\mc E, \mc F \in K^0(\mc X)$ on which $\lambda_{-1}$ is defined.
    \item Given a $G$-equivariant pullback square
    \[
    \begin{tikzcd}
        Z \arrow[r] \arrow[d]& W \arrow[d] \\
        X \arrow[r, "f"] & Y
    \end{tikzcd}
    \]
    and an $f$-perfect, $G$-equivariant complex $\ms F \in D^b\Coh(X)$, there is a refined Gysin map $f^{\ms F}: K_0(\mc W) \to K_0(\mc Z)$ defined by $[\ms G] \mapsto \sum_i (-1)^i [\Tor_i^Y(\ms F, \ms G)]$,
    satisfying various properties analogous to those of the refined Gysin map on Chow groups (see \cite[\S 3]{Anderson-Payne}). If $\ms F = \ms O_X$, we write $f^{\ms O_X} =: f^!$; if further $W = Y$, then $f^! = f^*$.
    \item Given $\mc X = [X/G]$ and $\mc Y = [Y/H]$, we have an exterior product
$\boxtimes: K_0(\mc X) \times K_0(\mc Y) \to K_0(\mc X \times \mc Y)$
defined on a $G$-equivariant sheaf $\mc F$ on $X$ and an $H$-equivariant sheaf $\mc G$ on $Y$ by
$$[\mc F] \boxtimes [\mc G] := [p_1^* \mc F \otimes p_2^* \mc G]$$
where $p_1^* \mc F \otimes p_2^* \mc G$ is the sheaf on $X \times Y$ with induced $G \times H$-equivariant structure. (To see that this descends to the Grothendieck group, note that $p_1^* \mc F \otimes p_2^* \mc G = Lp_1^* \mc F \otimes^L Lp_2^* \mc G$.)
\end{enumerate}

The following seems to be well-known, but for lack of a reference, we include a proof:
\begin{lem} \label{lemma: pushforward exterior product}
    If $i: X \to X'$, respectively $j: Y \to Y'$, is a proper and $G$-, respectively $H$-, equivariant morphism of quasiprojective varieties, then
    $$(i \times j)_*(\mc F \boxtimes \mc G) = (i_* \mc F) \boxtimes (j_* \mc G) \in K_0(\mc X' \times \mc Y')$$
    for $\mc F \in K_0(\mc X)$ and $\mc G \in K_0(\mc Y)$.
\end{lem}

\begin{proof}
    We start with the corresponding statement at the level of non-equivariant complexes of sheaves.
    Consider the Cartesian squares
    \[\begin{tikzcd}
	{X \times Y} & {X' \times Y} & Y \\
	{X \times Y'} & {X' \times Y'} & {Y'} \\
	X & {X'} & {\Spec \BC}
	\arrow[from=1-1, to=1-2]
	\arrow[from=1-1, to=2-1]
	\arrow[from=1-2, to=1-3]
	\arrow[from=1-2, to=2-2]
	\arrow[from=1-3, to=2-3, "a"]
	\arrow[from=2-1, to=2-2]
	\arrow[from=2-1, to=3-1]
	\arrow[from=2-2, to=2-3]
	\arrow[from=2-2, to=3-2]
	\arrow[from=2-3, to=3-3, "b"]
	\arrow[from=3-1, to=3-2, "c"]
	\arrow[from=3-2, to=3-3, "d"]
\end{tikzcd}\]
    (where we use the same letter to denote all parallel arrows, and all functors are assumed derived). Then, for $\ms F \in D^b\Coh(X)$ and $\ms G \in D^b\Coh(Y)$, we have
    $$(i \times j)_*(\ms F \boxtimes \ms G) := c_* a_*((a^*b^*\ms F) \otimes (c^*d^* \ms G)) = c_*(b^* \ms F \otimes a_*c^*d^*\ms G)$$
    by the projection formula \cite[Tag 08EU]{stacks-project}. Next, note that $X \times Y'$ and $X' \times Y$ are $\Tor$-independent over $X' \times Y'$, and so we have the base change formula $a_* c^* = c^* a_*$ \cite[Tag 08IB]{stacks-project}. Thus
    $$ c_*(b^* \ms F \otimes a_*c^*d^*\ms G) = c_*(b^* \ms F \otimes c^*a_*d^*\ms G) = (c_*b^* \ms F) \otimes (a_*d^* \ms G)$$
    by another application of the projection formula. Now, as $d$ and $b$ are flat, by another application of base-change, we have
    $$(c_* b^* \ms F) \otimes (a_* d^* \ms G) = (b^* c_* \ms F) \otimes (d^* a_* \ms G) = c_* \ms F \boxtimes a_* \ms G =: i_* \ms F \boxtimes j_* \ms G.$$
    Finally, note that since all isomorphisms involved are canonical, they respect the $(G \times H)$-equivariant structure.
\end{proof}

We end with the features unique to equivariant $K$-theory. First, if $g \in G$ acts trivially on $X$, we note that any $\mc E \in K_0(\mc X)$ or $K^0(\mc X)$ can be uniquely decomposed into a sum of $g$-eigensheaves $\sum_\zeta \mc E_\zeta$, where $g$ acts on $\mc E_\zeta$ with eigenvalue $\zeta$. We write $\mc E^{\fix} := \mc E_1$ and $\mc E^{\mov} := \sum_{\zeta \ne 1} \mc E_\zeta$. We also write $K_0(\mc X) = K_0(\mc X)^{\fix} \oplus K_0(\mc X)^{\mov}$ for the corresponding decomposition of $K_0(\mc X)$ (and similarly for $K^0(\mc X)$).

Lastly, we come to fixed point localization. Given $g \in G$, let $\mf m_g \subset K^0(BG) \cong \BZ[G^\vee]$ be the prime ideal of character sums vanishing on $g$. For any $K^0(BG)$-module $M$ we write $M_{\mf m_g}$ for the localization of $M$ at $\mf m_g$.
\begin{prop} \label{prop: Thomason localization}
    \cite[Th\'eor\`eme 2.1]{Thomason} Given $g \in G$, the inclusion of the fixed locus $\iota: X^g \hookrightarrow X$ induces an isomorphism $\iota_*: K_0(\mc X^g)_{\mf m_g} \xrightarrow{\sim} K_0(\mc X)_{\mf m_g}$.
\end{prop}
We have the following description of the inverse isomorphism, extending \cite[Lemme 3.3]{Thomason} (which is the case $X=Y$):
\begin{lem} \label{lem: localization inverse}
    Let $G$ be a finite abelian group and $f: X \hookrightarrow Y$ a $G$-equivariant closed embedding of quasiprojective varieties, with $Y$ being smooth. Let $f_g: X^g \to Y^g$ be the restriction to the fixed loci and $\iota_Y: Y^g \hookrightarrow Y$ the inclusion. Then
    $$f_g^* (\lambda_{-1}(\mc N^\vee_{\iota_{\mc Y}})^{-1}) \cdot \iota_{\mc Y}^!: K_0(\mc X)_{\mf m_g} \to K_0(\mc X^g)_{\mf m_g}$$
    is inverse to the isomorphism $\iota_{\mc X*}$, where $\mc N_{\iota_{\mc Y}}^\vee \in K^0(\mc Y^g)$ is the class of the conormal bundle and $\lambda_{-1}(\mc N^\vee_{\iota_{\mc Y}})^{-1} \in K^0(\mc Y^g)_{\mf m_g}$ is the unique class such that $\lambda_{-1}(\mc N^\vee_{\iota_{\mc Y}})^{-1} \cdot \lambda_{-1}(\mc N^\vee_{\iota_{\mc Y}}) = 1 \in K^0(\mc Y^g)_{\mf m_g}$ \cite[Lemme 3.3]{Thomason}.
\end{lem}

\begin{proof}
    First, since $Y$ is smooth, we see that $\ms O_{Y^g}$ is an $\iota_Y$-perfect complex, with $X \times_Y Y^g = X^g$, so that the above gives a well-defined morphism $K_0(\mc X) \to K_0(\mc X^g)$. Moreover, both the Gysin morphism and the $K^0(\mc X^g)$-action commute with the action of $K^0(BG)$, and we obtain a morphism of $\mf m_g$-localizations.

    Also, since $\iota_{\mc X*}$ is an isomorphism and $\lambda_{-1}(\mc N^\vee_{\iota_{\mc Y}}) \in K^0(\mc Y^g)_{\mf m_g}$ is invertible, it suffices to prove that
    \begin{equation} \label{eq: nonsmooth localization}
        \iota_{\mc X*}(\iota_{\mc Y}^!(\iota_{\mc X*}\mc F)) = \iota_{\mc X*}(f_g^* (\lambda_{-1}(\mc N^\vee_{\iota_{\mc Y}})) \cdot  \mc F)
    \end{equation}
    for any $\mc F \in K_0(\mc X^g)_{\mf m_g}$.
    
    For any $F \in \Coh(X^g)$ we have (with all functors derived)
    \begin{multline} \label{eq: Thomason 3.3}
    \iota_{Y*} \iota_Y^* f_* \iota_{X*} F = \iota_{Y*} \iota_Y^*\iota_{Y*} f_{g*} F = \iota_{Y*} \ms O_{Y^g} \otimes \iota_{Y*} f_{g*} F = \iota_{Y*}(\iota_Y^* \iota_{Y*} \ms O_{Y^g} \otimes f_{g*} F) \\
    = \iota_{Y*} f_{g*}(f_g^*(\iota_Y^* \iota_{Y*} \ms O_{Y^g}) \otimes F) = f_* \iota_{X*}(f_g^*(\iota_Y^* \iota_{Y*} \ms O_{Y^g}) \otimes F)
    \end{multline}
    by repeated applications of the projection formula \cite[08EU]{stacks-project}. Now, viewing both sides of (\ref{eq: Thomason 3.3}) as sheaves on $X \subset Y$, we have
    $$[\iota_{Y*} \iota_Y^* f_* \iota_{X*} F ] = [\iota_{X*}(f_g^*(\iota_Y^* \iota_{Y*} \ms O_{Y^g}) \otimes F)] \in K_0(\mc X).$$
    On the left side, by definition
    $$ \iota_Y^!(\iota_{X*} [F]) = [\iota_Y^* f_* \iota_{X*} F] \in K_0(X^g),$$
    considering the right side as a complex supported on $X^g$, and so since pushing forward by closed immersions doesn't change this complex, we have
    $$\iota_{X*}(\iota_Y^!(\iota_{X*} [F]))) = [\iota_{Y*} \iota_Y^* f_* \iota_{X*} F].$$
    On the right side, we have
    $$[\iota_{X*}(f_g^*(\iota_Y^* \iota_{Y*} \ms O_{Y^g}) \otimes F)] = \iota_{X*}(f_g^*([\iota_Y^* \iota_{Y*} \ms O_{Y^g}] \cdot [F])) = \iota_{X*}(f_g^* (\lambda_{-1}(\mc N_{\iota_Y}^\vee)) \cdot [F])$$
    by \cite[(3.3.3)]{Thomason}. Moreover, since all the isomorphisms involved are canonical, if $F \in \Coh(X^g)$ is $G$-equivariant, then the identities above hold for equivariant sheaves. Since $K_0(\mc X^g)_{\mf m_g}$ is generated over $K^0(BG)_{\mf m_g}$ by the classes of $G$-equivariant sheaves and equation (\ref{eq: nonsmooth localization}) is $K^0(BG)_{\mf m_g}$-linear, this completes the proof.
\end{proof}

\subsubsection{Equivariant Chow groups and the Edidin-Graham $\tau$ functor} \label{subsubsec: EG def of Chow}
Equivariant Chow groups are defined in \cite{EG-equivint}, following the approximation technique of \cite{Totaro}. In the case that $G$ is finite abelian, the definition can be stated as follows: Let $V$ be a $G$-representation and $U \subset V$ the open subset on which $G$ acts freely. Assume that $V$ is such that $\dim V - \dim (V \setminus U) > \dim X$. (For example, decomposing $G \cong \prod_i \BZ/a_i\BZ$ as a product of cyclic groups and letting $\chi_i$ be a generator of $(\BZ/a_i\BZ)^\vee$, with corresponding 1-dimensional representation $V_i$ defined by $G \twoheadrightarrow \BZ/a_i \BZ \xrightarrow{\chi_i} \BC^*$, we may take $V :=\oplus V_i^{\dim X + 1}$.) Since $X \times U$ is quasiprojective and $G$ acts freely, the (stack) quotient $X \times_G U =: X_G$ is a quasiprojective variety (``approximation scheme"), and we define $\CH_i(\mc X) := \CH_{i+\dim V}(X_G)$ for $0 \le i \le \dim X$, and otherwise $\CH_i(X) := 0$. (In other words, we have that $p: X_G \to \mc X$ is an open subset of an affine bundle over $\mc X$ whose complement is of sufficiently high codimension that the pullback on Chow groups is an isomorphism in the relevant degree range.) Chern classes $\ch^{EG}$ taking values in the operational Chow groups $\CH^*(\mc X)$ are defined similarly, pulling back to $X_G$ and using the standard definitions on this space. The equivariant Chow groups are independent of the choice of $V$ \cite[Definition-Proposition 1]{EG-equivint} and have the usual functorialities for $G$-equivariant morphisms $X \to Y$ \cite[Proposition 3]{EG-equivint}, i.e., for representable morphisms of quotient stacks. Making the same choice of $U$ for both $X$ and $Y$, it follows immediately from the corresponding statements on $X_G$ and $Y_G$ that the Chern classes, proper pushforward, and Gysin pullback satisfy the standard compatibilities (e.g. the projection formula).

Proper pushforward and flat pullback in Chow along nonrepresentable morphisms of quotient stacks $f:[X/G] \to [Y/H]$ can be described as follows: Consider the diagram
\begin{equation} \label{diagram: proper pushforward}
\begin{tikzcd}
    X \arrow[d, "q"] \times_{[Y/H]} Y \arrow[rr, "p"] &  & Y\arrow[d, "r"] \\
     X \arrow[r, "s"] & \left [X/G\right ] \arrow[r, "f"] & \left[Y/H\right]
\end{tikzcd}
\end{equation}
Note that $X \times_{[Y/H]} Y \to [X/G]$ is a finite, flat, representable morphism of degree $|G| \cdot |H|$. Since pushforward, respectively pullback, by such a morphism induces a surjection, respectively injection, on rational Chow groups after passing to approximation schemes \cite[Example 1.7.4]{Fulton}, we obtain a surjection $\CH_*(X \times_{[Y/H]} Y)_\BQ \to CH_*([X/G])_\BQ$, respectively injection in the opposite direction. Thus proper pushforward $f_*: \CH_*([X/G])_\BQ \to \CH_*([Y/H])_\BQ$ is determined by the pushforward map $(r \circ p)_*: \CH_*([X \times_{[Y/H]} Y])_\BQ \to \CH_*([Y/H])_\BQ$, which is as described in \cite{EG-equivint} since now $X \times_{[Y/H]} Y \to [Y/H]$ is proper and representable. Similarly, flat pullback $f^*: \CH_*([Y/H])_\BQ \to \CH_*([X/H])_\BQ$ is determined by the flat, representable pullback $(r \circ p)^*: \CH_*([Y/H])_\BQ \to \CH_*(X \times_{[Y/H]} Y)_\BQ$.
(The fact that these are well-defined follows, e.g., from the existence of projective pushforward and flat pullback for Chow groups of Artin stacks, see \cite[\S 2.2]{Kresch}.) As described in \cite[\S 3.6]{Kresch}, nonrepresentable proper pushforward satisfies the usual projection formula.

\begin{defn} \label{defn: tauEG}
    If $\mc X = [X/G]$ is a nice quotient, the Edidin-Graham tau functor $\tau^{EG}$ is defined by
    $$\tau^{EG}: K_0(\mc X) \xrightarrow{p^*} K_0(X_G) \xrightarrow{\tau^F} \CH_*(X_G)_\BQ \rightarrow \CH_{*-\dim V}(\mc X)_\BQ$$
    where the first arrow is $G$-equivariant flat pullback along the map $X \times U \to X$, the second is Fulton's $\tau$ functor \cite[Theorem 18.2]{Fulton}, and the third is the identification in the definition of $\CH_*(\mc X)$ in degrees $* \ge \dim V$ and 0 otherwise.
\end{defn}

\begin{rmk}
    This appears to differ from the definition in \cite[\S 3.2]{EG-RRequivChow} by a factor of $\td^G(V)$. However, since $G$ is finite abelian, $V$ is a sum of irreducible 1-dimensional representations $V_i$, where $\ch_1^G(V_i) = 0 \in \CH_*(\mc X)_\BQ$ since $V_i^{\otimes |G|}$ is the trivial representation.
\end{rmk}

\begin{rmk} \label{rmk: tauEG independent}
    Choosing two nice presentations $\mc X = [X/G] = [Y/H]$, by pulling back to $X_G \times_{\mc X} Y_H$ and using the compatibility of $\tau^F$ with smooth pullback \cite[Theorem 18.3(4)]{Fulton}, we see that $\tau^{EG}$ is in fact independent of the choice of nice presentation.
\end{rmk}

The functor $\tau^{EG}$ satisfies the expected functorialities with respect to $G$-equivariant morphisms $X \to Y$ \cite[Theorem 3.1]{EG-RRequivChow}. However, it does not in general commute with pushforward along nonrepresentable morphisms. The simplest example is the following (cf. \cite[Remarque, \S 4.1]{Toen}):
\begin{example} \label{ex: need stacky tau}
    Consider the nonrepresentable morphism $f: BG \to \Spec \BC$, for $G$ a nontrivial finite abelian group. Let $p:U/G \to BG$ be an approximation scheme, and $V \in K_0(BG)$ the class of a nontrivial 1-dimensional representation. Then
    $$\tau^{EG}(V) = [BG] \in \CH_*(BG)_\BQ = \CH_0(BG)_{\BQ}$$ since
    $$\tau^F(p^* V) = \td(p^* V) \cap \Td(U/G) = [U/G] + \ldots $$
    so in particular
    $$f_* \tau^{EG}(V) = \frac{1}{|G|} [\Spec \BC] \ne 0 \in \CH_*(\Spec \BC)_\BQ,$$
    whereas $f_* V = V^G = 0$ and so $\tau^{EG}(f_* V) = 0$.
\end{example}

One issue is the following:
\begin{prop} \label{prop: EG factors through m1}
    The functors $\ch^{EG}, \td^{EG}: K^0(\mc X) \to \CH^*(\mc X)_\BQ$, respectively $\tau^{EG}: K_0(\mc X) \to \CH_*(\mc X)_\BQ$, factor through the projection $K^0(\mc X) \to K^0(\mc X)_{\mf m_1}$, respectively $K_0(\mc X) \to K_0(\mc X)_{\mf m_1}$.
\end{prop}

\begin{proof}
    Given $\mc F \in \ker(K^0(\mc X) \to K^0(\mc X)_{\mf m_1})$, there must be some $r \in K^0(BG) \setminus \mf m_1$ such that $r \cdot \mc F = 0$. Note that $\ch^{EG} = \mathrm{rk}: K^0(BG) \to \CH^*(BG) = \CH^0(BG) = \BZ$ has kernel $\mf m_1$. Now, letting $q: \mc X \to BG$ be given by the projection $X \to \Spec \BC$, we have
    $$0 = \ch^{EG}(r \cdot \mc F) = q^*\ch^{EG}(r) \cap \ch^{EG}(\mc F) \in \CH^*(\mc X)$$
    and since $q^* \ch^{EG}(r)$ is given by multiplication by a nonzero integer, we conclude that $\ch^{EG}(\mc F) = 0$. Thus $\ch^{EG}$ factors through the projection to $K^0(\mc X)_{\mf m_1}$, and since $\td^{EG}$ is defined in terms of $\ch^{EG}$, the same is true of $\td^{EG}$.

    Finally, the needed statement for $\tau^{EG}$ is \cite[Corollary 5.1]{EG-RRequivChow}.
\end{proof}
In the next section we will define a $\tau$ functor which incorporates the projections to all the components $K_0(\mc X)_{\mf m_g}$, not only $K_0(\mc X)_{\mf m_1}$, whose formation is in general not compatible with nonrepresentable proper pushforward.

\subsubsection{The inertia stack and stacky tau functor} \label{subsubsec: stacky tau}
In order to define a stacky tau functor compatible with nonrepresentable proper pushforward, we must pass to the \textit{inertia stack}.
Recall that the inertia stack $I \mc X$ of a stack $\mc X$ is defined to be the fiber product $\mc X \times_{\mc X \times \mc X} \mc X$, where each map $\mc X \to \mc X \times \mc X$ is the diagonal; this parametrizes points of $\mc X$ with an automorphism, i.e., pairs $(f \in \mc X(S), \alpha \in \Aut_{\mc X(S)}(f))$. Let $\iota: I\mc X \to \mc X$ be given by (either) projection, i.e., by forgetting the automorphism. Note that any morphism of stacks $f: \mc X \to \mc Y$ induces a morphism of inertia stacks $If: I \mc X \to I \mc Y$ compatible with $f$ under the projections $\iota_{\mc X}$ and $\iota_{\mc Y}$.

When $\mc X = [X/G]$ is a global quotient by a finite abelian group $G$, the inertia stack admits a simple description \cite[Exercise 4.2.14(c)]{Alper}
\begin{equation} \label{eq: inertia stack description}
   I \mc X \cong \coprod_{g \in G} [X^g/G]. 
\end{equation}
The projection $\iota: I \mc X \to \mc X$ is given by the $G$-equivariant maps $X^g \hookrightarrow X$ for each $g$.

We start by noting that
$$K^0(BG)_\BC = \BZ[G^\vee] \otimes_\BZ \BC = \ms O_G(G) = \prod_{g \in G} \ms O_G(G)/\mf m_g = \prod_{g \in G} \ms O_G(G)_{\mf m_g},$$
so that the $\ms O_G(G)_{\mf m_g}$ are canonically both quotient rings and subrings of $K^0(BG)_\BC$. Since $K_0(\mc X)_\BC$ is a $K^0(BG)_\BC$-module, using
$$\pi_g: K_0(\mc X)_\BC \to K_0(\mc X)_{\mf m_g, \BC}$$
to denote the canonical projection induced by $K^0(BG)_\BC \to \ms O_G(G)_{\mf m_g}$, the above decomposition gives a canonical isomorphism
\begin{equation} \label{eq: pi}
    \pi = (\oplus_g \pi_g): K_0(\mc X)_\BC \xrightarrow{\sim} \oplus_{g \in G} K_0(\mc X)_{\mf m_g, \BC}.
\end{equation}

By the fixed-point localization theorem (Proposition \ref{prop: Thomason localization}) we have an isomorphism
\begin{equation} \label{eq: iota}
    \iota_* = (\oplus_g\iota_{g*}): \oplus_{g \in G} K_0(\mc X^g)_{\mf m_g, \BC} \xrightarrow{\sim} \oplus_{g \in G} K_0(\mc X)_{\mf m_g, \BC}.
\end{equation}

Next, breaking each $\mc E \in K_0(\mc X^g)$ into $g$-eigensheaves $\mc E_\zeta$, we have a well-defined group isomorphism
$$\rho_g: K_0(\mc X^g)_\BC \to K_0(\mc X^g)_\BC, \sum_\zeta a_\zeta \mc E_\zeta \mapsto \sum_\zeta \zeta a_\zeta \mc E_\zeta$$
compatible with the ring isomorphisms on $K^0(\mc X^g)_\BC, K^0(BG)_\BC$ defined by the same formula (which we also call $\rho_g$).
% \sum a_i \phi_i such that \sum a_i \phi_i(h) = 0
% becomes \sum a_i \phi_i(g) \phi_i
Note that $\rho_g(\mf m_h) = \mf m_{g^{-1}h}$ for any $h \in G$, and so $\rho_g$ acts on the decomposition of $K_0(\mc X)_\BC$ by
\begin{equation} \label{eq: rho components}
    \rho_g: K_0(\mc X^g)_{\mf m_h, \BC} \xrightarrow{\sim} K_0(\mc X^g)_{\mf m_{g^{-1}h, \BC}}.
\end{equation}
In particular, we can define an isomorphism
\begin{equation} \label{eq: rho}
    \rho = (\oplus_g \rho_g): \oplus_{g \in G} K_0(\mc X^g)_{\mf m_g, \BC} \xrightarrow{\sim} \oplus_{g \in G} K_0(\mc X^g)_{\mf m_1, \BC}.
\end{equation}

Finally, by \cite[Corollary 5.1]{EG-RRequivChow}, the Edidin-Graham tau functor $\tau^{EG}: K_0(\mc X^g) \to \CH_*(\mc X^g)_\BQ$ described in Definition \ref{defn: tauEG} in fact factors through an isomorphism $K_0(\mc X^g)_{\mf m_1, \BQ} \xrightarrow{\sim} \CH_*(\mc X^g)_\BQ$, so that we have an isomorphism
\begin{equation}\label{eq: tauEG}
    \tau^{EG} = (\oplus_g \tau^{EG}): \oplus_{g \in G} K_0(\mc X^g)_{\mf m_1, \BC} \xrightarrow{\sim} \oplus_{g \in G} \CH_*(\mc X^g)_\BC.
\end{equation}

\begin{defn} \label{defn: stacky tau}
    Given a nice quotient $\mc X = [X/G]$,
    the stacky tau functor $\tau: K_0(\mc X)_\BC \to \CH_*(I\mc X)_\BC$ is defined by
    $$\tau: K_0(\mc X)_\BC \xrightarrow{\pi} \oplus_{g \in G} K_0(\mc X)_{\mf m_g, \BC} \xrightarrow{\iota_*^{-1}} \oplus_{g \in G} K_0(\mc X^g)_{\mf m_g, \BC} \xrightarrow{\rho} \oplus_{g \in G} K_0(\mc X^g)_{\mf m_1, \BC}$$
    $$\xrightarrow{\tau^{EG}} \oplus_{g \in G} \CH_*(\mc X^g)_\BC = \CH_*(I\mc X)_\BC$$
    where the isomorphisms $\pi, (\iota_*)^{-1}, \rho, \tau^{EG}$ are defined in equations (\ref{eq: pi}), (\ref{eq: iota}), (\ref{eq: rho}), (\ref{eq: tauEG}), respectively, and the last identification comes from the description of $I\mc X$ in (\ref{eq: inertia stack description}).
\end{defn}

\begin{rmk}
    In the smooth case, this is the Edidin-Graham definition as found in, e.g., \cite[Theorem 5.1]{Edidin-RRforDM}, using Thomason's explicit description of $\iota_*^{-1}$ \cite[Lemme 3.3]{Thomason}. We will see in Corollary \ref{cor: agrees with Toen} that this $\tau$ also matches the functor constructed by To\"en in \cite[Th\'eor\`eme 4.10(3)]{Toen} (and thus in particular is independent of the choice of nice presentation). However, the author is not aware of a literature reference for Definition \ref{defn: stacky tau} in this form in the singular case; Edidin-Graham only work with smooth stacks, whereas To\"en uses smooth hyperresolutions to extend his tau functor in the singular case. In any case, this more concrete definition is useful in proving the needed properties of $\tau$.
\end{rmk}

To describe the compatibilities of $\tau$, we will also need definitions of the inertia stack-valued Chern and Todd classes. (These definitions are more or less taken from To\"en \cite[D\'efinitions 4.5, 4.7]{Toen}, although the definition of Todd classes other than $\widetilde{\Td}(\mc X)$ is from \cite[Definition A.0.5]{Tseng}.)

\begin{defn} \label{defn: stacky ch}
    The stacky Chern character is the ring homomorphism
    $$\tch: K^0(\mc X)_\BC \xrightarrow{\iota^*} K^0(I\mc X)_\BC \xrightarrow{\rho} K^0(I\mc X)_\BC \xrightarrow{\ch^{EG}} \CH^*(I \mc X)_\BC.$$
\end{defn}

Now recall the decomposition $K^0(I \mc X) = K^0(I \mc X)^{\fix} \oplus K^0(I\mc X)^{\mov}$, where $K^0(I \mc X)^{\fix}$, respectively $K^0(I \mc X)^{\mov}$, is the Grothendieck group of locally free sheaves on $I \mc X$  such that $g$ acts on the restriction to $\mc X^g$ with eigenvalue 1, respectively eigenvalues $\ne 1$. Also recall the operation $\lambda_{-1} := \sum_i (-1)^i \lambda_i$ defined on those classes of $K^0(I\mc X)$ for which the sum terminates.

\begin{prop} \label{prop: moving part of ttd}
    There is a group homomorphism $K^0(I \mc X)^{\mov} \to (\CH^*(I \mc X)_\BC)^\times$ uniquely defined by the property of extending $\ch^{EG} \circ \rho \circ \lambda_{-1}$ on classes $v \in K^0(I \mc X)^{\mov}$ for which $\lambda_{-1}(v)$ is well-defined. We again call this extended morphism $\ch^{EG} \circ \rho \circ \lambda_{-1}$.
\end{prop}

\begin{proof}
    We start by noting that $K^0(I \mc X)^{\mov}$ is generated by classes of equivariant vector bundles $V = V^{\mov}$, where $\lambda_{-1}([V])$ is well-defined and furthermore, on each component of $I\mc X$ we have that
    $$\rho \circ \lambda_{-1}([V]) = \sum_{i = 0}^{\mathrm{rk} V} (-1)^i \rho ([\wedge^i V])$$
    is of virtual rank
    $$\prod_{\zeta \ne 1} (1 - \zeta)^{\dim V_\zeta} \ne 0$$
    so that $\ch^{EG}(\rho \circ \lambda_{-1}(V))$ is invertible. It remains to note that given an identity $[V] = [V'] + [V'']$, we have $\lambda_{-1}([V]) = \lambda_{-1}([V']) \cdot \lambda_{-1}([V''])$ and thus also
    $$\ch^{EG} \circ \rho \circ \lambda_{-1}([V]) = \ch^{EG} \circ \rho \circ \lambda_{-1}([V']) \cap \ch^{EG} \circ \rho \circ \lambda_{-1}([V'']).$$
\end{proof}

\begin{defn} \label{defn: stacky td}
    We define an operation $\td^I: K^0(I\mc X) \to \CH^*(I\mc X)_\BC$ by
    $$\td^I(\mc E) := (\ch^{EG} \circ \rho \circ \lambda_{-1})(( \mc E^{\mov})^{\vee})^{-1} \cap \td^{EG}(\mc E^{\fix}) \in \CH^*(I \mc X)_\BC,$$
    where the first term is defined in Proposition \ref{prop: moving part of ttd}.
    The stacky Todd class of $\mc F \in K^0(\mc X)$ is 
    defined by
    $$\ttd(\mc F) := \td^I(\iota^*\mc F).$$
\end{defn}

\subsection{Identities} \label{subsec: stacky tau identities}
In this section, we prove identities for the stacky $\tau, \tch, \ttd$ generalizing those for $\tau^F, \ch, \td$ in the scheme case.
Again, throughout this section, we work only with nice quotients (Definition \ref{defn: nice quotient}).
% and all morphisms of stacks are of the form $f: [X/G] \to [Y/H]$ induced by a group homomorphism $G \to H$ and a $G$-equivariant map $X \to Y$ (where $G$ acts on $Y$ through $H$).

\begin{prop} \label{prop: ttd properties}
    Let $f: \mc X \to \mc Y$ be a morphism of nice quotients.\\
    (a) We have
    $$If^* \td^I(\mc E) = \td^I(If^* \mc E) \in \CH^*(I \mc X)_\BC$$
    for any $\mc E \in K^0(I\mc Y)$. In particular
    $$If^*\ttd(\mc F) = \ttd(f^* \mc F) \in \CH^*(I \mc X)$$
    for any $\mc F \in K^0(\mc Y)$. The respective statements also hold for $\ch^I := \ch^{EG} \circ \rho$ and $\tch$.\\
    (b) If $f$ is in addition proper, we have
    $$If_*(\td^I(If^* \mc E) \cap \alpha) = \td^I(\mc E) \cap If_* \alpha \in \CH_*(I \mc Y)_\BC$$
    for any $\mc E \in K^0(I\mc Y)$ and $\alpha \in \CH_*(I\mc X)$. In particular
    $$If_*(\ttd(f^* \mc F) \cap \alpha) = \ttd(\mc F) \cap If_* \alpha \in \CH_*(I \mc Y)_\BC$$
    for any $\mc F \in K^0(\mc Y)$ and $\alpha \in \CH_*(I \mc X)$.
\end{prop}

\begin{thm} \label{thm: tau properties}
    Let $\mc X$ and $\mc Y$ be nice quotients.\\
%    (a) If $\mc X$ is smooth, we have
%    $$\tau(\ms O_{\mc X}) = \ttd(\ms O_{\mc X}) \cap [I\mc X] \in \CH_*(I\mc X)_\BC.$$
    (a) For any $\mc F \in K_0(\mc X)$ and $\mc E \in K^0(\mc X)$ we have
    $$\tau(\mc E \cdot \mc F) = \tch(\mc E) \cap \tau(\mc F) \in \CH_*(I\mc X)_\BC.$$
    (b) For $\mc F \in K_0(\mc X)$ and $\mc G \in K_0(\mc Y)$ we have
    $$\tau(\mc F \boxtimes \mc G) = \tau(\mc F) \times \tau(\mc G) \in \CH_*(I(\mc X \times \mc Y))_\BC = \CH_*(I\mc X \times I\mc Y)_\BC.$$
    (c) If $f: \mc X \to \mc Y$ is a proper morphism, we have
    $$If_* \tau(\mc F) = \tau(f_* \mc F) \in \CH_*(I\mc Y)_\BC$$
    for any $\mc F \in K_0(\mc X)$.\\
    (d) If $f: \mc X \to \mc Y$ is a representable morphism and $\mc X$ and $\mc Y$ are smooth, then
    $$\tau(f^*\mc F) = \ttd(T_f) \cap If^! \tau(\mc F) \in \CH_*(I\mc X)_\BC$$
    for any $\mc F \in K_0(\mc Y)$. In particular, if $\mc X$ is smooth, then
    $$\tau(\mc E) = \tch(\mc E) \cap \ttd(T_{\mc X}) \cap [I\mc X] \in \CH_*(I\mc X)_\BC$$
    for any $\mc E \in K^0(\mc X) = K_0(\mc X)$.\\
    (e) Suppose $f: \mc X \to \mc Y$ is a representable regular embedding fitting into a Cartesian square of equidimensional nice quotients
    \[
    \begin{tikzcd}
        \mc X \arrow[r,hook, "f"] \arrow[d,hook, "c"] & \mc Y \arrow[d, hook]\\
        \mc Z \arrow[r,hook, "e"] & \mc W
    \end{tikzcd}
    \]
    such that all morphisms are representable and $\mc Z$ and $\mc W$ are smooth with $\dim \mc W - \dim \mc Z = \dim \mc Y - \dim \mc X$. Then
    $$\tau(f^* \mc F) = \ttd(T_f) \cap (Ie)^! \tau(\mc F) \in \CH_*(I\mc X)_\BC$$
    for any $\mc F \in K_0(\mc Y)$.
\end{thm}

\begin{rmk}
    The appearance of $(Ie)^!$ instead of $(If)^!$ in (e) may require some explanation. The issue is that $(If)^!$ is not always well-defined; for example, in Proposition \ref{prop: G and G-1 inverse} we will deal with a representable regular embedding $f: \mc X \to \mc Y$ such that $If$ is not regular (see Remark \ref{rmk: not reg}). The approach here with $\mc Z$ and $\mc W$ can be described as introducing a derived structure for which $If$ is quasismooth; then the corresponding Gysin pullback is exactly $(Ie)^!$. 
\end{rmk}

\begin{cor} \label{cor: agrees with Toen}
    The functor $\tau$ of Definition \ref{defn: stacky tau} agrees with the one constructed in \cite[Th\'eor\`eme 4.10]{Toen} on nice quotients.
\end{cor}

\begin{proof}
    First, note that $\tch$ and $\ttd(T_{\mc X}) \cap [I \mc X]$ agree with $Ch$ and $Td(\mc X)$, respectively, as defined in \cite[D\'efinitions 4.5, 4.7]{Toen}. It follows from Theorem \ref{thm: tau properties}(d) and \cite[Th\'eor\`eme 4.10(4)]{Toen} that $\tau$ and To\"en's functor, which we call $\tau_T$, match for any smooth nice quotient $\mc X$.

    To show that $\tau$ and $\tau_T$ agree on all nice quotients, we proceed by induction on $\dim \mc X$. Note that by canonical (and thus $G$-equivariant) resolution of singularities, we can find a proper birational map $f: \widetilde{\mc X} \to \mc X$ where $\widetilde{\mc X}$ is a smooth nice quotient. Let $i: \mc Z \hookrightarrow \mc X$ be a closed substack of lower dimension such that $f$ is an isomorphism over the complement $\mc U\xhookrightarrow{j} \mc Z$. Then, using the localization exact sequence
    $$K_0(\mc Z) \xrightarrow{i_*} K_0(\mc X) \xrightarrow{j^*} K_0(\mc U) \to 0$$
    \cite[Theorem 2.7]{Thomason-foundations} and the fact that $j^* \circ f_* = j_{\widetilde{\mc X}}^*$ (using the canonical, thus $G$-equivariant, isomorphism of \cite[Tag 08IB]{stacks-project}), we have a surjection
    $$i_* \oplus f_*: K_0(\mc Z) \oplus K_0(\widetilde{\mc X}) \twoheadrightarrow K_0(\mc X).$$
    Now, $\tau$ and $\tau_T$ agree on $K_0(\mc Z)$ by the induction hypothesis, and on $K_0(\widetilde{\mc X})$ by the previous paragraph. Then it suffices to note that $\tau$ and $\tau_T$ both commute with proper pushforward (by Theorem \ref{thm: tau properties}(c) and \cite[Th\'eor\`eme 4.10(1)]{Toen}, respectively).
\end{proof}

\begin{cor} \label{cor: part e}
    Let $B$ be a smooth variety, $\mc X$ a smooth nice quotient flat over $B$, and $\mc Y_i$ nice quotients flat over $B$ with representable embeddings into smooth nice quotients $\mc Y_i \hookrightarrow \mc Z_i$. Consider the morphism
    $$\delta_{\mc X}: \mc Y_1 \times_B \mc X \times_B \mc Y_3 \to \mc Y_1 \times_B \mc X \times \mc X \times_B \mc Y_2$$
    which is a flat base change of the diagonal $\Delta_{\mc X}: \mc X \to \mc X \times \mc X$. Then
    $$\tau(\delta_{\mc X}^* \mc F) = \ttd(T_{\delta_{\mc X}}) \cap (I\Delta_{\mc X})^! \tau(\mc F) = \ttd(-p_2^*T_{\mc X}) \cap (I\Delta_{\mc X})^! \tau(\mc F)$$
    for any $\mc F \in K_0(\mc Y_1 \times_B \mc X \times \mc X \times_B \mc Y_2)$.
\end{cor}

\begin{proof}
    Choosing a nice presentation $\mc X = [X/G]$, recall that $\Delta_{\mc X}$ can be represented by $(\sigma, p_2): G \times X \to X \times X$, where $\sigma$ is the action map and $p_2$ the projection. Then consider the factorization
    $$\Delta_{\mc X}: \mc X \xrightarrow{a} [X \times X/G] \xrightarrow{b} \mc X \times \mc X$$
    obtained from the $(G \times G)$-equivariant morphisms
    $$G \times X \xrightarrow{(1,\sigma,p_2)} G \times X \times X \xrightarrow{p_{23}} X \times X.$$
    We note that $a$ is a regular embedding, and $b$ is \'etale; letting $\delta_a$ and $\delta_b$ be the corresponding base changes, the same is true of $\delta_a$ (by flatness) and $\delta_b$. Furthermore, note that $\delta_a$ fits into a Cartesian square of representable morphisms
    \[
    \begin{tikzcd}
        \mc Y_1 \times_B \mc X \times_B \mc Y_2 \arrow[r, "\delta_a", hook] \arrow[d, hook] &  \mc Y_1 \times_{B,p_1} [X \times X/G] \times_{p_2,B} \mc Y_2\arrow[d, hook] \\
        \mc Z_1 \times \mc X \times \mc Z_2 \arrow[r, hook, "e"] & \mc Z_1 \times [X \times X/G] \times \mc Z_2
    \end{tikzcd}
    \]
    where the bottom terms are smooth, and both bottom and top arrows are flat base changes of $a$, thus of the same codimension.
    
    Then we have
    $$\tau(\delta_{\mc X}^* \mc F) = \tau(\delta_a^* \delta_b^* \mc F) = \ttd(T_{\delta_a}) \cap (Ia)^!\tau(\delta_b^* \mc F)$$
    by Theorem \ref{thm: tau properties}(e) applied to the square above (for each component of $\mc X$ separately, if there are components of varying dimensions), since $Ie$ is a flat base change of $Ia$
    $$ = \ttd(T_{\delta_a}) \cap (Ia)^!(I\delta_b)^*\tau( \mc F)$$
    since $\pi \circ \delta_b^* = I\delta_b^* \circ \pi$ because $\delta_b$ is representable; $\iota_*^{-1}$ commutes with $I\delta_b^*$ by Lemma \ref{lem: localization inverse}, using $\mc Z_1 \times [X \times X/G] \times \mc Z_2$ and $\mc Z_1 \times \mc X \times \mc X \times \mc Z_2$ as smooth embeddings; $\rho$ commutes because $\delta_b$ is representable; and $\tau^{EG}$ commutes by \cite[Theorem 3.1(d)(ii)]{EG-RRequivChow} (or alternatively by \cite[Th\'eor\`eme 4.10(2)]{Toen}, since by Corollary \ref{cor: agrees with Toen} the two $\tau$ functors agree)
    $$= \ttd(T_{\delta_a}) \cap (Ia)^!(Ib)^!\tau( \mc F) = \ttd(T_{\delta_a}) \cap (I\Delta_{\mc X})^!\tau( \mc F) = \ttd(T_{\delta_{\mc X}}) \cap (I\Delta_{\mc X})^!\tau( \mc F)$$
    since $(I\delta_b)^* = (Ib)^!$ is flat pullback, and $T_{\delta_a} = T_{\delta_{\mc X}}$ since $\delta_b$ is \'etale.
\end{proof}

For the proof of Theorem \ref{thm: tau properties}, it will be convenient to have concrete presentations of the various morphisms of nice quotients. The following two lemmas are well-known, but we include proofs for lack of a precise reference.
\begin{lem} \label{lemma: nonrepresentable presentation}
    If $f: \mc X \to \mc Y$ is a morphism of nice quotients, then there exist nice quotient representations $\mc X = [X/G]$ and $\mc Y = [Y/H]$, a group homomorphism $G \to H$, and a $G$-equivariant morphism $X \to Y$ (where $G$ acts on $Y$ through $H$) inducing the morphism $f$.
\end{lem}

\begin{proof}
    Pick nice quotient representations $\mc X = [Z/J]$ and $\mc Y = [Y/H]$. Take $X := Z \times_{\mc Y} Y$ and $G := J \times H$. Note that $J$ acts on $X$ through the action on $Z$ over $\mc X$, and $H$ acts on $X$ through the action on $Y$; these actions commute and thus define an action of $G$ on $X$, compatible with the action of $J$ on $Z$, respectively of $H$ on $Y$, via the projection maps. We claim that the projection $X \to Z$ compatible with the projection of groups $G \to J$ induces an isomorphism of quotient stacks $a: [X/G] \cong [Z/J]$: indeed, it suffices to note that $[X/G] \times_{[Z/J]} Z = [X/H] = Z$, i.e., $a$ pulls back to an isomorphism under the cover $Z \to [Z/J]$ \cite[Tags 04XD, 041Y]{stacks-project}.
\end{proof}

\begin{lem} \label{lem: representable representation}
    If $f: \mc X \to \mc Y$ is representable morphism of nice quotients, there is a finite abelian group $G$, quasiprojective varieties $X$ and $Y$ with $\mc X = [X/G]$ and $\mc Y = [Y/G]$, and a $G$-equivariant morphism $X \to Y$ inducing the map $f$. Conversely, given a morphism of quasiprojective varieties $X \to Y$ equivariant with respect to the action of a finite group $G$, the corresponding morphism of nice varieties $[X/G] \to [Y/G]$ is representable.
\end{lem}

\begin{proof}
    Choose a nice quotient presentation $\mc Y = [Y/G]$, and take $X := \mc X \times_{\mc Y} Y$. Then we have a Cartesian diagram
    \[
    \begin{tikzcd}
        X \arrow[d] \arrow[r] & Y \arrow[d] \arrow[r] & \Spec \BC \arrow[d] \\
        \mc X \arrow[r] & \left[Y/G\right] \arrow[r] & BG
    \end{tikzcd}
    \]
    where $X \to Y$ is the needed $G$-equivariant map and $\mc X = [X/G]$. (For quasiprojectivity of $X$, note that $\mc X$ has a nice quotient presentation $[X'/G']$, so that $X' \times_{\mc X} X$ is a finite cover of the quasiprojective variety $X'$, therefore quasiprojective \cite[Tag 0C4M]{stacks-project}, and since it is a finite cover of $X$, then $X$ too is quasiprojective \cite[Tag 0C4N]{stacks-project}.)

    In the other direction, such a $G$-equivariant morphism induces a Cartesian diagram
        \[
    \begin{tikzcd}
        X \arrow[d] \arrow[r] & Y \arrow[d] \arrow[r] & \Spec \BC \arrow[d] \\
        \left[X/G\right] \arrow[r, "f"] & \left[Y/G\right] \arrow[r] & BG
    \end{tikzcd}
    \]
    Since $Y \to [Y/G]$ is an \'etale cover and $X$ a scheme, we conclude that the induced map $f$ is representable \cite[Tag 04ZP]{stacks-project}.
\end{proof}

With these in hand, we turn to the proofs:

\begin{proof}[Proof of Proposition \ref{prop: ttd properties}]
    (a) Let $f: \mc X \to \mc Y$ be a morphism of nice quotients and $\mc E \in K^0(I \mc Y)$. We note that the canonical isomorphism of $\mc E$ induced by the canonical isomorphism on $I\mc Y$ pulls back to the canonical isomorphism on $If^*\mc E$. Thus $If^* ( \mc E^{\fix}) = (If^* \mc E)^{\fix}$, and similarly for the dual of the moving part, and $If^*$ commutes with $\rho$ and $\lambda_{-1}$. Then it suffices to note that $If^*$ commutes with the equivariant classes $\td^{EG}$ and $\ch^{EG}$ \cite[\S 3.6]{Kresch}. Finally, for $\mc F \in K^0(\mc Y)$, it suffices to note that $If^*\iota_{\mc Y}^* \mc F = \iota_{\mc X}^* f^* \mc F$.\\
    
    (b) We have
    $$If_*(\td^I(If^* \mc E) \cap \alpha) = If_*(If^*\td^I(\mc E) \cap \alpha) = \td^I(\mc E) \cap If_* \alpha,$$
    where the first equality follows from part (a) and the second from the projection formula for $If$, which is again a proper morphism of nice quotients.
\end{proof}

\begin{proof}[Proof of Theorem \ref{thm: tau properties}(a)]
    Let $\mc X = [X/G]$ be a nice presentation. We trace through the steps in the definition of $\tau$. First, we have
    $$\pi(\mc E \cdot \mc F) = \mc E \cdot \pi(\mc F)$$
    since the actions of $K^0(BG) \to K^0(\mc X)$ on $K_0(\mc X)$ commute, and thus the product with $\mc E$ is an isomorphism of $K^0(BG)$-modules. Then by the projection formula
    $$\iota_*^{-1} \circ \pi(\mc E \cdot \mc F) = \iota_*^{-1}(\mc E \cdot \pi(\mc F)) = \iota^* \mc E \cdot \iota_*^{-1}(\pi(\mc F)).$$
    Next, we note that
    $$\rho(\iota^* \mc E \cdot \iota_*^{-1}(\pi(\mc F))) = \rho(\iota^* \mc E) \cdot \rho(\iota_*^{-1}(\pi(\mc F))) $$
    and so
    $$\tau(\mc E \cdot \mc F) = \tau^{EG} \circ \rho \circ \iota_*^{-1} \circ \pi(\mc E \cdot \mc F) = \tau^{EG}(\rho(\iota^* \mc E) \cdot \rho(\iota_*^{-1}(\pi(\mc F))))$$
    $$=\ch^{EG}(\rho(\iota^* \mc E)) \cap \tau^{EG}(\rho(\iota_*^{-1}(\pi(\mc F))))) = \tch(\mc E) \cap \tau(\mc F)$$
    by \cite[Theorem 3.1(c)]{EG-RRequivChow}.
\end{proof}

\begin{proof}[Proof of Theorem \ref{thm: tau properties}(b)]
    Let $\mc X = [X/G]$ and $\mc Y = [Y/H]$ be nice presentations, so that we have a nice presentation $\mc X \times \mc Y \cong [X \times Y/G \times H]$. We start by claiming that
    $$\pi_{(g,h)}(\mc F \boxtimes \mc G) = \pi_g(\mc F) \boxtimes \pi_h(\mc G)$$
    Indeed, since $\mf m_g \subset K^0(BG)_\BC$ annihilates $\pi_g(\mc F)$ and similarly for $\mc G$, we certainly have
    $$\pi_g(\mc F) \boxtimes \pi_h(\mc G) \subset K_0(\mc X \times \mc Y)_{\mf m_{(g,h)}, \BC} \subset K_0(\mc X \times \mc Y)_\BC$$
    and on the other hand
    $$\sum_{g \in G, h \in H} \pi_g(\mc F) \boxtimes \pi_h(\mc G) = \mc F \boxtimes \mc G.$$
    Next,
    $$\iota_{(g,h) *} (\iota_{g*}^{-1}(\pi_g(\mc F) \boxtimes \iota_{h*}^{-1}(\pi_h(\mc G)))) = (\iota_g \times \iota_h)_*(\iota_{g*}^{-1}(\pi_g(\mc F) \boxtimes \iota_{h*}^{-1}(\pi_h(\mc G)))) = \pi_g(\mc F) \boxtimes \pi_h(\mc G)$$
    by Lemma \ref{lemma: pushforward exterior product}, i.e., writing $\iota_*^{-1} \circ \pi(\mc F) =: (\mc F_g)_{g \in G}$ and similarly for $\mc G$, we have
    \begin{equation} \label{eq: boxtimes loc}
        \iota_*^{-1} \circ \pi (\mc F \boxtimes \mc G) = (\mc F_g \boxtimes \mc G_h)_{(g,h) \in G \times H}.
    \end{equation}

    Now, we have
    $$\rho_{(g,h)}(\mc F_g \boxtimes \mc G_h) = \rho_g(\mc F_g) \boxtimes \rho_h(\mc G_h)$$
    and so
    $$\tau^{EG}(\rho_g(\mc F_g) \boxtimes \rho_h(\mc G_h)) = \tau^{EG}(\rho_g(\mc F_g)) \times \tau^{EG}(\rho_h(\mc G_h)),$$
    which follows from the corresponding fact for $\tau^F$ \cite[Example 18.3.1]{Fulton} when one uses representations $V_G, V_H, V_G \oplus V_H$ to compute the functors $\tau^{EG}$ for the groups $G,H, G \times H$ (respectively).
    By definition, this means
    $$\tau(\mc F \boxtimes \mc G) = \tau(\mc F) \times \tau(\mc G).$$
\end{proof}

\begin{proof}[Proof of Theorem \ref{thm: tau properties}(c)]
    First, by Lemma \ref{lemma: nonrepresentable presentation}, there are nice presentations $\mc X = [X/G]$ and $\mc Y = [Y/H]$ such that $f$ is induced by a morphism $\tilde f:X \to Y$ equivariant with respect to a homomorphism $\phi: G \to H$. Note that the morphism $If: I\mc X \to I \mc Y$ breaks into morphisms
    $If_g: \mc X^g \to \mc Y^{\phi(g)}$ induced by $\tilde f|_{X^g}: X^g \to Y^{\phi(g)}$. 

    Writing $\iota_*^{-1} \circ \pi(\mc F) = (\mc F_g)_{g \in G}$, we claim that
    \begin{equation} \label{eq: pushforward loc}
        \iota_*^{-1} \circ \pi(f_* \mc F)_h = \sum_{g \in \phi^{-1}(h)} If_* \mc F_g \in K_0(\mc Y^h)_{\mf m_h,\BC}.
    \end{equation}
    Since
    $$\sum_{h \in H}\iota_{\mc Y, h*} \sum_{g \in \phi^{-1}(h)} If_{g*} \mc F_g = \iota_{\mc Y*} If_{*} \left(\sum_{g \in G} \mc F_g\right) = f_* \iota_{\mc X*}\left(\sum_{g \in G} \mc F_g\right) = f_* \mc F,$$
    it suffices to verify that $If_{g*} \mc F_g \in K_0(\mc Y^h)_{\mf m_h, \BC}$ for $h = \phi(g)$, or in other words that for any $r \in \mf m_h \subset K^0(BH)$ we have $r \cdot If_{g*} \mc F_g = 0$. This follows since, with notation as in the following commutative (but not necessarily Cartesian) square
    \[
    \begin{tikzcd}
    \mc X^g \arrow[r, "If_g"] \arrow[d, "p"] & \mc Y^h \arrow[d, "q"] \\
    BG \arrow[r, "\phi"] & BH
    \end{tikzcd}
    \]
    the projection formula \cite[Proposition 3.7]{Khan} implies that
    $$q^*r \cdot If_{g*} \mc F_g = If_{g*}(If^*_g q^*r \cdot \mc F_g) = If_{g*}(p^* \phi^*r \cdot \mc F_g) = 0,$$
    since $\phi^* r \in \mf m_g$ and thus by assumption that $\mc F_g \in K_0(\mc X^g)_{\mf m_g, \BC}$ we have $p^* \phi^* r \cdot \mc F_g = 0$.

    Next, we note that $If_* \circ \rho_X = \rho_Y \circ If_*$ since the canonical automorphism on $If_* \mc G$ for any $\mc G \in K_0(I \mc X)$ is given by the composition of $If_*$ with the canonical automorphism on $\mc G$, and so in particular the eigensheaf decomposition is preserved by $If_*$.
    
    Finally, to show that $If_* \circ \tau^{EG} = \tau^{EG} \circ If_*$ we use compatibility of $\tau^{EG}$ with representable proper pushforward \cite[Theorem 3.1(b)]{EG-RRequivChow}: Since $p: \coprod_{g \in G} X^g \to I\mc X$ is proper and representable, we have a diagram
    \[
    \begin{tikzcd}
        K_0(\coprod_{g \in G} X^g)_{\mf m_1, \BQ} \arrow{r}{\sim}[swap]{\tau^{EG}} \arrow[d, "p_*"] & \CH_*(\coprod_{g \in G} X^g)_\BQ \arrow[d, two heads, "p_*"] \\
        K_0(I \mc X)_{\mf m_1, \BQ} \arrow{r}{\sim}[swap]{\tau^{EG}} & \CH_*(I \mc X)_\BQ
    \end{tikzcd}
    \]
    (where the horizontal arrows are isomorphisms by \cite[Corollary 5.1]{EG-RRequivChow}, and the right vertical arrow is a surjection by \cite[Example 1.7.6]{Fulton}). Thus, for any $\mc G \in K_0(I\mc X)_{\mf m_1, \BQ}$ we can find $\mc H \in K_0(\coprod_{g \in G} X^g)_\BQ$ with $p_* \mc H = \mc G$. To show that $\tau^{EG}(If_* \mc G) = If_* \tau^{EG}(\mc G)$ it suffices to see that $\tau^{EG}((If \circ p)_* \mc H) = (If \circ p)_* \tau^{EG}(\mc H)$, but now note that $If \circ p: \coprod_{g \in G} X^g \to I \mc Y$ is again representable and proper, as we can factor $\coprod_{g \in G} X^g \to \coprod_{h \in H} Y^h \to I\mc Y$.
    
    Thus we conclude that
    $$\tau(f_* \mc F) = \tau^{EG} \circ \rho \circ \iota_*^{-1} \circ \pi(f_* \mc F) = \tau^{EG} \circ \rho\left(\sum_{g \in \phi^{-1}(h)} If_{g*} \mc F_g\right)_{h \in H} $$
    $$= \tau^{EG} \left( \sum_{g \in \phi^{-1}(h)} If_{g*} \rho_g(\mc F_g) \right)_{h \in H} = \left( \sum_{g \in \phi^{-1}(h)}If_{g*} \tau^{EG}(\rho_g(\mc F_g)) \right)_{h \in H} = If_* \tau(\mc F).$$
\end{proof}

\begin{proof}[Proof of Theorem \ref{thm: tau properties}(d)]
    First, since $f$ is representable, by Lemma \ref{lem: representable representation} we may choose nice presentations $\mc X = [X/G]$ and $\mc Y = [Y/G]$ such that $f$ is induced by a $G$-equivariant morphism $\tilde f: X \to Y$. Since $\mc X$ and $\mc Y$ are smooth, the same is true of $X$ and $Y$, as well as the fixed loci $X^g, Y^g$ \cite[Proposition 3.1]{Thomason}.

    By Lemma \ref{lem: smooth localization} we have
    $$\tau(f^*\mc F) = \td^I(\mc N_{\iota_{\mc X}}) \cap \tau^{EG}(\rho(\iota_{\mc X}^* f^* \mc F)),$$
    where we note that
    $$\tau^{EG}(\rho(\iota_{\mc X}^* f^* \mc F)) = \tau^{EG}(\rho(If^* \iota_{\mc Y}^* \mc F)) = \tau^{EG}(If^*(\rho(\iota_{\mc Y}^* \mc F)))$$
    $$= \td^{EG}(T_{If}) \cap If^!\tau^{EG}(\rho(\iota_{\mc Y}^* \mc F)))$$
    by the proof of \cite[Theorem 3.1(d)(ii)]{EG-RRequivChow}, 
    since $If: I \mc X \to I\mc Y$ is induced by the $G$-equivariant morphism of smooth quasiprojective varieties $\coprod_{g \in G} \tilde f|_{X^g}
    : \coprod X^g \to \coprod Y^g$ (a morphism of smooth varieties being l.c.i.), and we replace the hypothesis on $G$ with the quasiprojectivity of $X^g$ and $Y^g$, so that one may now apply \cite[Proposition 23(1)]{EG-equivint}).

    We conclude that
    $$\tau(f^* \mc F) = \td^I(\mc N_{\iota_{\mc X}}) \cap \td^{EG}(T_{If}) \cap If^!\tau^{EG}(\rho(\iota_{\mc Y}^* \mc F))$$
    $$= \td^I(\mc N_{\iota_{\mc X}}) \cap \td^{EG}(T_{If}) \cap If^!(\td^I(\mc N_{\iota_{\mc Y}})^{-1} \cap \td^I(\mc N_{\iota_{\mc Y}}) \cap
    \tau^{EG}(\rho(\iota_{\mc Y}^* \mc F))) $$
    $$= \td^I(\mc N_{\iota_{\mc X}}) \cap \td^{EG}(T_{If}) \cap \td^I(If^* \mc N_{\iota_{\mc Y}})^{-1} \cap If^!(\td^I(\mc N_{\iota_{\mc Y}}) \cap
    \tau^{EG}(\rho(\iota_{\mc Y}^* \mc F)) )$$
    $$= \td^I(\mc N_{\iota_{\mc X}} - If^* \mc N_{\iota_Y}) \cap \td^{EG}(T_{If}) \cap If^!(\tau(\mc F)) $$
    again by Lemma \ref{lem: smooth localization}. Now, note that $\iota_{\mc X}^* T_f$ satisfies
    $$(\iota_{\mc X}^* T_f)^{\mov} = (\iota_{\mc X}^* T_{\mc X} - \iota_{\mc X}^* f^* T_{\mc Y})^{\mov} = \mc N_{\iota_{\mc X}} - If^* \mc N_{\iota_{\mc Y}}$$
    $$ (\iota_{\mc X}^* T_f)^{\fix} = (\iota_{\mc X}^* T_{\mc X} - \iota_{\mc X}^* f^* T_{\mc Y})^{\fix} = T_{I\mc X} - If^* T_{I \mc Y} = T_{If},$$
    so that we may rewrite the above as
    $$\tau(f^* \mc F) = \ttd(T_f) \cap If^!(\tau(\mc F)).$$

    Finally, for the last part, if $\mc X = [X/G]$ is a nice quotient presentation, consider the morphism $\pi: \mc X \to BG$. Note that $T_\pi = T_{\mc X}$ since $T_{BG}$ is of rank 0. Note also that
    $$\tau(\ms O_{BG}) = ([BG])_{g \in G} \in \CH_*(IBG) = \oplus_{g \in G} \CH_*(BG)$$
    since $\CH_*(BG) = \CH_0(BG)$, and on an approximation $U/G \to BG$ we have $\tau(\ms O_{U/G}) = [U/G] + \ldots$ \cite[Theorem 18.3(5)]{Fulton}. Then for any $\mc E \in K^0(\mc X)$ we have
    $$\tau(\mc E) \overset{(a)}{=} \tch(\mc E) \cap \tau(\ms O_{\mc X}) = \tch(\mc E) \cap \tau(\pi^* \ms O_{BG}) = \tch(\mc E) \cap \ttd(T_{\mc X}) \cap [I \mc X].$$
\end{proof}

\begin{lem} \label{lem: smooth localization}
    If $\mc X$ is a smooth nice quotient, then
    $$\tau(\mc F) = \td^I(\mc N_\iota) \cap \tau^{EG}(\rho(\iota^* \mc F)) \in \CH_*(I\mc X)_\BC$$
    for any $\mc F \in K_0(\mc X)_\BC$, where $\mc N_\iota$ is the normal bundle of the local embedding $\iota: I\mc X \to \mc X$.
\end{lem}

\begin{proof}
    Recall from Lemma \ref{lem: localization inverse} that for each $g \in G$ we have
    $$(\lambda_{-1} \mc N_{\iota_g}^\vee)^{-1} \cdot \iota_g^* = \iota_{g*}^{-1}: K_0(\mc X)_{\mf m_g, \BC} \to K_0(\mc X^g)_{\mf m_g, \BC}$$
    where the inverse $(\lambda_{-1} \mc N_{\iota_g}^\vee)^{-1}$ is taken in the quotient ring $K^0(\mc X^g)_{\mf m_g, \BC}$.
    Note that although this is not an inverse in $K^0(\mc X^g)_\BC$, it nevertheless satisfies the property that
    \begin{equation} \label{eq: lambda-1 inverse}
    \ch^{EG}(\rho((\lambda_{-1} \mc N_{\iota_g}^\vee)^{-1})) \cap \ch^{EG}(\rho(\lambda_{-1} \mc N_{\iota_g}^\vee)) = \ch^{EG}(\rho(\pi_g(\ms O_{\mc X^g}))) =  1 \in \CH^*(\mc X^g)_\BC,
    \end{equation}
    where $\pi_g$ is the projection onto the $\mf m_g$ component as in (\ref{eq: pi}), and we note that
    $$\ch^{EG}(\rho(\pi_g(\ms O_{\mc X^g}))) = \ch^{EG}(\rho(\ms O_{\mc X^g})) = \ch^{EG}(\ms O_{\mc X^g}) = 1$$
    since $\rho$ permutes components as described in equation (\ref{eq: rho components}) and $\ch^{EG}: K^0(\mc X^g)_\BC \to \CH^*(\mc X^g)_{\BC}$ factors through $K^0(\mc X^g)_{\mf m_1, \BC}$ (Proposition \ref{prop: EG factors through m1}).

    Then we have
    $$\tau(\mc F) = \tau^{EG} \circ \rho(\lambda_{-1}(\mc N^\vee_\iota)^{-1} \cdot \iota_g^* \pi(\mc F))$$
    $$ = \tau^{EG}(\rho(\lambda_{-1}(\mc N_\iota^\vee)^{-1}) \cdot \rho(\iota^* \pi( \mc F))) = \ch^{EG}(\rho(\lambda_{-1}(\mc N^\vee_\iota)^{-1})) \cap \tau^{EG}(\rho(\iota^* \pi(\mc F)))$$
    by \cite[Theorem 3.1(c)]{EG-RRequivChow}
    $$= \td^I(\mc N_\iota) \cap \tau^{EG}(\rho(\iota^* \pi(\mc F)))$$
    by definition and (\ref{eq: lambda-1 inverse})
    $$ = \td^I(\mc N_\iota) \cap \tau(\iota^* \mc F)$$
    since $\tau^{EG} \circ \rho$ factors through the projections $\pi_g: K_0(I\mc X^g)_\BC \to K_0(I\mc X^g)_{\mf m_g, \BC}$ by (\ref{eq: rho components}) and Proposition \ref{prop: EG factors through m1}, and $\pi_g(\iota_g^* \mc F) = \iota_g^* \pi_g(\mc F)$.
\end{proof}

\begin{proof}[Proof of Theorem \ref{thm: tau properties}(e)] 
    \textbf{Case 1:} Suppose all morphisms in the square are closed embeddings.
    By Lemma \ref{lem: pullback first step}, for any $\mc F \in K_0(\mc Y)_\BC$ we have
    $$\iota_{\mc X *}^{-1} \circ \pi(f^* \mc F) =\lambda_{-1}((Ic)^*(\iota_{\mc Z}^*\mc N^\vee_e)^{\mov}) \cdot (Ie)^!(\iota_{\mc Y*}^{-1} \circ \pi(\mc F)).$$
    Note that $\rho$ commutes with $(Ie)^!$ since $(Ie)^!(\mc F) = \sum (-1)^i\Tor_i^{I\mc W}(\ms O_{I \mc Z}, \mc F)$ and each $\ms O_{\mc Z^g}$ is $g$-fixed, so we have
    $$\rho \circ \iota_{\mc X *}^{-1} \circ \pi(f^* \mc F) = \rho(\lambda_{-1}((Ic)^*(\iota_{\mc Z}^*\mc N^\vee_e)^{\mov})) \cdot (Ie)^!(\rho \circ \iota_{\mc Y*}^{-1} \circ \pi(\mc F)).$$
    Thus
    $$\tau(f^* \mc F) = \tau^{EG}(\rho(\lambda_{-1}((Ic)^*(\iota_{\mc Z}^*\mc N^\vee_e)^{\mov})) \cdot (Ie)^!(\rho \circ \iota_{\mc Y*}^{-1} \circ \pi(\mc F)))$$
    $$ = \ch^{EG}(\rho(\lambda_{-1}((Ic)^*(\iota_{\mc Z}^*\mc N^\vee_e)^{\mov}))) \cap \tau^{EG}((Ie)^!(\rho \circ \iota_{\mc Y*}^{-1} \circ \pi(\mc F)))$$
    by \cite[Theorem 3.1(c)]{EG-RRequivChow}
    \begin{equation} \label{eq: tau pullback case 1}
     = \ch^{EG}(\rho(\lambda_{-1}((Ic)^*(\iota_{\mc Z}^*\mc N^\vee_e)^{\mov}))) \cap \td^{EG}((Ic)^* T_{Ie}) \cap (Ie)^! \tau^{EG}(\rho \circ \iota_{\mc Y*}^{-1} \circ \pi(\mc F))   
    \end{equation} 
    by Lemma \ref{lem: tauEG and Gysin}.
    Note that we have a surjection of vector bundles $c^* \mc N_e^\vee \twoheadrightarrow \mc N_f^\vee$ \cite[Tag 0473]{stacks-project}, where by assumption these are of the same rank; thus this map is an isomorphism, and
    $$\iota_{\mc X}^* T_f = - \iota_{\mc X}^* \mc N_f = -\iota_{\mc X}^*(c^* \mc N_e) = - (Ic)^*(\iota_{\mc Z}^* \mc N_e)$$
    has fixed part
    $$(\iota_{\mc X}^* T_f)^{\fix} = -(Ic)^*(\iota_{\mc Z}^* \mc N_e)^{\fix} = (Ic)^* T_{Ie}$$
    and moving part
    $$(\iota_{\mc X}^* T_f)^{\mov} = -(Ic)^*(\iota_{\mc Z}^* \mc N_e)^{\mov}.$$
    Furthermore, since $\ch^{EG} \circ \rho \circ \lambda_{-1}$ is a group homomorphism, as described in Proposition \ref{prop: moving part of ttd},
    we have
    $$\ch^{EG} \circ \rho \circ \lambda_{-1}(-(\iota_{\mc X}^* T^\vee_f)^{\mov}) = \ch^{EG} \circ \rho \circ \lambda_{-1}((\iota_{\mc X}^* T^\vee_f)^{\mov})^{-1}.$$
    Thus (\ref{eq: tau pullback case 1}) gives
    $$\tau(f^* \mc F) = \ttd(T_f) \cap (Ie)^! \tau(\mc F).$$
    \\
    \textbf{Case 2:} Suppose $e: \mc Z \hookrightarrow \mc W$ (and thus also $f: \mc X \to \mc Y$) is an open embedding. As in Lemma \ref{lem: representable representation}, by choosing a nice quotient presentation of $\mc W$ we can find a finite abelian group $G$ acting on quasiprojective varieties $X,Y,Z,W$ such that the morphisms of stacks are induced by $G$-equivariant morphisms of varieties.
    
    First, we claim that
    \begin{equation} \label{eq: tau pullback step 2}
     \iota_{\mc X *}^{-1}(\pi(f^*\mc F)) = \iota_{\mc X *}^{-1}(f^*(\pi \mc F)) = (If)^* \iota_{\mc Y*}^{-1}(\pi \mc F).   
    \end{equation}
    Indeed, the first equality follows from the fact that $f^*$ is compatible with the $K^0(BG)_\BC$-module structure and the second from the fact that
    $$f^*\iota_{\mc Y *} = \iota_{\mc X *}(If)^*,$$
    by the corresponding (canonical, thus $G$-equivariant) isomorphism of complexes of sheaves \cite[Tag 08IB]{stacks-project}, since $f$ is flat and $I \mc X = \mc X \times_{\mc Y} I\mc Y$ by \cite[Tag 06R5]{stacks-project}, as $\mc X \to \mc Y$ is an open embedding, therefore a monomorphism.

    We conclude that
    $$\tau(f^* \mc F) = \tau^{EG} \circ \rho(\iota_{\mc X *}^{-1}(\pi(f^*\mc F))) = \tau^{EG} \circ \rho((If)^* \iota_{\mc Y*}^{-1}(\pi \mc F))$$
    by (\ref{eq: tau pullback step 2})
    $$=\tau^{EG}((If)^* \rho(\iota_{\mc Y*}^{-1}(\pi \mc F))) $$
    since $(If)^*$ respects the $G$-equivariant structure and thus commutes with $\rho$
    $$ = (If)^*\tau^{EG}(\rho(\iota_{\mc Y*}^{-1}(\pi \mc F))) = (If)^* \tau(\mc F).$$
    by \cite[Theorem 3.1(d)(i)]{EG-RRequivChow}, since as described above $If$ is the base change of $f$, and thus an open embedding. Note that in this case $T_f$ is trivial. Furthermore, for the same reasons $Ie$ is an open embedding, and $(Ie)^! = (If)^*: \CH_*(I \mc Y) \to \CH_*(I \mc X)$ is simply the restriction.
    \\
    \textbf{The general case:} Suppose we are given a Cartesian square as described in the statement of the theorem. Again, choose a presentation by a finite abelian group $G$ and quasiprojective varieties $X,Y,Z,W$. By replacing $W$ and $Z$ by the $G$-stable open subsets $W' := W\setminus (\overline{Y} \setminus Y)$ and $Z' := Z \cap W'$ if necessary, we may assume that the vertical morphisms are closed embeddings. Now factoring $e: Z \hookrightarrow W$ as $Z \hookrightarrow W \setminus (\overline{Z} \setminus Z) \hookrightarrow W$ and considering the corresponding factorization of $f: X \hookrightarrow Y$, we reduce to the previous two cases.
\end{proof}

\begin{lem} \label{lem: tauEG and Gysin}
    Suppose there is a $G$-equivariant Cartesian square of varieties
    \[
    \begin{tikzcd}
        A \arrow[d, "a"] \arrow[r] & B \arrow[d, "b"] \\
        C \arrow[r, "c"]  & D 
    \end{tikzcd}
    \]
    such that all the morphisms are closed embeddings and $C$ and $D$ are smooth. Then
    $$\tau^{EG}(c^! \mc F) = \td^{EG}(-a^*\mc N_{\mc C/\mc D}) \cap c^! \tau^{EG}(\mc F)$$
    for all $\mc F \in K_0(\mc B)$.
\end{lem}

\begin{proof}
    We start by proving the corresponding statement for the Fulton tau functor $\tau^F$. Recalling the definition of $\tau^F$ in terms of the localized Chern character and a smooth embedding \cite[p.~349]{Fulton}, we have
    $$\tau^F(c^! \mc F) = \td(a^*T_C) \cap \ch^C_A(c^! \mc F) \cap [C]$$
    where
    $$\ch^C_A(c^! \mc F) \cap [C] = \ch^C_A(c^! \mc F) \cap c^*[D] = c^!(\ch^D_B(\mc F) \cap [D])$$
    by \cite[p.~343, (19)]{Fulton}. Thus
    $$\tau^F(c^! \mc F) = \td(a^*T_C) \cap c^!(\ch^D_B(\mc F) \cap [D])$$
    $$= \td(a^*T_C -a^*c^* T_D) \cap \tau^F(\mc F) = \td(-a^*\mc N_{C/D}) \cap \tau^F(\mc F).$$
    
    Now, let $p_A: A_G \to \mc A$, etc. be schematic approximations given by a fixed choice of $U$. Then
    $$p_A^* \circ \tau^{EF} \circ c^! := \tau^F \circ p_A^* \circ c^! = \tau^F \circ c^! \circ p_A^*$$
    by \cite[Lemma 3.2]{Anderson-Payne}, where here we use $c^!: K_0(B_G) \to K_0(A_G)$ to mean the Gysin pullback induced by $\mc C \to \mc D$, which is the same as that induced by $C_G \to D_G$ by Lemma \ref{lem: excess intersection}(a). Then it suffices to note that from the previous paragraph we have
    $$\tau^F \circ c^! \circ p_A^* = \td(-a_G^* \mc N_{C_G/D_G}) \cap c^! \circ \tau^F \circ p_A^* =: p_A^*(\td^{EG}(-a^* \mc N_{\mc C/\mc D}) \cap c^! \tau^{EG}(\mc F))$$
    since $p_C^* \mc N_{\mc C/\mc D} = \mc N_{C_G/D_G}$.
\end{proof}

\begin{lem} \label{lem: pullback first step}
    In the situation of Theorem \ref{thm: tau properties}(e), if all morphisms in the Cartesian square are closed embeddings, then we have
    $$\iota_{\mc X *}^{-1} \circ \pi(f^* \mc F) = f^!(\iota_{\mc Y *}^{-1} \circ \pi(\mc F)) = \lambda_{-1}((Ic)^*(\iota_{\mc Z}^*\mc N^\vee_{Z/W})^{\mov}) \cdot (Ie)^!(\iota_{\mc Y*}^{-1} \circ \pi(\mc F))$$
    for all $\mc F \in K_0(\mc Y)_\BC$.
\end{lem}

\begin{proof}
    As in Lemma \ref{lem: representable representation}, choosing a nice presentation $\mc W = [W/G]$, we may suppose the given Cartesian square is induced by a $G$-equivariant Cartesian square
    \[
    \begin{tikzcd}
        X \arrow[r,hook, "f"] \arrow[d, hook, "c"] & Y \arrow[d, hook, "d"] \\
        Z \arrow[r, hook, "e"] & W
    \end{tikzcd}
    \]
    with $Z$ and $W$ smooth, all morphisms closed embeddings, and $f$ a regular embedding.
    Now, from Lemma \ref{lem: localization inverse} applied to $X \hookrightarrow W$, we have
    \begin{equation} \label{eq: pullback 1}
        \iota_{\mc X*}^{-1}\circ \pi(f^* \mc F) = Ic^*e_g^*(\lambda_{-1}(\mc N_{\iota_{\mc W}}^\vee)^{-1}) \cdot \iota_W^!\circ \pi(f^* \mc F).
    \end{equation}
    where
    $$\iota_W^! \circ \pi(f^* \mc F) = \iota_W^!(f^* \circ \pi(\mc F)) = \iota_W^!(f^! \circ \pi( \mc F)) = f^!(\iota_W^! \circ \pi( \mc F))$$
    with the first equality since $f^*$ is a morphism of $K^0(BG)_\BC$-modules, the second by definition of Gysin pullback, and the third by \cite[Lemma 3.2]{Anderson-Payne}. Then from (\ref{eq: pullback 1}) we obtain
    \begin{equation*}
        \iota_{X*}^{-1}\circ \pi(f^* \mc F) = Ic^*Ie^*(\lambda_{-1}(\mc N_{\iota_{\mc W}}^\vee)^{-1}) \cdot f^!(\iota_{\mc W}^! \circ \pi(\mc F))
    \end{equation*}
    $$ = If^*Id^*(\lambda_{-1}(\mc N_{\iota_{\mc W}}^\vee)^{-1}) \cdot f^!(\iota_{\mc W}^! \circ \pi(\mc F)) = f^!(Id^*(\lambda_{-1}(\mc N_{\iota_{\mc W}}^\vee)^{-1}) \cdot \iota_{\mc W}^!(\pi(\mc F)))$$
    by Lemma \ref{lem: Gysin pullback K0}(b)
    $$= f^!(\iota_{\mc Y*}^{-1} \circ \pi(\mc F))$$
    by Lemma \ref{lem: localization inverse} applied to $d: Y \hookrightarrow W$.

    To prove the second equality in the statement of the lemma, we start by noting that $e$ is regular embedding since $Z,W$ are smooth \cite[Tag 0E9J]{stacks-project}, and $f$ is a regular embedding by assumption, so by Lemma \ref{lem: Tor indep}(b) we have that $Y$ and $Z$ are $\Tor$-independent over $W$.
    Thus by Lemma \ref{lem: excess intersection}(a) we have
    $$f^! = e^! : K_0(\mc Y^g)_{\mf m_g, \BC} \to K_0(\mc X^g)_{\mf m_g, \BC}$$
    and by Lemma \ref{lem: excess intersection}(b) we have
    $$e^! = (Ic)^* \lambda_{-1}((\iota_{\mc Z}^*\mc N_{\mc Z/\mc W}^\vee)^{\mov}) \cdot  (Ie)^!.$$
\end{proof}

\begin{lem} \label{lem: excess intersection}
    Consider a Cartesian diagram of $G$-equivariant varieties
    \[
    \begin{tikzcd}
        A \arrow[r, "a"] \arrow[d, "c"] & B \arrow[r, "b"] \arrow[d,hook, "e"] & C \arrow[d,hook, "f"] \\
        D \arrow[r, "h"] & E \arrow[r, "i"] & H
    \end{tikzcd}
    \]
    such that $e$ and $f$ are regular embeddings.\\
    (a) If $f$ and $i$ are $\Tor$-independent, then $e^! = f^!: K_0(\mc D) \to K_0(\mc A)$.\\
    (b) If $h$ and $f$ are closed embeddings, $C$ and $H$ are smooth, $G$ is finite abelian, and $E = H^g$ for some $g \in G$, then
    $$f^! = a^* \lambda_{-1}((b^*\mc N_{\mc C/\mc H}^\vee)^{\mov}) \cdot e^!: K_0(\mc D)_{\BC} \to K_0(\mc A)_{\mf m_g, \BC}.$$
\end{lem}

\begin{proof}
    (a) It suffices to show that given ring maps $R \to R' \to R''$ and $R \to S$, where $\Tor^R_i(R', S) = 0$ for $i > 0$, and given $M \in \Mod(R'')$, the natural (thus $G$-equivariant) map $\Tor_i^R(M, S) \to \Tor_i^{R'}(M, R' \otimes_R S)$ is an isomorphism. Picking a resolution $S^\bullet$ of $S$ by free $R$-modules, note that $R' \otimes_R S^\bullet$ is a resolution of $R' \otimes_R S$ by free $R'$-modules, as the homology groups of this complex calculate $\Tor^R_*(R', S)$. We obtain an isomorphism of the complexes $M \otimes_R S^\bullet = M \otimes_{R'} (R' \otimes_R S^\bullet)$ computing $\Tor$ groups on each side. \\
    (b) We start by noting that the Gysin morphism
    $$f^{(-)}(-): K_0(\mc C) \times K_0(\mc D) \to K_0(\mc A)$$
    (where here all elements of $K_0(\mc C)$ are $f$-perfect by smoothness of $H$) is a bilinear morphism of $K^0(BG)$-modules (Lemma \ref{lem: Gysin pullback K0}), and so defines a morphism
    $$f^{(-)}(-): K_0(\mc C)_{\mf m_g, \BC} \times K_0(\mc D)_{\mf m_g, \BC} \to K_0(\mc A)_{\mf m_g, \BC}.$$
    Recall that $f^!(-) := f^{[\ms O_{\mc C}]}(-)$, where
    $$[\ms O_{\mc C}] = b_*(\lambda_{-1}(\mc N^\vee_{\mc C^g/\mc C})^{-1}) \in K_0(\mc C)_{\mf m_g, \BC}$$
    by \cite[Lemme 3.3]{Thomason}, since $C$ is assumed smooth and $B = C^g$. Then by definition of Gysin pullback
    $$f^{b_*(\lambda_{-1}(\mc N^\vee_{\mc C^g/\mc C})^{-1})} = (f \circ b)^{\lambda_{-1}(\mc N^\vee_{\mc C^g/\mc C})^{-1}}: K_0(\mc D)_{\mf m_g, \BC} \to K_0(\mc A)_{\mf m_g, \BC}$$
    and
    $$(f \circ b)^{\lambda_{-1}(\mc N^\vee_{\mc C^g/\mc C})^{-1}} = a^*(\lambda_{-1}(\mc N^\vee_{\mc C^g/\mc C})^{-1}) \cdot (f \circ b)^{\ms O_B} $$
    by Lemma \ref{lem: Gysin pullback K0}(a). Thus
    \begin{equation} \label{eq: f!}
        f^!(-) = a^*(\lambda_{-1}(\mc N^\vee_{\mc C^g/\mc C})^{-1}) \cdot (f \circ b)^!(-): K_0(\mc D)_{\mf m_g, \BC} \to K_0(\mc A)_{\mf m_g, \BC}
    \end{equation}
    
    Next, we claim that for any $\mc F \in K_0(\mc D)_{\BC}$ we have
    $$(f \circ b)^!(\mc F) = (i \circ e)^!(\mc F) = 
    a^* e^*\lambda_{-1}(\mc N^\vee_{\mc H^g/\mc H}) \cdot e^! \mc F$$
    Indeed, consider the morphisms of relative $K$-groups
    $$(\cap \mc F)_{\mc H}: K^0_{\mc B}(\mc H) \to K_0(\mc A), \,\,\,\,\, (\cap \mc F)_{\mc E}: K^0_{\mc B}(\mc E) \to K_0(\mc A),$$
    as well as the pullback and pushforward
    $$i^*: K^0_{\mc B}(\mc H) \to K^0_{\mc B}(\mc E), \,\,\,\,\, i_*: K^0_{\mc B}(\mc E) \xrightarrow{\sim} K^0_{\mc B}(\mc H).$$
    We note that by the projection formula
    $$(\cap \mc F)_{\mc E} \circ i^* = (\cap \mc F)_{\mc H}: K^0_{\mc B}(\mc H) \to K_0(\mc A).$$
    Then
    \begin{equation} \label{eq: excess 1}
        (i \circ e)^!(\mc F) := (\cap \mc F)_{\mc H}([i_* e_* \ms O_{\mc B}]) = (\cap \mc F)_{\mc E}(i^*[i_* e_* \ms O_{\mc B}]).
    \end{equation}
    Next
    \begin{equation} \label{eq: excess 2}
        i^* \circ i_* = (\lambda_{-1}(\mc N^\vee_{\mc H^g/\mc H}) \cdot) : K^0_{\mc B}(\mc E) \to K^0_{\mc B}(\mc H) \to K^0_{\mc B}(\mc E)
    \end{equation}
    as in \cite[Lemme 3.3]{Thomason}: because $i_*$ is an isomorphism (both groups being identified with $K_0(\mc B)$ by smoothness of $E,F$), it suffices to check that for a complex of $G$-equivariant sheaves $S$ supported on $B$ we have
    $$i_* \circ i^* \circ i_*([S]) = [i_*i^*i_* S] = [i_* \ms O_E \otimes i_* S] = [i_*(i^* i_* \ms O_E \otimes S)] = i_*(\lambda_{-1}(\mc N^\vee_{\mc H^g/\mc H}) \cdot [S]).$$
    Thus
    $$(i \circ e)^!(\mc F) \overset{(\ref{eq: excess 1})}{=} (\cap \mc F)_{\mc E}(i^*i_*[e_* \ms O_B]) \overset{(\ref{eq: excess 2})}{=} (\cap \mc F)_{\mc E}(\lambda_{-1}(\mc N^\vee_{F^g/F}) \cdot [e_* \ms O_B])$$
    $$ = (\cap \mc F)_{\mc E}(e_*[e^*\lambda_{-1}(\mc N^\vee_{\mc H^g/\mc H})]) = e^{e^* \lambda_{-1}(\mc N^\vee_{\mc H^g/\mc H})}(\mc F) = a^* e^*\lambda_{-1}(\mc N^\vee_{\mc H^g/\mc H}) \cdot e^! \mc F$$
    by Lemma \ref{lem: Gysin pullback K0}(a).
    
    Comparing this to (\ref{eq: f!}), we see that
    $$f^!(\mc F) = a^* \lambda_{-1}(\mc N^\vee_{\mc C^g/\mc C})^{-1} \cdot (i \circ e)^!(\mc F) = a^*(\lambda_{-1}(\mc N^\vee_{\mc C^g/\mc C})^{-1} \cdot e^* \lambda_{-1}(\mc N^\vee_{\mc H^g/\mc H})) \cdot e^! \mc F \in K_0(\mc A)_{\mf m_g, \BC}.$$
    Thus it suffices to note that
    $$\lambda_{-1}(\mc N^\vee_{\mc C^g/\mc C})^{-1} \cdot e^* \lambda_{-1}(\mc N^\vee_{\mc H^g/\mc H}) = \lambda_{-1}((b^*\mc N_{\mc C/\mc H}^\vee)^{\mov}) \in K^0(\mc B)_{\mf m_g, \BC}$$
    since
    $$(b^*\mc N_{\mc C/\mc H}^\vee)^{\mov} = (e^* i^* T_{\mc H}^\vee - b^* T_{\mc C}^\vee) - (e^* i^* T_{\mc H}^\vee - b^* T_{\mc C}^\vee)^{\fix} = e^*(\mc N^\vee_{\mc H^g/\mc H}) - \mc N^\vee_{\mc C^g/\mc C}$$
    and thus
    $$\lambda_{-1}(e^*\mc N^\vee_{\mc H^g/\mc H}) = \lambda_{-1}((b^*\mc N_{\mc C/\mc H}^\vee)^{\mov}) \cdot \lambda_{-1}(\mc N^\vee_{\mc C^g/\mc C}) \in K^0(\mc B).$$
\end{proof}

\begin{lem} \label{lem: Gysin pullback K0}
    Suppose given a Cartesian square of $G$-equivariant varieties
    \[
    \begin{tikzcd}
        A \arrow[r, "a"] \arrow[d, "b"] & B \arrow[d, hook, "c"]\\
        C \arrow[r, "d"] & D
    \end{tikzcd}
    \]
    with $c$ a closed embedding. \\
    (a) If $D$ is smooth, the Gysin pullback functor
    $$c^{(-)}(-): K_0(\mc B) \times K_0(\mc C) \to K_0(\mc A)$$
    is $K^0(\mc B)$-linear, where $K^0(\mc B)$ acts on the left side via its action on $K_0(\mc B)$ and on the right side via the action of $a^*: K^0(\mc B) \to K^0(\mc A)$ on $K_0(\mc A)$. \\
    (b) Given a $c$-perfect $G$-equivariant complex $\ms F \in D^b\Coh(B)$, the Gysin pullback
    $$c^{\ms F}: K_0(\mc C) \to K_0(\mc A)$$
    is $K^0(\mc C)$-linear, where $K^0(\mc C)$ acts on the left as usual and on the right through $b^*: K^0(\mc C) \to K^0(\mc A)$.
\end{lem}

\begin{proof}
    For (a), note that since $c$ is a closed embedding into a smooth variety, all coherent sheaves on $B$ are in fact $c$-perfect, so the map above is well-defined on $K_0(\mc B)$.

    Now, for both parts of the lemma, it suffices to note that, in the affine case $R \to S, T$, given modules $N \in \Mod(S)$ and $M, L \in \Mod(T)$ with $L$ free and finitely generated over $T$, the natural (thus $G$-equivariant) morphism $\Tor_i^R(N, M \otimes_T L) \to \Tor^R_i(N, M) \otimes_T L = \Tor^R_i(N, M) \otimes_{S \otimes_R T} (S \otimes_R L)$ is an isomorphism. Taking a free $R$-resolution $N^\bullet$ of $N$, we have
    $$\Tor^R_i(N, M \otimes_T L) = H_i(N^\bullet \otimes_R M \otimes_T L) = H_i(N^\bullet \otimes_R M) \otimes_T L = \Tor^R_i(N, M) \otimes_T L$$
    where the second equality follows from the fact that $L$ is a flat $T$-module.
\end{proof}

\begin{lem} \label{lem: Tor indep}
    Suppose there is a Cartesian square of varieties
    \[
    \begin{tikzcd}
    A \arrow[r, "f'"] \arrow[d, "g'"] & B \arrow[d, "g"] \\
    C \arrow[r, hook, "f"] & D
    \end{tikzcd}
    \]
    where $f$ is a regular embedding of codimension $d$.\\
    (a) If $B$ is Cohen-Macaulay, and $A$ and $B$ are equidimensional with $\dim B - \dim A = d$, then $f'$ is also a regular embedding of codimension $d$. \\
    (b) If $f'$ is a regular embedding of codimension $d$, then
    $B$ and $C$ are $\Tor$-independent over $D$. In particular, we have a canonical isomorphism
    $$Lf^* Rg_* \ms F = Rg'_* Lf'^* \ms F$$
    for any $\ms F \in D^b\Coh(B).$
\end{lem}

\begin{proof}
    (a) Choose any pair of points $x \in C \subset D$ and $y \in g^{-1}(x)$. By assumption, $\ms I_{C, x} \subset \ms O_{D,x}$ is generated by a regular sequence $f_1, \ldots, f_d$. Since $A$ and $B$ are varieties, we have $\dim \ms O_{B,y} - \dim \ms O_{A,y} = (\dim B - \dim \bar{y}) - (\dim A - \dim \bar{y}) = d$ \cite[Tag 00OS]{stacks-project}. Now, since $g^* f_1, \ldots, g^* f_d$ generate $\ms I_{A,y} \subset \ms O_{B,y}$ and $B$ is Cohen-Macaulay, it follows that $g^* f_1, \ldots, g^* f_d$ form a regular sequence \cite[Tag 02JN]{stacks-project}. \\
    (b) Choose points $x \in C \subset D, y \in g^{-1}(x)$ and a Koszul-regular sequence $f_1, \ldots, f_d \in \ms I_{C,x}$. (Regular implies Koszul-regular by \cite[Tag 062F]{stacks-project}.) Then $\Tor_i^{\ms O_{D,x}}(\ms O_{C,x}, \ms O_{B,y})$ can be calculated using the Koszul resolution $K_\bullet(\ms O_{D,x}, f_1, \ldots, f_d)$ of $\ms O_{C,x}$. Since $g^* f_1, \ldots, g^* f_d$ generate $\ms I_{A,y}$ and this can be generated by a regular sequence of length $d$, we conclude that $g^* f_1, \ldots, g^* f_d$ is still a Koszul-regular sequence \cite[Tag 066A]{stacks-project}.
    In particular the pullback Koszul complex $K_\bullet(\ms O_{D,x}, f_1, \ldots, f_d) \otimes_{\ms O_{D,x}} \ms O_{B,y} = K_\bullet(\ms O_{B,y}, g^*f_1, \ldots, g^*f_d)$ is exact in positive degree, and so $\Tor_i^{\ms O_{D,x}}(\ms O_{C,x}, \ms O_{B,y}) = 0$ for $i > 0$, proving Tor independence. The second part now follows from \cite[Tag 08IB]{stacks-project}.   
\end{proof}

\section{Fourier transform for compactified Pryms} \label{sec: main theorem}
\subsection{Relative correspondences and homological realization} \label{subsec: stacky MSY background}
We quickly review the definitions and notations of \cite[\S 1.3.2-3, 2.2]{MSY} (as extended to \textit{nice quotients}, i.e., global quotients of quasiprojective varieties by finite groups). In this section, all functors are derived.

\subsubsection{Coherent sheaves} \label{subsubsec: coherent sheaves}
Fix smooth nice quotients $\mc X_i = [X_i/G_i]$ with flat, proper maps to a smooth variety $B$. Given $\mc F \in D^b\Coh(\mc X_1 \times_B \mc X_2)$, which in this case can be defined as the bounded derived category of the abelian category of $G_1 \times G_2$-equivariant sheaves on $X_1 \times_B X_2$ (and similarly for $\mc X_1, \mc X_2$),
we have a relative Fourier-Mukai transform
$$\Phi_{\mc F}: D^b\Coh(\mc X_1) \to D^b\Coh(\mc X_2), \mc G \mapsto p_{2*}(p_1^* \mc G \otimes \mc F).$$
Given $\mc F \in D^b\Coh(\mc X_1 \times_B \mc X_2)$ and $\mc G \in D^b\Coh(\mc X_2 \times_B \mc X_3)$ we define a composition
\begin{equation} \label{eq: composition on Db}
    \mc G \circ \mc F := p_{13*}\delta_{\mc X_2}^*(p_{12}^* \mc F \otimes p_{34}^* \mc G) \in D^b\Coh(\mc X_1 \times_B \mc X_3)
\end{equation}
where $p_{ij}$ is the projection onto the corresponding factors and
$$\delta_{\mc X_2}: \mc X_1 \times_B \mc X_2 \times_B \mc X_3 \to \mc X_1 \times_B \mc X_2 \times \mc X_2 \times_B \mc X_3$$
is the (flat) base-change of the diagonal map $\Delta_{\mc X_2}: \mc X_2 \to \mc X_2 \times \mc X_2$. Then
$$\Phi_{\mc G \circ \mc F} = \Phi_{\mc G} \circ \Phi_{\mc F}: D^b\Coh(\mc X_1) \to D^b\Coh(\mc X_3).$$

For the purposes of \S \ref{sec: projectors and mult}, we will also define a multiple composition, as follows: given $\mc X_1, \mc X_2, \mc X_3, \mc Y_1, \mc Y_2$ and $B$ as above, we define a map
$$D^b\Coh(\mc X_1 \times_B \mc X_2 \times_B \mc X_3) \times  D^b\Coh(\mc Y_1 \times_B \mc X_1) \times D^b\Coh(\mc Y_2 \times_B \mc X_2) \to D^b\Coh(\mc Y_1 \times_B \mc Y_2 \times_B \mc X_3)$$
$$(\mc G, \mc F_1, \mc F_2) \mapsto \mc G \circ (\mc F_1 \boxtimes \mc F_2):= p_{125*}\delta_{\mc X_1}^*\delta_{\mc X_2}^*(p_{12}^*\mc F_1 \otimes p_{34}^* \mc F_2 \otimes p_{567}^*\mc G) $$
where
$$\delta_{\mc X_2}: \mc Y_1 \times_B \mc X_1 \times \mc Y_2 \times_B \mc X_1 \times_B \mc X_2  \times_B \mc X_3 \to \mc Y_1 \times_B \mc X_1 \times \mc Y_2 \times_B \mc X_2 \times \mc X_1 \times_B \mc X_2 \times_B \mc X_3$$
and
$$\delta_{\mc X_1}:  \mc Y_1 \times_B \mc Y_2 \times_B \mc X_1 \times_B \mc X_2  \times_B \mc X_3 \to \mc Y_1 \times_B \mc X_1 \times \mc Y_2 \times_B \mc X_1 \times_B \mc X_2  \times_B \mc X_3$$
are (flat) base-changes of the diagonals $\Delta_{\mc X_2}$ and $\Delta_{\mc X_1}$, respectively. (Note that $\delta_{\mc X_2} \circ \delta_{\mc X_1}$ is a base change of the diagonal $\Delta_{\mc X_1 \times \mc X_2}$, but not a \textit{flat} base change when the base $B$ is nontrivial.)

\subsubsection{Relative correspondences and Chow motives} \label{subsubsec: relative Chow}
Chow groups of nice quotients are discussed in \S \ref{subsubsec: EG def of Chow}. Given equidimensional, smooth nice quotients $\mc X_i$ with proper (but not necessarily flat) morphisms to a smooth variety $B$, we define the group of (complex-valued) relative correspondences to be
$$\Corr^k_B(\mc X_1, \mc X_2) := \CH_{\dim \mc X_2 - k}(\mc X_1 \times_B \mc X_2)_{\BC}.$$
Given $\alpha \in \Corr^k_B(\mc X_1, \mc X_2)$ and $\beta \in \Corr^l_B(\mc X_2, \mc X_3)$, we define the composition
$$\beta \circ \alpha := p_{13*} \Delta_{\mc X_2}^!(\alpha \times \beta) \in \Corr^{k+l}_B(\mc X_1, \mc X_3)$$
where $p_{13}$ and $\Delta_{\mc X_2}$ are as in \S \ref{subsubsec: coherent sheaves}. We say that $\alpha \in \Corr^0_B(\mc X, \mc X)$ is a \textit{projector} if $\alpha \circ \alpha = \alpha$.

We also recall from \cite[\S 2.2.3]{MSY} the following notion of multiple correspondence: Given $\mc X_i$ and $B$ as before, we let
$$\Corr^k_B(\mc X_1, \mc X_2; \mc X_3) := \Corr^k_B(\mc X_1 \times_B \mc X_2, \mc X_3).$$
Compositions are described as follows: given smooth $\mc Y_i$ proper over $B$, $\alpha_3 \in \Corr^{k_3}_B(\mc X_3, \mc Y_3)$, and $\beta \in \Corr^k_B(\mc X_1, \mc X_2; \mc X_3)$, we define
$$\alpha_3 \circ \beta \in \Corr_B^{k+k_3}(\mc X_1 \times_B \mc X_2, \mc Y_3) = \Corr^{k+k_3}_B(\mc X_1, \mc X_2; \mc Y_3)$$
by the formula above, viewing $\beta$ as an element of $\Corr^k_B(\mc X_1 \times_B \mc X_2, \mc X_3)$. On the other hand, given
$\alpha_i \in \Corr^{k_i}_B(\mc Y_i, \mc X_i)$, we define
$$\beta \circ (\alpha_1 \times \alpha_2) := p_{125*} \Delta_{\mc X_1}^!\Delta_{\mc X_2}^!(\alpha_1 \times \alpha_2 \times \beta) \in  \Corr^{k+k_1+k_2}_B(\mc Y_1, \mc Y_2; \mc X_3)$$
where $p_{125}, \Delta_{\mc X_1}, \Delta_{\mc X_2}$ are as in \S \ref{subsubsec: coherent sheaves}.

Returning to varieties (although one could make the same definitions for nice quotients), the category of relative Chow motives $CHM(B)$ over $B$ has objects $(X, \alpha, n)$, where $X$ is a smooth variety, proper over $B$; $\alpha \in \Corr^0_B(X, X)$ is a projector; and $n \in \BZ$ is an integer. Morphisms are given by
$$\Hom((X, \alpha, n), ( Y, \beta, m)) := \beta \circ \Corr_B^{m-n}( X, Y) \circ \alpha.$$
Note that if $n=m$ and $X=Y$ there is a preferred morphism given by $\Delta_X \in \Corr^0_B(X, X)$.
We define the motive of $X$ to be
$$h(X) := (X, \Delta_{X}, 0).$$

\subsubsection{Constructible sheaves} \label{subsubsec: constructible sheaves}
As noted in \cite[\S 2.8]{EG-locforequivint}, there is a cycle class map
$$cl: \CH_*(\mc X) \to H^{BM}_{2*}(\mc X)$$
given by the standard cycle class map of \cite[\S 19.1]{Fulton} on the approximation $X_G$ (see \S \ref{subsubsec: EG def of Chow}). Recall that the cycle class on schemes commutes with Chern classes \cite[Proposition 19.1.2]{Fulton}, proper pushforward \cite[Lemma 19.1.2]{Fulton}, and Gysin pullback by local complete intersections \cite[Example 19.2.1]{Fulton}. It follows immediately that the same is true for Chern classes of nice quotients and for representable morphisms. Moreover, by taking representable covers as in (\ref{diagram: proper pushforward}), we see that the same holds also for nonrepresentable morphisms.

Let $\mc X_i$ be smooth nice quotients equipped with proper morphisms $\pi_i: \mc X_i \to B$. Then there is a natural isomorphism \cite[Lemma 2.15(2)]{Corti-Hanamura}
\begin{equation} \label{eq: BM = sheaf Hom}
   H^{BM}_{2 \dim \mc X_2 - i}(\mc X_1 \times_B \mc X_2) = \Hom_{D^b_c(B)}(\pi_{1*} \BC_{\mc X_1}, \pi_{2*} \BC_{\mc X_2}[i]), 
\end{equation}
where $D^b_c(B)$ is the bounded derived category of constructible $\BC$-sheaves on $B$. Moreover, this isomorphism is compatible with composition, when composition on $H^{BM}_*$ is defined as for relative Chow correspondences \cite[Lemma 2.17]{Corti-Hanamura}. The corresponding statement for multiple correspondences
$$H^{BM}_{2 \dim \mc X_3 - i}(\mc X_1 \times_B \mc X_2 \times_B \mc X_3) = \Hom_{D^b_c(B)}(\pi_{1*} \BC_{\mc X_1} \otimes \pi_{2*} \BC_{\mc X_2}, \pi_{3*} \BC_{\mc X_3}[i]) $$
is \cite[Lemma 2.2]{MSY}. Thus, by the compatibilities in the previous paragraph, after applying the cycle class map, composition of relative Chow correspondences corresponds to composition of morphisms of sheaves. (Again, although the results cited in this paragraph are for schemes, the same follows for the nice quotients $\mc X_i = [X_i/G_i]$ by taking covers, using the pullback isomorphisms $\pi_* \BC_{\mc X_i} = (\pi_* \BC_{X_i})^{G_i}$, $H^{BM}_*(\mc X_1 \times_B \mc \mc X_2) = H^{BM}_*(X_1 \times_B X_2)^{G_1 \times G_2}$, etc. Recall that the diagonal $\Delta_{\mc X}: \mc X \to \mc X \times \mc X$ pulls back to the morphism $(\sigma, p_2): G \times X \to X \times X$.)

We note in particular that map of sheaves corresponding to the class of the small diagonal $[\Delta_{\mc X}^{s}] \in \Corr^0_B(\mc X, \mc X; \mc X)$ is the cup product
$$\cup: \pi_*\BC_{\mc X} \otimes \pi_*\BC_{\mc X} \to \pi_* \BC_{\mc X}.$$

As explained in \cite[\S 2.2.2]{MSY}, there is a homological realization map
$$CHM(B) \to D^b_c(B), (\pi: X \to B, \alpha, n) \mapsto \alpha_*(\pi_* \BC_{ X}[2n]),$$
where by $\alpha_*(\pi_* \BC_X[2n])$ we mean the $2n$-shifted image of $\pi_* \BC_{ X}$ under the projector which is the image of $\alpha$ under
$$\CH_{\dim  X}( X \times_B  X) \xrightarrow{cl} H^{BM}_{2\dim X}( X \times_B X) = \Hom_{D^b_c(B)}(f_*\BC_{ X}, f_*\BC_{ X}).$$
(The fact that it is well-defined to ``take the image" of a projector in $D^b_c(B)$ is \cite[Lemma 2.14]{Corti-Hanamura}.) In particular, the realization of $h(X)$ is $\pi_* \BC_{ X}$.

Recall that given a proper map $\pi: X \to B$, where $X$ is smooth, the decomposition theorem of \cite{BBD} states that the pushforward $\pi_* \BC_X$ splits (noncanonically) as a sum of its shifted perverse cohomology sheaves, i.e., 
$$\pi_* \BC_X \cong \oplus_i \pH^i(\pi_* \BC_X)[-i],$$
and furthermore each $\pH^i(\pi_* \BC_X)$ is a sum of simple perverse sheaves. 
If a finite group $G$ acts on $X$ over $B$, note that there is an induced action on $\pi_* \BC_X$, thus on each perverse cohomology group, which splits into a sum of semisimple perverse sheaves indexed by representations. In particular, there is a splitting
$$\pi_* \BC_X \cong \oplus_\rho (\pi_* \BC_X)_\rho$$
where $\rho$ runs over the set of irreducible representations of $G$ and each term is again a sum of shifted simple perverse sheaves, each supported on some closed subset of $B$.

Lastly, recall that the perverse truncations
$$\ptau_{\le k} \pi_* \BC_X \to \pi_* \BC_X, \,\,\, \pi_* \BC_X \to \ptau_{\ge k} \BC_X$$
are canonical and, after taking global sections, induce a canonical filtration
$$P_k H^*(X, \BC) := \mathrm{Im}(H^*(X, \ptau_{\le k+\dim B} \pi_* \BC_X) \to H^*(X, \BC)) \subset H^*(X, \BC).$$

\subsection{The compactified Prym setup} \label{subsec: compactified Prym}
For the rest of the paper, we fix a family of projective, locally planar, integral curves $X$ over a smooth quasiprojective variety $B$, a smooth projective curve $C$ of genus $g \ge 2$, and a flat, finite map $c: X \to C \times B$ over $B$ of degree $n$.

\begin{defn} \label{defn: compactified Prym}
    Let $\oJ^0(X/B)$ be the relative compactified Jacobian of $X/B$, parametrizing torsion free sheaves of rank 1 and degree 0 on the fibers $X_b$.
    Consider the relative norm map
    $$\Nm_{X/C}: \oJ^0(X/B) \to \J^0(C) \times B$$
    given pointwise by 
    $$\ms L_b \in \oJ^0(X_b) \mapsto \det (c_* \ms L_b) \otimes \det(c_* \ms O_{X_b})^\vee \in \J^0(C).$$
    (Note that the pushforward of a rank 1 torsion-free sheaf on $X_b$ is a rank $n$ vector bundle on $C$.)\\
    The compactified Prym fibration associated to $X \to B$ and $C$ is
    $$\pi: \oPrym(X/B) := \Nm_{X/C}^{-1}(\{\ms O_C\} \times B) \to B.$$
\end{defn}

\begin{rmk} \label{rmk: i is reg}
    Given our assumptions, the variety $\oJ^0(X/B)$ is Cohen-Macaulay \cite[Theorem 9]{AIK}. Note that $\Nm_{X/C}$ is $\J^0(C)$-equivariant, where $\J^0(C)$ acts on $\oJ^0(X/B)$ by pullback and tensor product, and on $\J^0(C) \times B$ by tensor product by the $n^{\mathrm{th}}$ power (in particular, transitively over each point of $B$). Then all fibers are of the same dimension $g(\tilde X_b) - g$, and by miracle flatness we conclude that $\Nm_{X/C}$ is flat. In particular, the embedding $\oPrym(X/B) \hookrightarrow \oJ^0(X/B)$, which is the flat base change of $\{\ms O_C\} \hookrightarrow \J^0(C)$, is regular of codimension $g$.
\end{rmk}

\begin{rmk} \label{rmk: etale cover}
    From the tensor product morphism
    $$\J^0(C) \times \oPrym(X/B) \to \oJ^0(X/B)$$
    we see that $\oJ^0(X/B)$ is the quotient of $\J^0(C) \times \oPrym(X/B)$ by the (free) diagonal action of the $n$-torsion subgroup $\J^0(C)[n]$. In particular, $\oPrym(X/B)$ is smooth if and only if the same is true of $\oJ^0(X/B)$.
\end{rmk}

\begin{defn} \label{defn: Gamma}
    Set $\Gamma := \J^0(C)[n]$. Then $\Gamma$ acts on $\oPrym(X/B)$ over $B$ by tensor product: for $\gamma \in \Gamma$ we set
    $$\gamma \cdot \ms L_b := c_b^*\gamma \otimes \ms L_b \in \oPrym(X_b).$$
    We use $\oPrym(X/B)^\gamma$ to denote the fixed locus and
    $$B_\gamma := \pi(\oPrym(X/B)^\gamma) \subset B$$
    to denote its image under $\pi$. (Here, the use of a subscript rather than superscript is to indicate that $B_\gamma$ is not a fixed locus, and indeed $B$ is not equipped with a nontrivial action of $\Gamma$.) Abusively, we write $IB := \coprod_{\gamma \in \Gamma} B_\gamma$ and $\iota_B: IB \to B$ for the obvious inclusion map, although $IB$ is not an inertia stack. 
\end{defn}

\begin{rmk} \label{rmk: endoscopic locus defs}
    Another, perhaps more usual, characterization of the \textit{endoscopic locus} $B_\gamma$ is given by
    $$b \in B_\gamma \iff \gamma \in \ker(\J^0(C) \to \J^0(\tilde X_b)),$$
    where $\nu: \tilde X_b \to X_b$ is the normalization.
    Indeed, if we have $b \in B_\gamma$, then by definition there is some $F \in \oPrym(X_b)$ such that $c^*_b \gamma \otimes F \cong F$. Then $\nu^* c^*_b \gamma$ fixes the line bundle $\nu^*F/T \in \J^0(\tilde X_b)$, where $T \subset \nu^* F$ is the torsion subsheaf, thus $\nu^* c^*_b \gamma \cong \ms O_{\tilde X_b}$. In the other direction, note that we have a pushforward map $\nu_*: \J^{-k}(\tilde X_b) \to \oJ^0(X_b)$, for $k := h^0(X_b, \nu_* \ms O_{\tilde X_b}/\ms O_{X_b})$ (since the pushforward of a line bundle from $\tilde X_b$ remains torsion-free of rank 1), and if $\nu^* c^*_b \gamma \cong \ms O_{\tilde X_b}$, then by the projection formula, the image of $\nu_*$ is a $\gamma$-fixed point in $\oJ^0(X_b)$; moreover, note that the map $\nu_*$ is $\J^0(C)$-equivariant, and since $\oPrym(X/B)$ meets each $\J^0(C)$-orbit, in particular it meets the image of $\nu_*$.
\end{rmk}

\begin{defn} \label{def: Mvee}
    Set $M := \oPrym(X/B)$. We define
    $$M^\vee := [\oPrym(X/B)/\Gamma] = [\oJ^0(X/B)/\J^0(C)].$$
    We also introduce notation
    $$IM^\vee = \coprod_{\gamma \in \Gamma} [M^\gamma/\Gamma] =: \coprod_{\gamma \in \Gamma} M^{\gamma, \vee}.$$
    Since $\Gamma$ acts along the fibers of $M \to B$, we can descend $\pi$ to a morphism $\pi^\vee: M^\vee \to B$. Note as well that we have a natural morphism $\pi_{IB}: IM^\vee \to IB$ such that $\iota_B \circ \pi_{IB} = \pi^\vee \circ \iota_{M^\vee}$.
\end{defn}
The notation reflects the fact that $M \to B$ and $M^\vee \to B$ are dual abelian fibrations: for a general $b \in B$, the fiber $M_b$ is an extension of an abelian variety by a finite abelian group, and $M^\vee_b$ is the dual finite abelian group gerbe on an abelian variety, in a sense to be described in \S \ref{subsubsec: duality for ab stacks}. (If the finite group in question is trivial, then this is simply the dual abelian variety.)

\begin{defn} \label{def: Weil pairing}
    Consider the canonical pairing on $\Gamma$ induced by the intersection pairing on $H^1(C, \BZ/n\BZ) = \J^0(C)[n]$. This gives an isomorphism
    $$\Gamma \to \Gamma^\vee, \gamma \mapsto \langle \gamma,-\rangle,$$
    where $\Gamma^\vee$ is the dual abelian group.
    For any fixed $\gamma \in \Gamma$, we use $\kappa$ to denote the corresponding element $\langle \gamma, - \rangle \in \Gamma^\vee$.
\end{defn}

Since $\Gamma$ acts on $\oPrym(X/B)$ over $B$, as described in \S \ref{subsubsec: constructible sheaves} there is a decomposition
$$R\pi_*\BC_{\oPrym} = \oplus_{\kappa \in \Gamma^\vee} (R\pi_* \BC_{\oPrym})_\kappa.$$

\begin{defn} \label{defn: delta invt}
    Recall that given a reduced curve $Y$ with normalization $\nu: \tilde Y \to Y$, the delta invariant is defined by $\delta(Y) := h^0(Y, \nu_* \ms O_{\tilde Y}/\ms O_Y),$ or alternatively the dimension of the affine (as opposed to abelian variety) part of $\J^0(Y)$. Given a subvariety $B' \subset B$, we define $\delta(B')$ to be $\delta(X_b)$ for a general point $b \in B'$.
\end{defn}

\begin{defn} \label{defn: good fibration}
    Let $X \to C \times B$ be as defined at the beginning of \S \ref{subsec: compactified Prym}. Let $M := \oPrym^0(X/B) \to B$ be the corresponding relative compactified Prym variety. 
    We say that $\pi: M \to B$ is a \textit{good compactified Prym fibration} if
    \begin{enumerate}
        \item $M$ and $IB$ are smooth, and $\pi_{IB}: IM^\vee \to IB$ is flat
        \item (``Ng\^o conditions") For each $\gamma \in \Gamma$, the component $(R\pi_* \BC_M)_\kappa$ has full support on $B_\gamma$, and each $B_\gamma$ satisfies the support inequality $\delta(B_\gamma) \ge \codim B_\gamma$
        \item The general fiber of $X|_{B_\gamma} \to B_\gamma$ is at worst nodal.
    \end{enumerate}
\end{defn}

\begin{rmk} \label{rmk: Severi/Ngo}
    By the Severi inequality (see, e.g., \cite[Lemma 4.1]{MS-chi-indep}), smoothness of $M$ (or equivalently, of $\oJ^0(X/B)$, by Remark \ref{rmk: etale cover}) implies that $\delta(B_\gamma) \le \codim B_\gamma$. Thus, the second part of (2) gives an equality $\delta(B_\gamma) = \codim B_\gamma$.
\end{rmk}

\begin{rmk} \label{rmk: Ngo}
    We call the conditions (2) ``Ng\^o conditions" because in the cases of interest, they will follow from \cite[Proposition 7.2.3]{Ngo-fundamental}. However, a priori we do not require the fibration $M \to B$ to be a weak abelian fibration as defined by Ng\^o; for the purposes of this paper, the two listed conditions suffice.
\end{rmk}

\begin{rmk} \label{rmk: condition 3}
    Condition (3) could probably be weakened, given an analysis of the Arinkin Poincar\'e sheaf and compactified Jacobian of curves with worse than nodal singularities, along the lines of \cite{FHHO}. However, the current formulation suffices for our cases of interest.
\end{rmk}

Recall that the moduli space $\check{M}_{n, \ms L}^D(C)$ of semistable, $D$-twisted $\SL_n$ Higgs bundles on $C$ with fixed determinant $\ms L \in \Pic(C)$ (where $D$ is a divisor on $C$), together with the Hitchin fibration $h: \check{M}_{n, \ms L}^D(C) \to \oplus_{i=2}^n H^0(C, \ms O_C(D)) =: \BA$ is a compactified Prym \textit{torsor} fibration, over the family of curves given by the spectral curve $X \subset |D| \times \BA \to C \times \BA$, where $|D|$ is the total space of the bundle $\ms O_C(D)$ on $C$. Since $|D|$ is a smooth surface, all curves in the family are locally planar. We now restrict to the open subset $B \subset \BA$ over which the spectral curves are integral (known as the \textit{elliptic locus}), and write $\pi: M \to B$ for the restricted Hitchin fibration. For each $b \in B$ we have
$$M_b = (\det a_{b*})^{-1}(\ms L) = \Nm_{X_b/C}^{-1}(\ms L \otimes \det(a_{b*} \ms O_{X_b})^\vee) = \Nm^{-1}_{X_b/C}(\ms L(n(n-1) D/2))$$
since $a_{b*} \ms O_{X_b} \cong \oplus_{i=0}^{n-1} \ms O(-iD)$. In particular, choosing $\ms L := \ms O(-\frac{n(n-1)}{2} D)$, we recover a compactified Prym fibration.

\begin{prop} \label{prop: SL Higgs is good}
    Let $D$ be either $K_C$ or a divisor on $C$ of degree $\ge 2g$. The subset $M$ over the elliptic locus of the moduli space of semistable $D$-twisted $\SL_n$ Higgs bundles of determinant $\ms L$, where $\deg \ms L \equiv \pm \frac{n(n-1)}{2} \deg D$ mod $n$, is a good compactified Prym fibration.
\end{prop}

\begin{proof}
    First of all, note that tensoring by any $\ms M \in \Pic(C)$ defines an isomorphism $\check{M}_{n, \ms L}^D \cong \check{M}_{n, \ms L \otimes \ms M^{\otimes n}}^D$ over the Hitchin base; similarly, dualizing defines an isomorphism $\check{M}_{n, \ms L}^D \cong \check{M}^D_{n, \ms L^\vee}$ over the Hitchin base. Thus, we reduce to the case $\ms L = \ms O(-\frac{n(n-1)}{2} D)$, which as described above is a compactified Prym fibration.
    \\
    \\
    For condition (1), start by recalling that the open subvariety of $\check{M}_{n, \ms L}^D(C)$ of stable Higgs bundles is smooth. (This follows from the $\GL$ result \cite[Proposition 7.4]{Nitsure} by Remark \ref{rmk: etale cover}.) Now note that $M$ is contained in this locus, as the spectral curve of a direct sum of Higgs bundles breaks into a union of the spectral curves of the summands and is thus not integral.
    
    Next, recall from \cite[(29)]{MS} that $M^\gamma \to B_\gamma$ can be described as a $\BZ/m\BZ$-Galois quotient of an open subset of a relative Hitchin fibration $M(\pi) \to A(\pi)$, where $m := \mathrm{ord}_\Gamma(\gamma)$. Note that there is an \'etale cover $M(\pi) \times \J^0(C) \times H^0(C, \ms O(D)) \to M_{n/m, \deg \ms L}^{D'}(C')$ over the isomorphism $A(\pi) \times H^0(C, \ms O(D)) \to A(C')$, where $h_{C'}: M_{n/m, \deg \ms L}^{D'}(C') \to A(C')$ is the $D'$-twisted $\GL_{n/m}$ Hitchin fibration for the \'etale cover $C' \to C$ corresponding to $\gamma$. Since $A(\pi) \cong \BA^{\dim}$ with a linear action of $\BZ/m\BZ$, the quotient is again $\BA^{\dim}$, and in particular the open subset $B_\gamma$ is smooth. Furthermore, since $h_{C'}$ is flat \cite[Proposition 2.4.6]{dC-SL}, we conclude that the same is true of $M(\pi) \to A(\pi)$; in particular, the fiber dimension of its finite quotient $M^\gamma \to B_\gamma$ is constant, and so by miracle flatness, this morphism is flat.
    \\
    \\
    For condition (2), we note first that $M \to B$ is a weak abelian fibration by \cite[Theorem 4.8.1]{dC-SL} (which is stated for $\deg D > 2g-2$, but the proof holds for $D = K_C$). Then, as described in \cite[\S 4.9]{dC-SL}, from the characterization of the endoscopic loci in terms of $\ker(\J^0(C) \to \J^0(\tilde X_b)) = \pi_0(\Prym(X_b))^\vee$ (see Remark \ref{rmk: endoscopic locus defs} and \cite[Theorem 1.1(1)]{HP}), both statements follow from \cite[Proposition 7.2.3]{Ngo-fundamental}.
    \\
    \\
    Condition (3) may be well-known, at least in the case $\deg D \ge 2g$, but for lack of a reference, we provide a proof. First, any $\gamma \in \Gamma \cong (\BZ/n\BZ)^{2g}$ is contained in the subgroup generated by some $\beta \in \Gamma$ of order $n$, thus $M^\beta \subset M^\gamma$ and so $B_\beta \subset B_\gamma$. So without loss of generality one may assume $\gamma$ is of order $n$. Then if $\tilde C \to C$ is the \'etale degree $n$ cover trivializing $\gamma \in \J^0(C)[n]$, and $\tilde D$ the pullback of $D$, recall (e.g. from \cite[\S 1.5, Lemma 3.4]{MS}) that there is a $\BZ/n\BZ$-quotient map $H^0(\tilde C, \ms O_{\tilde C}(\tilde D))_{var} \to B_\gamma$, where $H^0(\tilde C, \ms O_{\tilde C}(\tilde D))_{var}$ is the sum of the nontrivial eigenspaces of the $\Gal(\tilde C/C)$ action on $H^0(\tilde C, \ms O_{\tilde C}(\tilde D))$; on the level of spectral curves, for each $\sigma \in H^0(\tilde C, \ms O_{\tilde C}(\tilde D))_{var}$, the spectral curve over the image is given by descending the union of the spectral curves $V(t-\sigma) \subset |\ms O_{\tilde C}(\tilde D)|$ of the Galois conjugates of $\sigma$. It thus suffices to check that for the general $\sigma \in H^0(\tilde C, \ms O_{\tilde C}(\tilde D))_{var}$, there is no $p \in \tilde C$ such that three distinct Galois conjugates of $\sigma$ take the same value on $p$, and moreover, when two branches $V(t-\sigma)$ and $V(t - x \cdot \sigma)$ meet, they do so transversely.
    % Furthermore, since the set of such $\sigma$ is clearly invariant under translation by $H^0(\tilde C, \ms O_{\tilde C}(\tilde D))_{fix}$, it suffices to prove that the subset of such $\sigma$ is nonempty in $H^0(\tilde C, \ms O_{\tilde C}(\tilde D))$.

    To check that no more than two branches meet at a point (when $n > 2$): trivialize $\ms O_{\tilde C}(\tilde D)$ over an open subset $U \subset \tilde C$, and choose two disjoint elements $x,y \ne 1 \in \Gal(\tilde C/C)$. We define a morphism
    \begin{equation} \label{eq: bpf map}
        \alpha: U \times H^0(\tilde C, \ms O_{\tilde C}(\tilde D))_{var} \to \BA^2, (p, \sigma) \mapsto (\sigma(p) - x \cdot \sigma(p), \sigma(p) - y \cdot \sigma(p)).
    \end{equation}
    From Lemma \ref{lem: alpha rk} we have $\dim \alpha^{-1}(0,0) = h^0(\tilde C, \ms O(\tilde D))_{var} - 1$, thus $\alpha^{-1}(0,0)$ does not dominate $H^0(\tilde C, \ms O(\tilde D))_{var}$, and we can find an open subset $V_{U,x,y} \subset H^0(\tilde C, \ms O(\tilde D))_{var}$ such that for any section $\sigma \in V_{U,x,y}$, elements of the set $\{\sigma, x \cdot \sigma, y \cdot \sigma\}$ do not all take the same value on any point $p \in U$. Taking the intersection over the finitely many choices of $x,y \in \Gal(\tilde C/C)$, as well as over a finite cover $C = \cup_i U_i$, we conclude that for general $\sigma$, the union $\cup_{x \in \Gal(\tilde C/C)} V(t - x \cdot \sigma)$ consists of points where no more than two branches meet.

    It remains to check that for general $\sigma$ and $x \in \Gal(\tilde C/C)$, the intersection $V(t - \sigma) \cap V(t - x \cdot \sigma)$ is transverse, or in other words, the zeroes of $\sigma - x \cdot \sigma$ are simple. Let $G' := \langle x \rangle \subset \Gal(\tilde C/C)$, and consider $\varpi: \tilde C \to \tilde C/G'$. By definition $(1-x)$ acts invertibly on $H^0(\tilde C, \ms O(\tilde D))_{var}$, where now we mean the variant subspace with respect to the action of $G'$ on $\tilde C$ over $\tilde C/G'$. Then it suffices to check that the general element of $H^0(\tilde C, \ms O(\tilde D))_{var}$ has simple zeroes. By Bertini's theorem \cite[Example 5.2.14]{PosinAG}, the general element of
    $H^0(\tilde C, \ms O(\tilde D))_{var}$ has only simple zeroes away from the basepoints, and since the finitely many basepoints are simple by Lemma \ref{lem: bpf}(b), the general element has only simple zeroes on them as well.
\end{proof}

\begin{lem} \label{lem: bpf}
    Fix a divisor $D$ on $C$ and an \'etale $\BZ/m\BZ$-cover $\varpi: C' \to C$. Let $D'$ be the pullback of $D$ to $C'$.\\
    (a) If $\deg D \ge 2g$, then for any character $\chi$ of $\BZ/m\BZ$, the component $H^0(C', \ms O_{C'}(D'))_\chi$ is basepoint free.\\
    (b) If $D = K_C$, then $p \in C'$ is not a basepoint of $H^0(C', \ms O_{C'}(D'))_{var}$ unless $m = 2$ and $C'$ is hyperelliptic with hyperelliptic divisor given by the $\BZ/2\BZ$-orbit of $p$. In this case $p$ is a simple basepoint.\\
    % (c) If $m \ne 1,2,4$, then $C'$ is not hyperelliptic.
\end{lem}

\begin{proof}
    (a) Suppose $p$ is a basepoint of $H^0(C', \ms O_{C'}( D'))_\chi$. Then the same is true of the $\BZ/m\BZ$-orbit $O_p$, so $H^0(C', \ms O(D'))_\chi = H^0(C', \ms O(D' - O_p))_\chi$. Using the identifications of \cite[Remark 1.6]{MS}, this means that $H^0(C, \ms O(D) \otimes L_\chi) = H^0(C, \ms O(D-\varpi(p))\otimes L_\chi)$ for a suitable torsion line bundle $L_\chi \in \J^0(C)$. However, this is impossible by Riemann-Roch (since by the condition on $\deg D$, there are no nontrivial $h^1$ terms).
    \\
    (b) Suppose $p \in C'$ is a basepoint of $H^0(C', \omega_{C'})_{var}$. In this case
    \begin{multline*}
        h^0(C',\omega_{C'}) - h^0(C, \omega_C) = h^0(C', \omega_{C'})_{var} = \\ h^0(C', \omega_{C'}(-O_p))_{var}  = h^0(C', \omega_{C'}(-O_p)) - h^0(C, \omega_C(-\varpi(p)))
    \end{multline*}
    $$\implies 1 = h^0(C, \omega_C) - h^0(C, \omega_C(-\varpi(p))) = h^0(C', \omega_{C'}) - h^0(C', \omega_{C'}(-O_p)).$$
    By Clifford's theorem, this is impossible unless $\deg O_p = 2$ (i.e., $m=2$), $C'$ is hyperelliptic, and $O_p$ is the hyperelliptic divisor.

    Suppose this is the case, and let $\sigma$ be the hyperelliptic involution. Then $\sigma$ commutes with all automorphisms of $C'$, in particular $\Gal(C'/C)$, and so defines an involution $\sigma_C$ of $C$. From the induced map $C'/\sigma = \BP^1 \to C/\sigma_C$ we see that $C/\sigma_C \cong \BP^1$, and $\sigma_C$ is a hyperelliptic involution on $C$. If $O_p$ is a hyperelliptic divisor, then we must have $O_p = p + \sigma(p)$, thus $\varpi(p)$ is a fixed point of $\sigma_C$, and $2\varpi(p)$ is a hyperelliptic divisor on $C$. Then
    $$h^0(C, \omega_C) - h^0(C, \omega_C(-2\varpi(p)) = 1$$
    whereas
    $$h^0(C', \omega_{C'}) - h^0(C', \omega_{C'}(-2O_p)) = 2$$
    and so $h^0(C', \omega_{C'})_{var} > h^0(C', \omega_{C'}(-2O_p))_{var} = h^0(C', \omega_{C'}(-2p))_{var}$, i.e., $p$ is a basepoint of order 1.
    % \\
    % (c) If $C'$ is hyperelliptic, the hyperelliptic involution $\sigma$ commutes with $\Gal(C'/C) = \BZ/m\BZ$. Then the $2g_{\tilde C} + 2 = 2m(g-1)+4$ fixed points of $\sigma$ break into a union of (free) $\BZ/m\BZ$-orbits, thus $4/m \in \BZ$.
\end{proof}

\begin{lem} \label{lem: alpha rk}
    In the situation of Proposition \ref{prop: SL Higgs is good}, if $n \ge 3$, then the morphism $\alpha$ of (\ref{eq: bpf map}) satisfies
    $$\dim \alpha^{-1}(0,0) = h^0(\tilde C, \ms O(\tilde D))_{var} - 1.$$
\end{lem}

\begin{proof}
    We begin with the more straightforward case $\deg D \ge 2g$. We claim that for each $p \in U$, the linear map $\alpha(p,-)$ is of full rank. Indeed, since $\{x,y,1\}$ are distinct, the characters $1-x, 1-y$ of $\Gal(\tilde C/C)^\vee$ are linearly independent, i.e., for any $(\alpha_1, \alpha_2) \ne (0,0) \in \BC^2$ there is some $\chi \in \Gal(\tilde C/C)^\vee$ with $\alpha_1(1-x)(\chi) + \alpha_2(1-y)(\chi) \ne 0$. By Lemma \ref{lem: bpf}(a) there is some $\sigma \in H^0(\tilde C, \ms O(\tilde D))_\chi$ with $\sigma(p) \ne 0$ and so $\alpha(p, \sigma) \notin \langle (-\alpha_2, \alpha_1) \rangle.$ Since this holds for any choice of $(\alpha_1,\alpha_2) \ne (0,0)$, we conclude that $\textrm{im}(\alpha(p,-))$ is not contained in any line, i.e., $\alpha(p,-)$ is of full rank 2. Thus in particular $\alpha^{-1}(0,0) \subset U \times H^0(\tilde C, \ms O(\tilde D))_{var}$ is of codimension 2.

    Now, suppose that $D = K_C$. We claim that $\alpha(p,-)$ is of rank 2 for general $p \in U$ and of rank $\ge 1$ for all $p$; this is still enough to conclude that $\alpha^{-1}(0,0)$ is of codimension 2. For general $p \in U$, the image $\varpi(p) \in C$ is not a basepoint of $\omega_C \otimes L_\chi$ for any $L_\chi \in \J^0(C)[n]$, thus as in \ref{lem: bpf}(a) we see that no $H^0(\tilde C, \ms O(\tilde D))_\chi$ has a basepoint at $p$, and the argument of the previous paragraph shows that $\alpha(p,-)$ is of rank 2. To show that $\alpha(p,-)$ is of rank $\ge 1$, i.e., nonzero, for any $p \in U$: let $G' := \langle x,y \rangle \subset \Gal(\tilde C/C) = \BZ/n\BZ$, and consider $\varpi: \tilde C \to \tilde C/G' $. As $\{1,x,y\}$ are distinct, we have $\deg \varpi \ge 3$, and so by Lemma \ref{lem: bpf}(b), there is $\sigma \in H^0(\tilde C, \omega_{\tilde C})_{var}$ (variant with respect to $G'$) with $\sigma(p) \ne 0$; by definition of $G'$, at least one of $(1-x) \sigma(p), (1-y)\sigma(p) \ne 0$, and so $\alpha(p,\sigma) \ne 0$, as required.
\end{proof}

\subsection{Poincar\'e sheaves} \label{subsec: Poincare}
Classically, given an abelian scheme $A \to B$, there is a dual abelian scheme $A^\vee \to B$ and a Poincar\'e line bundle $\mc P$ on $A \times_B A^\vee$ such that Fourier-Mukai transform by $\mc P$ defines a derived equivalence $D^b\Coh(A) \cong D^b\Coh(A^\vee)$. Less classically, it was known that the Leray (or equivalently, perverse) filtration of $H^*(A, \BC)$ could be recovered from cohomological Fourier-Mukai transform by $\ch(\mc P)$ and $\ch(\mc P^{-1})$. (See \cite[\S 1.2]{MSY} for a more detailed discussion and references.)

Going beyond the smooth case, Arinkin \cite{Arinkin} developed a version of this theory for the compactified Jacobian $\oJ^0(X/B) \to B$ for a family of curves $X/B$ as in \S \ref{subsec: compactified Prym}. In particular, he defined a Poincar\'e sheaf $\mc P_{J/J}$ on $\oJ^0(X/B) \times_B \oJ^0(X/B)$ as the pushforward of a naturally defined (normalized) Poincar\'e line bundle on the locus $\J^0(X/B) \times_B \oJ^0(X/B) \cup \oJ^0(X/B) \times_B \J^0(X/B)$ whose complement is of codimension $\ge 2$. The resulting sheaf is maximal Cohen-Macaulay, flat over each factor $\oJ^0(X/B)$, and symmetric in the two factors; moreover, it defines a Fourier-Mukai autoequivalence of $D^b\Coh(\oJ^0(X/B))$.

The following is similar but not identical to \cite[Theorem 4.7]{GS} (as we stay in the degree 0 case but allow more general compactified Prym fibrations):

\begin{prop} \label{prop: GS}
    Let $X \to C \times B$ be as described at the beginning of \S \ref{subsec: compactified Prym}, and assume that $\oPrym(X/B)$ is smooth.
    The restriction of the Arinkin Poincar\'e sheaf $\mc P_{J/J}$ to $\oJ^0(X/B) \times_B \oPrym(X/B)$ can be uniquely descended to a sheaf $\mc P$, flat over each factor of $[\oJ^0(X/B)/\J^0(C)] \times_B \oPrym(X/B) = M^\vee \times_B M$.
    Taking $\mc P^{-1} := R\mc Hom(\mc P^T, \pi_B^* \ms A)[\dim \pi]$, where $\ms A \in \Pic(B)$ is as defined in Lemma \ref{lem: oPrym is flat}, and $(-)^T$ denotes the pullback under the involution $M \times_B M^\vee \to M^\vee \times_B M$ switching the two factors, we have
    $$\mc P \circ \mc P^{-1} \cong \ms O_{\Delta_M}, \,\, \mc P^{-1} \circ \mc P \cong \ms O_{\Delta_{M^\vee}}.$$
\end{prop}

\begin{proof}
    For the first part, see \cite[\S 4.1]{GS} or Proposition \ref{prop: GS descent}.

    Now, we note that the left and right adjoints to the Fourier-Mukai transform $\Phi_{\mc P^T}: D^b\Coh(M) \to D^b\Coh(M^\vee)$ are Fourier-Mukai transforms with kernels given by the standard formulas $\mc P^T_L := R\mc Hom(\mc P, \pi_2^! \ms O_M)$ and $\mc P_R^T := R \mc Hom(\mc P, \pi_1^! \ms O_{M^\vee})$ in $D^b\Coh(M^\vee \times_B M)$. (This follows from the usual proof, see, e.g., \cite[Theorem 1.1]{Rizzardo}, using the theory of Grothendieck duality on Deligne-Mumford stacks, see \cite{Neeman}.) In particular, from Lemma \ref{lem: oPrym is flat}, we see that $\mc P_L^T \cong \mc P_R^T \cong (\mc P^{-1})^T$.

    From \cite[Proposition 4.5]{GS} we have $(\mc P^{-1})^T \circ \mc P^T \cong \ms O_{\Delta_M}$. Thus $\Phi_{\mc P^T}$ is fully faithful \cite[Corollary 1.23]{Huybrechts-FM}. Without loss of generality, we may assume that $B$ is connected; otherwise apply this argument to each component separately. Then, as described in Lemma \ref{lem: indecomposability}, we have an orthogonal direct sum decomposition of $D^b(M^\vee)$ into indecomposable components $D^b(M^\vee)_{\chi'}$ indexed by characters of the subgroup $\Gamma' \subset \Gamma$ that acts trivially on $M$. Choosing any point $b\in B$, since any $\gamma \in \Gamma'$ acts trivially on $\Prym(X_b)$, we see that $\gamma$ must pull back to the trivial line bundle on $X_b$, thus also the normalization $\tilde X_b$; in particular, from \cite[Theorem 1.1]{HP} we obtain a surjection $\pi_0(\Prym(X_b)) = \ker(\J^0(C) \to \J^0(\tilde X_b))^\vee \twoheadrightarrow (\Gamma')^\vee$. Then for each $\chi' \in (\Gamma')^\vee$, picking a point $x$ in the corresponding component of $\pi_0(\Prym(X_b))$, it follows from Corollary \ref{cor: eigenvalue of P} and Proposition \ref{prop: smooth and singular Poincare} that $\Phi_{\mc P^T}(\ms O_x) \ne 0 \in D^b\Coh(M^\vee)_{\chi'}$. Thus the image of $\Phi_{\mc P^T}$ meets all the subcategories $D^b\Coh(M^\vee)_{\chi'}$. Then, since the left and right adjoints of $\Phi_{\mc P^T}$ agree, it follows from the proof of \cite[Proposition 1.54]{Huybrechts-FM} that $\Phi_{\mc P^T}$ is an equivalence (since $\mc D_1'$ and $\mc D_2'$ give a decomposition of each $D^b\Coh(M^\vee)_{\chi'}$ with $\mc D_1' \cap D^b\Coh(M^\vee)_{\chi'} = D^b\Coh(M^\vee)_{\chi'}$ for each $\chi'$).
    Then the left and right adjoint of $\Phi_{\mc P^T}$ is the inverse, and so we have
    $$(\mc P^{-1})^T \circ \mc P^T \cong \ms O_{\Delta_M}, \,\,\, \mc P^T \circ (\mc P^{-1})^T \cong \ms O_{\Delta_{M^\vee}}$$
    by Lemma \ref{lem: FM id}, giving the transposes of the needed identities.
\end{proof}

\begin{lem} \label{lem: oPrym is flat}
    The fibrations $\pi: M \to B$ and $\pi^\vee: M^\vee \to B$ are syntomic (so in particular, flat and Gorenstein), and the relative canonical bundles $\omega_\pi$ and $\omega_{\pi^\vee}$ are the pullbacks of the same line bundle $\ms A := \det \pi_{X/B*} \omega_{X/B} \in \Pic(B)$ under $\pi$ and $\pi^\vee$, respectively.
\end{lem}

\begin{proof}
    By \cite[Theorem 9]{AIK} the compactified Jacobian fibration $\pi_J: \oJ^0(X/B) \to B$ is flat, and the fibers are local complete intersections; in other words $\pi_J$ is syntomic \cite[Tag 01UF]{stacks-project}. Recall from Remark \ref{rmk: etale cover} the \'etale map
    $\J^0(C) \times \oPrym(X/B) \to \oJ^0(X/B)$ over $B$. It follows that $M = \oPrym(X/B)$ and $M^\vee = [M/\Gamma]$ are syntomic over $B$.

    As for the relative canonical bundles, recall from \cite[Corollary 9]{Arinkin-cohomology} that $\omega_{\oJ^0(X/B)/B}$ is pulled back from $\ms A \in \Pic(B)$. Using the \'etale map above and triviality of $\omega_{\J^0(C)}$, we see that the same is true of $\omega_\pi$. To check this for $\pi^\vee$, choose any point $b \in B$. Note that any element in the kernel of the pullback map $c_b^*: \Gamma \to \J^0(X_b)$ acts trivially on $\oPrym(X_b)$, hence on $\omega_{\oPrym(X_b)}$. Thus the $\Gamma$-equivariant structure is determined by the $\Gamma/\ker(c^*)$-equivariant structure on $\omega_{\oPrym(X_b)} \cong \ms O_{\oPrym(X_b)}$. We note that restricting to the locus of line bundles $\Prym(X_b)$, the quotient $[\Prym(X_b)/(\Gamma/\ker(c^*))]$ is an algebraic group, therefore has trivial canonical bundle. Since $\J^0(X_b) \subset \oJ^0(X_b)$ is a dense open subset by \cite[Theorem 9]{AIK}, via the \'etale map above the same is true for $\Prym(X_b) \subset \oPrym(X_b)$.
    Then since the $\Gamma$-equivariant structure on $\ms O_{\oPrym(X_b)}$ for $\omega_{[\oPrym(X_b)/\Gamma]}$ agrees with that for $\ms O_{[\oPrym(X_b)/\Gamma]}$ on a dense open subset, we conclude that the two agree, i.e., $\omega_{[\oPrym(X_b)/\Gamma]} \cong \ms O_{[\oPrym(X_b)/\Gamma]}$. Returning to the family over $B$, this suffices to show that the descent data for $(\pi^\vee)^* \ms A$ and $\omega_{\pi^\vee}$ are the same.
\end{proof}

\begin{lem} \label{lem: indecomposability}
    If $B$ is connected, and $\Gamma' \subset \Gamma$ is the subgroup acting trivially on $M$, then $D^b\Coh(M^\vee)$ splits as an orthogonal direct sum of indecomposable categories $D^b\Coh(M^\vee)_{\chi'}$ running over the characters $\chi'$ of $\Gamma'$, where $D^b\Coh(M^\vee)_{\chi'}$ is the bounded derived category of coherent sheaves on which $\Gamma'$ acts by $\chi'$.
\end{lem}

\begin{proof}
    Clearly there is such a direct sum decomposition, so the only point to check is indecomposability of the factors.
    This follows from the proof of \cite[Lemma 4.2]{BKR} applied to each $D^b\Coh(M^\vee)_{\chi'}$, letting $D$ be a free $\Gamma/\Gamma'$-orbit, replacing $\ms O_D$ with $\ms O_D \otimes \chi$ for any character $\chi$ of $\Gamma$ extending $\chi'$, and letting $\rho_i$ run over the set of such characters. The one additional point to note is that $\ms O_M \otimes \rho_i$ is an indecomposable $\Gamma$-sheaf because $\Gamma$ acts transitively on the set of connected components of $M$: to see this, since $B$ is connected, it suffices to note for each $b \in B$ that $\Gamma$ acts transitively on the set of connected components of $\Prym(X_b)$ \cite[Theorem 1.1(2)]{HP}, which is a dense open subset of $\oPrym(X_b)$ by \cite[Theorem 9]{AIK} and Remark \ref{rmk: etale cover}.
\end{proof}

The following is probably well-known, but for lack of a reference in this generality,
% though see \cite[Lemma 2.11]{BK-FM} for the relative non-stacky case and \cite{Kawamata} for the non-relative orbifold case)
we include a short proof:

\begin{lem} \label{lem: FM id}
    Let $\mc X$ be a smooth nice quotient, flat and proper over a smooth variety $B$. If Fourier-Mukai transform $\Phi_{\mc F}$ by $\mc F \in D^b\Coh(\mc X \times_B \mc X)$ induces the identity morphism on $D^b\Coh(\mc X)$, then $\mc F \cong \ms O_{\Delta_{\mc X/B}}.$
\end{lem}

\begin{proof}
    Let $\mc X = [X/G]$ be a nice presentation, with quotient map $\alpha: X \to \mc X$. Then $\mc G := (\alpha \times \alpha)^* \mc F$ induces the Fourier-Mukai transform
    $$\Phi_{\mc G} = \alpha^* \circ \alpha_* = \Phi_{(\alpha \times \alpha)^* \Delta_{\mc X *} \ms O_{\mc X}}$$ on $D^b\Coh(X)$. 
    Since $X$ is quasiprojective and $\alpha^* \circ \alpha_*$ preserves $\Coh(X)$, it follows from \cite[Theorem 5.7]{Genovese} that $\mc G \cong (\alpha \times \alpha)^* \Delta_{\mc X*} \ms O_{\mc X}$. Then by descent $\mc F = \Delta_{\mc X*} \ms L$ for some $\ms L \in \Pic(\mc X)$, and $\Phi_{\mc F}(\ms O_{\mc X}) = \ms O_{\mc X}$ implies $\ms L \cong \ms O_{\mc X}$.
\end{proof}

\subsection{Statement of the main theorem} \label{subsec: main theorem}
We are now in a position to define the projectors that will recover the perverse filtration of a good compactified Prym fibration, so that we can state a sharpened form of Theorem \ref{thm: main intro}.

Let $M \to B$ be a good compactified Prym fibration, and let $\mc P$ (and $\mc P^{-1}$) be the Arinkin Poincar\'e sheaf (and its inverse) described in Proposition \ref{prop: GS}.

\begin{defn}
    Define $b_\gamma := \dim B_\gamma$, $h_\gamma := \dim M^\gamma - \dim B_\gamma$, and $d_\gamma := h-h_\gamma$. We also write $b_1 =: b$ and $h_1 =: h$.
\end{defn}

\begin{rmk} \label{rmk: h-hgamma}
    We will show in Corollary \ref{cor: h-hgamma} that $h-h_\gamma = b-b_\gamma$. We use this fact freely throughout the paper.
\end{rmk}

\begin{rmk} \label{rmk: fermionic shift}
    When $M$ is an open subset of a moduli space of ($D=K_C$) $\SL_n$-Higgs bundles (see Proposition \ref{prop: SL Higgs is good}), then $d_\gamma$ is the ``fermionic shift" of \cite{HT}, as noted in \cite[Remark 0.6]{MS}. As in the formula for the stringy mixed Hodge polynomial in \cite{HT}, a class of codimension $i$ on $M^{\gamma,\vee}$ should be considered to be of codimension $i + d_\gamma$ on the inertia stack $IM^\vee$, even though in fact $\dim IM^\vee - \dim M^{\gamma,\vee} = \dim M - \dim M^\gamma = 2d_\gamma$.
\end{rmk}

Recall the notations $IB, \iota_B$ from Definition \ref{defn: Gamma}.
We will use $\pi_{IB}$ to denote all natural morphisms to $IB$, including in particular $IM^\vee \times_B M, M \times_B IM^\vee \to IM^\vee \to IB$. Recall also the stacky functors $\td^I$ and $\tau$ from \S \ref{subsubsec: stacky tau}.

\begin{defn} \label{def: G and p}
    Define Chow correspondences
    $$\mf G := \td^I(\pi_{IB}^* T_{\iota_B}-T_{IM^\vee \times_B M}) \cap \tau(\mc P) \in \Corr^*_B(IM^\vee, M)$$
    $$\mf G^{-1} := \td^I(-\pi_{IB}^* T_{IB} +T_{\iota_{M \times_B M^\vee}}) \cap \tau(\mc P^{-1}) \in \Corr^*_B(M, IM^\vee).$$
    We can decompose $\mf G$ and $\mf G^{-1}$ by degree and component of $IM^\vee$
    $$\mf G = \sum_{\gamma \in \Gamma} \sum_{i \ge 0} \mf G_{\gamma,i}, \,\, \mf G_{\gamma,i} \in \CH_{\dim M^\gamma \times_B M + d_\gamma - i}(M^{\gamma, \vee} \times_B M) = \Corr_B^{i-h_\gamma}(M^{\gamma, \vee}, M)$$
    $$\mf G^{-1} = \sum_{\gamma \in \Gamma} \sum_{i \ge 0} \mf G_{\gamma,i}^{-1}, \,\, \mf G_{\gamma,i}^{-1} \in \CH_{\dim M \times_B M^\gamma + d_\gamma - i}(M \times_B M^{\gamma, \vee}) = \Corr_B^{i-2h + h_\gamma}(M, M^{\gamma, \vee}).$$
    Furthermore, we define correspondences
    \begin{equation} \label{eq: pGgk def}
    \mf p_{\gamma,k}^{\mf G} := \sum_{i \le k} \mf G_{\gamma, i} \circ \mf G_{\gamma, 2h-i}^{-1} ,\,\,\,\,\,\,\,\, \mf p^{\mf G}_{(k_\gamma)_\gamma} := \sum_\gamma \mf p^{\mf G}_{\gamma, k_\gamma}, \,\,\,\,\,\,\,\, \mf p^{\mf G}_k := \mf p^{\mf G}_{(k_\gamma=k)_\gamma}
    \end{equation}
    and
    \begin{equation} \label{eq: qGgk def}
     \mf q_{\gamma,k}^{\mf G} := \sum_{i \ge k} \mf G_{\gamma, i} \circ \mf G_{\gamma, 2h-i}^{-1} ,\,\,\,\,\,\,\,\, \mf q^{\mf G}_{(k_\gamma)_\gamma } := \sum_\gamma \mf q^{\mf G}_{\gamma, k_\gamma}, \,\,\,\,\,\,\,\, \mf q^{\mf G}_k := \mf q^{\mf G}_{(k_\gamma = k)_\gamma}  
    \end{equation}
    in $\Corr^0_B(M,M)$.
\end{defn}

The reason for the complicated Todd class convention is, first, to ensure that $\mf G$ and $\mf G^{-1}$ are inverses and, second, to match $\mf G$ with the (non-Todd-twisted) Chern class of the Poincar\'e bundle of a corresponding smooth Prym fibration in \S \ref{subsec: proof of realization}.

\begin{thm} \label{thm: main sec 2}
    Let $M \to B$ be a good compactified Prym fibration.\\
    (i) For any tuple $(k_\gamma) \in \BN^\Gamma$, the correspondences $\pg_{(k_\gamma)_\gamma}$ and $\qg_{(k_\gamma)_\gamma}$ are projectors.
    Moreover, $\pg_{(k_\gamma)_\gamma}$ and $\qg_{(k_\gamma+1)_\gamma}$ are orthogonal and sum to $[\Delta_{M}]$.\\
    (ii) Fix $\gamma \in \Gamma$, with corresponding character $\kappa \in \Gamma^\vee$, as well as $k \in \BN$. By (i) we may define Chow motives $P_{\gamma,k}h(M) := (M, \pg_{\gamma,k},0)$ and $Q_{\gamma,k}h(M) := (M, \qg_{\gamma,k},0)$, equipped with the preferred morphisms to and from $h(M)$. The homological realization of $P_{\gamma, k}h(M)\to h(M)$ is given by the inclusion of the $\kappa$ component and truncation morphism
    $$(\ptau_{\le k + b} \pi_* \BQ_M)_\kappa \to \ptau_{\le k + b} \pi_* \BQ_M \to \pi_* \BQ_M.$$
    Similarly, the homological realization of $h(M) \to Q_{\gamma, k}h(M)$ is given by 
    $$\pi_* \BQ_M \to \ptau_{\ge k+b} \pi_* \BQ_M \to (\ptau_{\ge k+b} \pi_* \BQ_M)_\kappa.$$
    (iii) For any degrees $k,l \in \BN$, the composition
    $$\qg_{k+l+1} \circ [\Delta_{M}^{s}] \circ (\pg_k \times \pg_l) = 0.$$
    In particular, after taking the homological realization, we see that the perverse filtration $P_\bullet H^*(M)$ is multiplicative.
\end{thm}

\subsection{Proof of Theorem \ref{thm: main intro}(a)} \label{subsec: proof of (a)}
We start with the following basic relation:

\begin{prop} \label{prop: G and G-1 inverse}
    Given a smooth (but not necessarily good) compactified Prym fibration such that $IB$ is smooth and $IM^\vee \to IB$ is flat, we have
    $$\mf G \circ \mf G^{-1} = [\Delta_M] \in \CH_*(M \times_B M)$$
    $$\mf G^{-1} \circ \mf G = [\Delta_{IM^\vee}] \in \CH_*(IM^\vee \times_B IM^\vee).$$
\end{prop}

\begin{proof}
    We start with the first identity: by definition
    $$\mf G \circ \mf G^{-1} = Ip_{13*}(\Delta_{IM^\vee}^!(\mf G^{-1} \times \mf G))$$
    $$= Ip_{13*}(\td^I(\pi_{IB}^*(T_{\iota_B} - T_{IB}) + p_{12}^* T_{\iota_{M \times_B M^\vee}} - p_{23}^* T_{IM^\vee \times_B M}) \cap \Delta_{IM^\vee}^!(\tau(\mc P) \times \tau(\mc P^{-1})))$$
    $$= Ip_{13*}(\td^I(\pi_{IB}^*(T_{\iota_B} - T_{IB})
    + p_2^* T_{\iota_{M^\vee}} - p_2^* T_{IM^\vee} + \pi_B^* T_B - p_3^* T_M) \cap \delta_{IM^\vee}^!(\tau(\mc P) \times \tau(\mc P^{-1})))$$
    (where to compute $T_{IM^\vee \times_B M}$ we used that $IM^\vee \to IB$ is flat, and that $IM^\vee \times_B M = (IM^\vee \times M) \times_{(IB \times B)} IB$)
    \begin{equation*}
        \begin{aligned}[c]
        = Ip_{13*}(\td^I(-p_2^* \iota_{M^\vee}^*T_{M^\vee} - p_3^* T_M) \cap \Delta_{IM^\vee}^!(\tau(\mc P) \times \tau(\mc P^{-1}))) \\
        = \td(-p_2^*T_M) \cap Ip_{13*}(\ttd(-p_2^* T_{M^\vee}) \cap \Delta_{IM^\vee}^!(\tau(\mc P) \times \tau(\mc P^{-1}))) \\
        = \td(-p_2^*T_M) \cap Ip_{13*}(\ttd(-p_2^* T_{M^\vee}) \cap \Delta_{IM^\vee}^!(\tau(\mc P \boxtimes \mc P^{-1}))) \\
        = \td(-p_2^*T_M) \cap Ip_{13*}(\tau( \delta_{IM^\vee}^*(\mc P \boxtimes \mc P^{-1})))\\
        = \td(-p_2^*T_M) \cap \tau(p_{13*}( \delta_{IM^\vee}^*(\mc P \boxtimes \mc P^{-1})))\\
        = \td(-p_2^* T_M) \cap \tau(\mc P^{-1} \circ \mc P)\\
        = \td(-p_2^* T_M) \cap \tau(\ms O_{\Delta_M})
        \end{aligned}
        \qquad
        \begin{aligned}[c]
        \hphantom{x}\\
            \textrm{(Proposition \ref{prop: ttd properties}(b))}\\
            \textrm{(Theorem \ref{thm: tau properties}(b))}\\
            \textrm{(Corollary \ref{cor: part e})}\\
            \textrm{(Theorem \ref{thm: tau properties}(c))}\\
            \hphantom{x}\\
            \textrm{(Proposition \ref{prop: GS})}
        \end{aligned}
    \end{equation*}
    $$= \td(-p_2^* T_M) \cap \Delta_{M*} \tau(\ms O_M) = \td(-p_2^* T_M) \cap \Delta_{M*} (\td(T_M) \cap [M])$$
    $$= \Delta_{M*}(\td(-T_M+T_M) \cap [M]) = [\Delta_M].$$
\\
    The proof of the second identity is similar, though not identical: we have
    $$\mf G^{-1} \circ \mf G = Ip_{13*}(\Delta^!_M(\mf G \times \mf G^{-1}))$$
    $$= Ip_{13*}(\td^I(\pi_{IB}^*(T_{\iota_B} - T_{IB}) - p_{12}^* T_{IM^\vee \times_B M} + p_{23}^* T_{\iota_{M \times_B M^\vee}}) \cap \Delta_M^!(\tau(\mc P) \times \tau(\mc P^{-1})))$$
    $$ = Ip_{13*}(\td^I(-p_1^* T_{IM^\vee} - p_2^* T_M + p_3^* T_{\iota_{M^\vee}}) \cap \Delta_M^!(\tau(\mc P) \times \tau(\mc P^{-1})))$$
    $$ = \td^I(-p_1^* T_{IM^\vee} + p_2^* T_{\iota_{M^\vee}}) \cap  Ip_{13*}(\td^I(- p_2^* T_M) \cap \Delta_M^!(\tau(\mc P) \times \tau(\mc P^{-1})))$$
    $$ = \td^I(-p_1^* T_{IM^\vee} + p_2^* T_{\iota_{M^\vee}}) \cap  Ip_{13*}(\td^I(- p_2^* T_M) \cap \Delta_M^!(\tau(\mc P \boxtimes \mc P^{-1})))$$
    $$ = \td^I(-p_1^* T_{IM^\vee} + p_2^* T_{\iota_{M^\vee}}) \cap  Ip_{13*}(\tau(\delta_M^*(\mc P \boxtimes \mc P^{-1})))$$
    $$ = \td^I(-p_1^* T_{IM^\vee} + p_2^* T_{\iota_{M^\vee}}) \cap  \tau(p_{13*}(\delta_M^*(\mc P \boxtimes \mc P^{-1})))$$
    $$ = \td^I(-p_1^* T_{IM^\vee} + p_2^* T_{\iota_{M^\vee}}) \cap  \tau(\mc P^{-1} \circ \mc P)$$
    $$ = \td^I(-p_1^* T_{IM^\vee} + p_2^* T_{\iota_{M^\vee}}) \cap  \tau(\ms O_{\Delta_{M^\vee}}) = \td^I(-p_1^* T_{IM^\vee} + p_2^* T_{\iota_{M^\vee}}) \cap  I\Delta_{M^\vee *}\tau(\ms O_{M^\vee})$$
    $$ = \td^I(-p_1^* T_{IM^\vee} + p_2^* T_{\iota_{M^\vee}}) \cap  I\Delta_{M^\vee *}(\ttd(T_{M^\vee}) \cap [IM^\vee]) \qquad \textrm{(Theorem \ref{thm: tau properties}(d))}$$
    $$ = I\Delta_{M^\vee *}(\td^I(-T_{IM^\vee} + T_{\iota_{M^\vee}} + \iota_{M^\vee}^*T_{M^\vee}) \cap [IM^\vee]) = [\Delta_{IM^\vee}].$$
\end{proof}

The following is then an easy corollary:
\begin{proof}[Proof of Theorem \ref{thm: main intro}(a)]
    Consider the morphisms $\pi_* \BC_M \to \oplus \pi_*^\vee \BC_{M^\vee}[i]$ and $\pi_*^\vee \BC_{M^\vee} \to \oplus \pi_* \BC_M[i]$ induced by $cl(\mf G^{-1})$ and $cl(\mf G^{-1})$, as described in \S \ref{subsubsec: constructible sheaves}. By Proposition \ref{prop: G and G-1 inverse} and the compatibilities of the cycle class map, composing these morphisms in either direction yields the identity. In particular, after passing to global cohomology we obtain inverse isomorphisms $H^*(M, \BC) \leftrightarrow H^*(IM^\vee, \BC)$ induced by multi-degree algebraic correspondences. These may not be isomorphisms of mixed Hodge structures (and in fact, e.g., in the setting of \cite{HT} there is no such isomorphism without the proper degree/weight convention for the various components $H^*(M^{\gamma, \vee}, \BC) \subset H^*(IM^\vee, \BC)$, see Remark \ref{rmk: fermionic shift}). However, since each fixed-degree component $H^*(M, \BC) \to H^{*+k}(IM^\vee, \BC)$ (and vice versa) gives a morphism of mixed Hodge structures $H^*(M, \BC) \to H^{*+k}(IM^\vee, \BC)(k)$, we obtain an isomorphism of mixed Hodge structures
    $$\oplus_k H^k(M, \BC)(k) \cong \oplus_k H^k(IM^\vee, \BC)(k)$$
    (where for the inertia stack it is equivalent to use the standard topological notion of degree or the fermionic shift, which comes with a corresponding shift in weight and is thus canceled out by the Tate twist).
\end{proof}

\begin{rmk}
    A similar result is proved in \cite[Theorem A]{Popa}, starting with an arbitrary derived equivalence of smooth orbifolds with projective coarse spaces. (The difference for us is that $M, M^\vee$ are only projective over $B$, over which they are not smooth, necessitating the result on compatibility of $\tau$ with pullback in the singular case in Theorem \ref{thm: tau properties}(e).) As discussed in \cite[\S 3]{Popa}, this result should be understood as an example of the derived invariance of (orbifold) Hochschild homology, rather than part of the \cite{MSY} package.
\end{rmk}

\begin{rmk} \label{rmk: not reg}
    Note that the morphism $\delta_M: M^\vee \times_B M \times_B M^\vee \hookrightarrow M^\vee \times_B M \times M \times_B M^\vee$ used in Proposition \ref{prop: G and G-1 inverse} is a representable regular embedding such that the corresponding map of inertia stacks $I\delta_M$ is \textit{not} a regular embedding. Indeed, given $\gamma \in \Gamma$ such that $b-b_\gamma= \delta(B_\gamma) \ne 0$, we note that $(I\delta_M)_{(\gamma,\gamma)}$ is a flat base change of the inclusion
    $$M|_{B_\gamma} \hookrightarrow M|_{B_\gamma} \times M|_{B_\gamma}.$$
    Here $M|_{B_\gamma}$ is the compactified Jacobian of a family of curves whose general member is (nontrivially) nodal, thus singular (see \S \ref{subsubsec: FHHO setup} for a description of the normalization of the compactified Jacobian of a nodal curve), and so the diagonal map is not a regular embedding. 
\end{rmk}

\section{Homological realization} \label{sec: homological realization}
The goal of this section is to show that the projectors $\pg_{\gamma,k}, \qg_{\gamma,k}$ recover the perverse filtration on $(\pi_* \BC_M)_\kappa$ as described in Theorem \ref{thm: main sec 2}(ii). The smooth (but possibly disconnected) case is treated in \S \ref{subsec: hom realization smooth}. Following a comparison of Poincar\'e sheaves on the compactified Prym varieties of smooth and nodal curves in \S \ref{subsec: FHHO Pryms in families}, we deduce the general case from the smooth case in \S \ref{subsec: proof of realization}.

\subsection{The smooth case} \label{subsec: hom realization smooth}
Suppose $Y \to U$ is a family of \textit{smooth} projective curves over a smooth base $U$, equipped with a finite degree $m$ morphism to the smooth base curve $C$. Then $\J^0(Y/U) \to U$ is an abelian scheme, and the relative Prym variety $\Prym(Y/U) \to U$ is an extension of the abelian scheme $\Prym(Y/U)^\circ$ (given fiberwise by the component of $\Prym(Y/U)$ containing the trivial line bundle) by the finite abelian group scheme $\pi_0(\Prym(Y/U))$. The goal of this section is to understand the cohomological Fourier-Mukai transform by the Poincar\'e sheaf of $\Prym(Y/U)$ and its dual $[\Prym(Y/U)/\J^0(C)[m] \times U]$ in this smooth but stacky case.

\subsubsection{Duality for abelian stacks} \label{subsubsec: duality for ab stacks}
We start by recalling the theory of duality between extensions of abelian schemes by finite abelian groups and finite abelian group gerbes on abelian schemes, as described in \cite{Brochard}.

First, a \textit{commutative group stack} over $U$ is defined to be a stack $\mc X$ over $U$ with an addition morphism $\mu: \mc X \times_U \mc X \to \mc X$, together with associativity and commutativity morphisms satisfying certain compatibilities \cite[Definition 2.2]{Brochard}. The homotopy category $\mathrm{Ho}(CGS(U))$ of the 2-category of commutative group stacks over $U$ (i.e., objects are as before, and morphisms are isomorphism classes of morphisms in the 2-category) has a concrete description as follows: let $D^{[-1,0]}(U, \BZ)$ be the full subcategory of the derived category of sheaves of abelian groups on $U$ with cohomology in concentrated in degrees $-1$ and 0. Then there is a functor
$$\alpha: D^{[-1,0]}(U, \BZ) \to \mathrm{Ho}(CGS(U)), [G^{-1} \to G^0] \mapsto [G^0/G^{-1}]$$
inducing an equivalence of categories \cite[Theorem 2.5(1)]{Brochard}. For example, given an abelian group $A$ over $U$, we see that $A$ is the commutative group stack $\alpha([0 \to A])$ and $BA$ the commutative group stack $\alpha([A \to 0])$.
We define the \textit{dual} of a commutative group stack to be the stack of morphisms of commutative group stacks $D(\mc X) := \Hom(\mc X, B\BG_m).$ (In this section, all abelian groups and gerbes, e.g. $B\BG_m$ and $\J^0(C)$, are understood to be over $U$.)

\begin{example} \cite[Example 3.7]{Brochard}
    The dual of a constant finite abelian group $G$ over $U$ is $D(G) = BG^\vee$, where $G^\vee = \Hom_{\mathrm{Ab}}(G, \BC^*)$ is the standard dual group. The dual of an abelian scheme $A$ over $U$ is the dual abelian scheme $A^\vee$.
\end{example}

\begin{prop} \label{prop: Prym has correct dual}
    (i) The dual of the commutative group stack $\Prym(Y/U)$ is $[\J^0(Y/U)/\J^0(C)]$, and vice versa. \\
    (ii) The dual of the abelian scheme $\Prym(Y/U)^\circ$ is $[\J^0(Y/U)/(\J^0(C)/K)]$ for $K := \ker(\J^0(C) \to \J^0(Y/U)$.
\end{prop}

\begin{proof}
    Consider the ``exact sequence" (see \cite[Definition 2.11]{Brochard}) of commutative group stacks
    \begin{equation} \label{eq: duality 1}
        0 \to \J^0(C) \to \J^0(Y/U) \to [\J^0(Y/U)/\J^0(C)] \to 0.
    \end{equation}
    Applying the duality functor, we obtain a short exact sequence
    \begin{equation} \label{eq: duality 2}
            0 \to D([\J^0(Y/U)/\J^0(C)]) \to \J^0(Y/U) \to \J^0(C) \to 0
    \end{equation}
    by \cite[Example 3.7]{Brochard} (since Jacobians are self-dual abelian varieties) and \cite[Proposition 3.18.a(iii)]{Brochard} (since the norm map $\J^0(Y/U) \to \J^0(C)$ is surjective). Thus $D([\J^0(Y/U)/\J^0(C)]) = \ker(\J^0(Y/U) \to \J^0(C)) = \Prym(Y/U)$.

    Next, we claim that $K := \ker(\J^0(C) \to \J^0(Y/U))$ is a constant group scheme: indeed, by the projection formula it is a subobject of the constant finite group scheme $\J^0(C)[m]$, and on the other hand it is (analytically) locally constant since $\pi_0(\Prym(Y/U))$ is so by Ehresmann's theorem, and we have a canonical identification $\ker(\J^0(C) \to \J^0(Y/U))^\vee = \pi_0(\Prym(Y/U))$ by \cite[Theorem 1.1(1)]{HP}. Note also that $[\J^0(Y/U)/(\J^0(C)/K)]$ is an abelian scheme. Then 
    $$D(\Prym(Y/U)) = D(D([\J^0(Y/U)/\J^0(C)])) = [\J^0(Y/U)/\J^0(C)]$$
    by \cite[Example 4.10]{Brochard}.

    Finally, it follows from \cite[Theorem 4.11]{Brochard} applied to $\Prym(Y/U)$ that
    $$(\Prym(Y/U)^\circ)^\vee = H^0([\J^0(C) \to \J^0(Y/U)]) = [\J^0(Y/U)/(\J^0(C)/K)].$$
\end{proof}

Now, for any commutative group stack $\mc X$, we define the Poincar\'e bundle $\mc P_{\mc X/D\mc X}$ to be the pullback of the universal line bundle on $B\BG_m$ under the evaluation map $\mc X \times_U D(\mc X) \to B\BG_m$. Note that this recovers the usual definition of normalized Poincar\'e bundle when $\mc X$ is an abelian scheme (where normalized means trivial on $\{id\} \times_U A^\vee$ and $A \times_U \{id\}$).

\begin{lem} \label{lem: compare smooth Poincare sheaves}
    Let $f: \mc X \to \mc Y$ be a morphism of commutative group stacks. Then
    $$(1 \times Df)^* \mc P_{\mc X/D\mc X} = (f \times 1)^* \mc P_{\mc Y/D \mc Y}.$$
\end{lem}

\begin{proof}
    This follows from the statement that the following (non-Cartesian) diagram of commutative group stacks commutes
    \[
    \begin{tikzcd}
        \mc X \times_U D(\mc Y) \arrow[r, "f \times 1"] \arrow[d, "1 \times Df"] & \mc Y \times_U D(\mc Y) \arrow[d, "ev_{\mc Y}"] \\
        \mc X \times_U D(\mc X) \arrow[r, "ev_{\mc X}"] & B \BG_m
    \end{tikzcd}
    \]
    which is a tautology from the definition of $Df$.
\end{proof}

\begin{cor} \label{cor: multiplicativity of P}
    If $\mc X$ is a commutative group stack with addition map $\mu: \mc X \times_U \mc X \to \mc X$, then the Poincar\'e sheaf satisfies
    $$(\mu \times 1)^* \mc P_{\mc X/D\mc X} \cong p_{13}^* \mc P_{\mc X/D\mc X} \otimes p_{23}^* \mc P_{\mc X/D\mc X}.$$
    If the natural map $\mc X \to D(D(\mc X))$ is an isomorphism, then also
    $$(1 \times \mu)^* \mc P_{D\mc X/\mc X} \cong p_{12}^* \mc P_{D\mc X/\mc X} \otimes p_{13}^* \mc P_{D \mc X/\mc X}.$$
\end{cor}

\begin{proof}
    We apply Lemma \ref{lem: compare smooth Poincare sheaves} to $\mu$. Note that applying $D$ to the split exact sequence
    $$0 \to \mc X \to \mc X \times_U \mc X \to \mc X \to 0 = \alpha(0 \to \alpha^{-1}(\mc X) \to \alpha^{-1}(\mc X)^{\oplus 2} \to \alpha^{-1}(\mc X) \to 0),$$
    we obtain another split exact sequence, so $D(\mc X \times_U \mc X) \cong D(\mc X) \times_U D(\mc X)$. We then have $ev_{\mc X \times \mc X} \cong ev_{\mc X} \times_U ev_{\mc X}$, so
    $$\mc P_{\mc X \times \mc X/D(\mc X \times \mc X)} \cong p_{13}^* \mc P_{\mc X/D\mc X} \otimes p_{24}^* \mc P_{\mc X/D\mc X}.$$
    Since the two zero-section maps $\mc X \to \mc X \times_U \mc X$ compose with $\mu$ to give the identity, we see that $D\mu = \Delta_{D(\mc X)}: D(\mc X) \to D(\mc X \times_U \mc X)$ and thus
    $$(\mu \times 1)^* \mc P_{\mc X/D\mc X} \cong  (1 \times \Delta_{D(\mc X)})^* (p_{13}^* \mc P_{\mc X/D\mc X} \otimes p_{24}^* \mc P_{\mc X/D\mc X}) = p_{13}^* \mc P_{\mc X/D\mc X} \otimes p_{23}^* \mc P_{\mc X/D\mc X}.$$

    For the second statement, one applies Lemma \ref{lem: compare smooth Poincare sheaves} to $\Delta_{D(\mc X)} = D\mu$, noting that
    $$(1 \times DD\mu)^* \mc P_{D\mc X/DD\mc X} = (1 \times \mu)^* \mc P_{D \mc X/\mc X}$$
    and
    $$(\Delta_{D(\mc X)} \times 1)^* \mc P_{D\mc X \times D\mc X/ DD\mc X \times DD\mc X}  = (\Delta_{D(\mc X)} \times 1)^* \mc P_{D\mc X \times D\mc X/ \mc X \times \mc X} $$
    $$\cong (\Delta_{D \mc X} \times 1)^*(p_{13}^*\mc P_{D\mc X/\mc X} \otimes p_{24}^* \mc P_{D \mc X/\mc X}) = p_{12}^* \mc P_{D \mc X/\mc X} \otimes p_{13}^* \mc P_{D \mc X/\mc X}.$$
    
\end{proof}

\begin{cor} \label{cor: smooth Prym P/P vs J/J}
    The Poincar\'e bundle $\mc P_{D(\Prym)/\Prym}$ is obtained from $\mc P_{J/J}$ on $\J^0(Y/U) \times_U \J^0(Y/U)$ as in Proposition \ref{prop: GS}, by restricting to $\Prym(Y/U)$ in the second factor and descending to $[\J^0(Y/U)/\J^0(C)]$ in the first.
\end{cor}

\begin{proof}
    This follows directly from Lemma \ref{lem: compare smooth Poincare sheaves} applied to $i: \Prym(Y/U) \hookrightarrow \J^0(Y/U)$. Indeed, since (\ref{eq: duality 1}) and (\ref{eq: duality 2}) were shown to be dual in the proof of Proposition \ref{prop: G and G-1 inverse}, we see that $Di$ is the quotient $\J^0(Y/U) \to [\J^0(Y/U)/\J^0(C)]$.
\end{proof}

\begin{cor} \label{cor: eigenvalue of P}
    An element $\gamma \in K := \ker(\J^0(C) \to \J^0(Y/U))$ acts on $\mc P_{D(\Prym)/\Prym}$ such that the eigenvalue on $D(\Prym(Y/U)) \times_U \Prym^\alpha(Y/U)$ for any $\alpha \in \pi_0(\Prym(Y/U)) = K^\vee$ is given by $\alpha(\gamma) \in \BC^*$.
\end{cor}

\begin{proof}
    First, note that $D(BK)$ can be described as follows \cite[Corollary 3.5]{Brochard}: isomorphism classes of morphisms $BK \to B\BG_m$ are given by morphisms $[K \to 0] \to [\BG_m \to 0]$ in $D^{[-1,0]}(U, \BZ)$, i.e., $\Hom_{\mathrm{Ab}}(K, \BG_m) = K^\vee$. It follows from this description of  $D(BK)$ and the evaluation map that for each $\alpha \in K^\vee$, $K$ acts on $\mc P_{BK/K^\vee}|_{BK \times \alpha}$ by the character $\alpha$.

    Next, note that the morphism of complexes $[K \to 0] \to [\J^0(C) \to \J^0(Y/U)]$ gives a morphism of commutative group stacks $f: BK \to D(\Prym(Y/U))$ fitting into the canonical truncation exact sequence \cite[Example 2.13]{Brochard}
    \begin{equation} \label{eq: trunc}
    0 \to BK \to D(\Prym(Y/U)) \to [\Prym(Y/U)/(\J^0(C)/K)] \to 0.        
    \end{equation}
    Dualizing, by \cite[Remark 3.19, Corollary 11.5(ii)]{Brochard}, we obtain the exact sequence
    \begin{equation} \label{eq: trunc D}
     0 \to \Prym(Y/U)^\circ \to \Prym(Y/U) \to K^\vee \to 0. 
    \end{equation}
    Using this to identify $K^\vee = \pi_0(\Prym(Y/U))$, we see that $Df: \Prym(Y/U) \to \pi_0(\Prym(Y/U))$ is given by $\Prym(Y/U)^{\alpha} \mapsto \alpha$.    
    
    Finally, viewing $\mc P_{D(\Prym)/\Prym}$ as a $\J^0(C)[m]$-equivariant line bundle on $\Prym(Y/U) \times_U \Prym(Y/U)$, we note that $\J^0(C)[m]$ acts transitively on $\pi_0(\Prym(Y/U))$ \cite[Theorem 1.1(2)]{HP}, and so it suffices to calculate the eigenvalue on the zero section $\{\ms O\} \times_U \Prym(Y/U)$, i.e., identify the $K$-action on $(f \times 1)^* \mc P_{D(\Prym)/\Prym}$ on $BK \times_U \Prym(Y/U)$. Combining the isomorphism $(f \times 1)^* \mc P_{D(\Prym)/ \Prym} \cong (1 \times Df)^* \mc P_{BK/K^\vee}$ of Lemma \ref{lem: compare smooth Poincare sheaves} with the description of $\mc P_{BK/K^\vee}$ in the first paragraph and the description of $Df$ in the second paragraph, we see that $K$ acts on $(f \times 1)^* \mc P_{D(\Prym)/\Prym}|_{D(\Prym) \times \Prym^\alpha}$ by the character $\alpha \in K^\vee = \pi_0(Y/U)$.
\end{proof}

\subsubsection{Comparison of correspondences} \label{subsubsec: hom realization Fourier smooth}
With these preliminaries on the Poincar\'e bundle, we are ready to compare the cohomological correspondences defined in \S \ref{subsec: main theorem} with the correspondences $\mf F, \mf F^{-1}$ used in \cite{MSY}, still in the smooth case of $Y/U$.
To ease notation, we write $\Prym$ instead of $\Prym(Y/U)$ for the rest of \S \ref{subsec: hom realization smooth}.

Recall the constant group scheme $K := \ker(\J^0(C) \to \J^0(Y/U)) \subset \J^0(C)[m]$, and recall that $\Prym = \coprod_{\alpha \in K^\vee} \Prym^\alpha$. Also note that
$$I D(\Prym(Y/U)) =  I([\Prym(Y/U)/\J^0(C)[m]]) = \coprod_{\gamma \in K} D(\Prym(Y/U)).$$
On a nice quotient $\mc X$, define $\tCh(-) := \tch(-) \cap [I \mc X] \in \CH_*(I \mc X)$.
Then we consider
$$\tilde{\mf G} = (\tilde{\mf G}_{\gamma, \alpha})_{\gamma \in K, \alpha \in K^\vee} := \tCh(\mc P_{D(\Prym)/\Prym}) \in \oplus_{\gamma \in K, \alpha \in K^\vee} \CH_*(D(\Prym) \times_U \Prym^\alpha)$$
and
$$\mf F := \Ch(\mc P_{(\Prym^\circ)^\vee/\Prym^\circ}) \in \CH_*((\Prym^\circ)^\vee \times_U \Prym^\circ).$$
Note that $\mf F$ is now as defined in \cite{MSY}, since $\Prym^\circ$ is simply an abelian scheme, and $\tilde{\mf G} $ is as defined in Definition \ref{def: G and p}, by Theorem \ref{thm: tau properties}(d) and the fact that $K$ acts trivially and so $IB = \coprod_{k \in K} B$. (We save the notation $\mf G$ for the singular case, see \S \ref{subsec: proof of realization}.)

To compare $\mf F$ and $\tilde{\mf G}$ we need the following two pieces of notation: first, we have the (flat) $BK$-gerbe and coarse moduli space map
\begin{equation} \label{eq: def of g}
  g: D(\Prym) = [\Prym/\J^0(C)[m]] \to [\Prym/(\J^0(C)[m]/K)] = (\Prym^\circ)^\vee.  
\end{equation}
Second, for each $\alpha \in \pi_0(\Prym)$, by \cite[Theorem 1.1(2)]{HP} we may choose some $L_\alpha \in \J^0(C)[m]$ pulling back to a section in component $\alpha$. The multiplication by the pullback of $L_\alpha$ induces an isomorphism
$$f_\alpha: \Prym(Y/U)^\circ \to \Prym(Y/U)^\alpha.$$
We also obtain a line bundle $\mc P_\alpha := \mc P_{D(\Prym)/\Prym}|_{D(\Prym) \times \{L_\alpha\}}$ on $D(\Prym)$.

\begin{prop} \label{prop: components of smooth chP}
    For any $\gamma \in K$ and $\alpha \in K^\vee$, we have
    $$cl((1 \times f_\alpha)^*\tilde{\mf G}_{\gamma, \alpha}) = (\alpha(\gamma) \cdot cl((g \times 1)^* \mf F))_{\gamma \in K} \in \oplus_{\gamma \in K} H^{BM}_*(D(\Prym) \times_U \Prym^\circ).$$
\end{prop}

\begin{proof}
    First, note that
    $$(1 \times f_\alpha)^* \mc P_{D(\Prym)/\Prym} = (1 \times \mu(L_\alpha, -))^* \mc P_{D(\Prym)/\Prym} \cong p_1^* \mc P_\alpha \otimes \mc P_{D(\Prym)/\Prym}$$
    by Corollary \ref{cor: multiplicativity of P}, using the fact that $\Prym = D(D(\Prym))$ \cite[Theorem 4.13(i)]{Brochard}. In particular, if $j: \Prym^\circ \hookrightarrow \Prym$ is the inclusion, then
    \begin{equation} \label{eq: smooth comparison 1}
    (1 \times f_\alpha)^* \tilde{\mf G}_{\gamma,\alpha} = \tCh(p_1^* \mc P_\alpha \otimes (1 \times j)^*\mc P_{D(\Prym)/\Prym})_\gamma.   
    \end{equation}
    Now, from the pair of dual exact sequences (\ref{eq: trunc}) and (\ref{eq: trunc D}) we see that $j = Dg$. Thus, by Lemma \ref{lem: compare smooth Poincare sheaves} we have
    $$(1 \times j)^* \mc P_{D(\Prym)/\Prym} = (g \times 1)^* \mc P_{(\Prym^\circ)^\vee/\Prym^\circ}.$$
    Using the compatibility of the operational class $\tch$ with pullback (Proposition \ref{prop: ttd properties}(a)), we see that
    $$\tCh((1 \times j)^* \mc P_{D(\Prym)/\Prym}) = (g \times 1)^* \Ch(\mc P_{(\Prym^\circ)^\vee/\Prym^\circ}) \in \CH_*(D(\Prym) \times_U \Prym^\circ).$$
    Combining this with (\ref{eq: smooth comparison 1}), we have
    $$(1 \times f_\alpha)^* \tilde{\mf G}_{\gamma,\alpha} = \tch(p_1^* \mc P_\alpha) \cap (g \times 1)^*\mf F.$$
    Applying the cycle class map, we obtain the needed result from Lemma \ref{lem: Palpha top trivial}.
\end{proof}

\begin{lem} \label{lem: Palpha top trivial}
    We have $c_1(\mc P_\alpha) = 0 \in H^*(D(\Prym), \BC)$, and $K$ acts on $\mc P_\alpha$ by the character $\alpha \in K^\vee$. In particular, for any $(C_\gamma)_{\gamma \in K} \in \CH_*(D(\Prym) \times_U \Prym^\circ)$ we have
    $$cl(\tch(p_1^* \mc P_\alpha) \cap (C_\gamma)_{\gamma \in K}) = (\alpha(\gamma) \cdot cl(C_\gamma))_{\gamma \in K}.$$
\end{lem}

\begin{proof}
    First, since the $\J^0(C)[m]$-quotient map $k:\Prym \to D(\Prym)$ induces an injection on cohomology, it suffices to show the vanishing of $c_1(k^* \mc P_\alpha)$. Letting $i: \Prym \hookrightarrow \J^0$ be the inclusion, recall from the proof of Corollary \ref{cor: smooth Prym P/P vs J/J} that $k = Di \circ i$, and thus
    $$k^* \mc P_\alpha = (i \times 1)^* (Di \times 1)^*\mc P_{D(\Prym)/\Prym}|_{\Prym \times \{L_\alpha\}}$$
    $$= (i \times 1)^* (1 \times i)^*\mc P_{J/J}|_{\Prym \times \{L_\alpha\}} = \mc P_{J/J}|_{\Prym \times \{L_\alpha\}}$$
    by Lemma \ref{lem: compare smooth Poincare sheaves}. We now note that under the pullback $l: \J^0(C) \to \J^0(Y/U)$, we obtain a family of line bundles $(i \times l)^* \mc P_{J/J}$ on $\Prym(Y/U)$ over the connected base $\J^0(C)$ with $(i \times l)^* \mc P_{J/J}|_{\Prym \times \{\ms O_C\}} \cong \ms O_{\Prym}$. Then $c_1(\mc P_{J/J}|_{\Prym \times L_\alpha}) = c_1(\ms O_{\Prym}) = 0$ since these two classes are algebraically, thus also topologically, equivalent.

    Next, it follows from Corollary \ref{cor: eigenvalue of P} that $K$ acts on $\mc P_\alpha$ via the character $\alpha$, and so
    $$\tch(\mc P_\alpha) = (\alpha(\gamma) \cdot \ch^{EG}(\mc P_\alpha))_{\gamma \in K} \in \CH^*(ID(\Prym)) = \oplus_{\gamma \in K} \CH^*(D(\Prym)).$$
    Recalling from \S \ref{subsubsec: constructible sheaves} the compatibility of Chern classes with the cycle class map, we have
    $$cl(\tch(\mc P_\alpha) \cap (C_\gamma)_{\gamma \in K}) = (\alpha(\gamma) \cdot \ch^{EG}_{top}(\mc P_\alpha))_\gamma \cap (cl(C_\gamma))_\gamma = (\alpha(\gamma) \cdot cl(C_\gamma))_\gamma$$
    by the vanishing of $c_1(\mc P_\alpha)$.
\end{proof}

Let $d$ be the relative dimension of the projection $\pi_U: \Prym \to U$. Recall from Proposition \ref{prop: GS} that $\mc P_{D(\Prym)/\Prym}$ has (left and right) adjoint
$$\mc P^{-1}_{\Prym/D(\Prym)} := R\mc Hom(\mc P_{\Prym/D(\Prym)}, \pi_U^* \ms A)[d] = \mc P_{\Prym/D(\Prym)}^\vee \otimes \pi_U^* \ms A[d]$$
for a line bundle $\ms A = \det \pi_{X/U*} \omega_{X/U}$ on $U$. Similarly, in the non-stacky case $\mc P_{(\Prym^\circ)^\vee/\Prym^\circ}$ has adjoint
$$\mc P^{-1}_{\Prym^\circ/(\Prym^\circ)^\vee} := \mc P^{\vee}_{\Prym^\circ/(\Prym^\circ)^\vee} \otimes \pi_U^* \ms A[d].$$
We define
$$\tilde{\mf G}^{-1} = (\tilde{\mf G}^{-1}_{\gamma, \alpha})_{\gamma \in K, \alpha \in K^\vee} := \ttd(-\pi_{U}^* T_U) \cap \tau(\mc P^{-1}_{\Prym/D(\Prym)}) \in \bigoplus_{\gamma \in K, \alpha \in K^\vee} \CH_*(\Prym^\alpha \times_U D(\Prym))$$
and
$$\mf F^{-1} := \td(-\pi_{U}^* T_U) \cap \tau(\mc P^{-1}_{\Prym^\circ/(\Prym^\circ)^\vee}) \in \CH_*(\Prym^\circ \times (\Prym^\circ)^\vee).$$
Again, these are $\mf G^{-1}$ from Definition \ref{def: G and p} and $\mf F^{-1}$ from \cite{MSY}, respectively.
As before, we have a comparison:

\begin{prop} \label{prop: components of smooth chP-1}
    For any $\gamma \in K$ and $\alpha \in K^\vee$, we have
    $$cl((f_\alpha \times 1)^*\tilde{\mf G}_{\gamma, \alpha}^{-1}) = (\alpha(\gamma^{-1}) \cdot cl((1 \times g)^* \mf F^{-1}))_{\gamma \in K} \in \oplus_{\gamma \in K} H^{BM}_*(\Prym^\circ \times_U D(\Prym)).$$
\end{prop}

\begin{proof}
    Exactly as in Proposition \ref{prop: components of smooth chP} we have
    $$cl((f_\alpha \times 1)^* 
    \tCh(\mc P^\vee_{\Prym/D(\Prym)})_{\alpha, \gamma})= \alpha(\gamma^{-1}) cl((1 \times g)^*\Ch(\mc P^\vee_{\Prym^\circ/(\Prym^\circ)^\vee}))$$
    in $H^{BM}_*(\Prym^\circ \times_U D(\Prym)).$
    Then applying
    $$cl(\tch(\pi_U^* \ms A[d])_\gamma \cap \ttd(-\pi_U^* T_U + T_{\Prym^\circ \times D(\Prym)})_\gamma) $$
    $$= cl((1 \times g)^*(\ch(\pi_U^* \ms A[d]) \cap \td(-\pi_U^* T_U + T_{\Prym^\circ \times (\Prym^\circ)^\vee}))) \in H^*(\Prym^\circ \times_U D(\Prym))$$
    to both sides (where for the equality we use Proposition \ref{prop: ttd properties}(a) and the fact that $g^* T_{(\Prym^\circ)^\vee} = T_{D(\Prym)}$), the needed result follows from Theorem \ref{thm: tau properties}(a) and (d).
\end{proof}

\subsubsection{Fourier transform}
We are now ready to complete the proof of Theorem \ref{thm: main sec 2}(ii) in the smooth case. We start with a toy example of Fourier transform in the stacky setting, to see how the components of $\tau(\mc P_{D(\Prym)/\Prym})$ on the different pieces of the inertia stack will recover the decomposition of $H^*(\Prym)$ into $K$-isotypic components:

\begin{example} \label{ex: Fourier for finite G}
    Let $G$ be a finite abelian group, and let $G^\vee$ be its dual. Recall from the proof of Corollary \ref{cor: eigenvalue of P} that $BG$ and $G^\vee$ are dual commutative group stacks, and $\mc P_{BG/G^\vee}$ is the trivial vector bundle on $G^\vee$ such that $G$ acts on the fiber over $\alpha \in G^\vee$ by the character $\alpha$. We see that
    $$\tCh(\mc P_{BG/G^\vee}) = (\alpha(\gamma) [BG \times_U \{\alpha\}])_{\gamma \in G, \alpha \in G^\vee} \in \oplus_{\gamma \in G, \alpha \in G^\vee} \CH_*(BG \times_U \{\alpha\})$$
    with inverse
    $$\tCh(\mc P_{G^\vee/BG}^{-1}) = (\alpha(\gamma^{-1}) [\{\alpha\} \times_U BG])_{\gamma \in G, \alpha \in G^\vee} \in \oplus_{\gamma \in G, \alpha \in G^\vee} \CH_*( \{\alpha\} \times_U BG).$$
    %For example, from
    %$$[BG \times_U \{\beta\}] \circ [\{\alpha\} \times_U BG] = \frac{1}{\# G} [\{\alpha\} \times_U \{\beta\}],$$
    %it follows that
    %$$\tCh(\mc P_{BG/G^\vee}) \circ \tCh(\mc P_{G^\vee/BG}^{-1}) =  \sum_{\alpha, \beta \in G^\vee} \sum_{\gamma \in G}\frac{\alpha^{-1} \beta(\gamma)}{\# G} [\{\alpha\} \times_U \{\beta\}]  = \sum_{\alpha \in G^\vee} \frac{1}{\# G} [\{\alpha\} \times_U \{\alpha\}] = [\Delta_{G^\vee}].$$
    Consider the decomposition of $H^*(G^\vee)$ into isotypic components indexed by elements $\gamma \in G$: we have
    $$H^*(G^\vee)_\gamma = H^*(U) \cdot (\alpha(\gamma))_\alpha \subset H^*(G^\vee) = \oplus_{\alpha \in G^\vee} H^*(U).$$
    Then note that the cycle class of $\tCh(\mc P^{-1}_{G^\vee/BG})_\gamma$ for any $\gamma \in G$ gives a morphism
    $$H^*(G^\vee) = \oplus_\alpha H^*(U) \to H^*(BG) = H^*(U), (x_\alpha)_\alpha \mapsto \frac{1}{\# G} \sum_\alpha \alpha(\gamma^{-1}) x_\alpha,$$
    which is trivial on all isotypic components other than $H^*(G^\vee)_\gamma$ but defines an isomorphism $H^*(G^\vee)_\gamma \to H^*(BG)$.
\end{example}

We recall the following notation: for each $\gamma \in K$, we write
$$\tilde{\mf G}_\gamma =: \sum_{i \ge 0} \tilde{\mf G}_{\gamma, i}, \,\, \tilde{\mf G}_{\gamma, i} \in  \Corr^{i-d}_U(D(\Prym), \Prym)$$
and similarly for $\tilde{\mf G}^{-1}_\gamma$ (again matching Definition \ref{def: G and p}, since $h = h_\gamma =: d$). We also decompose
$$\mf F =: \sum_{i \ge 0} \mf F_i, \,\, \mf F_i \in \Corr_U^{i-d}((\Prym^\circ)^\vee, \Prym^\circ)$$
and similarly for $\mf F^{-1}$, as in \cite{MSY}. We define projectors $\mf p^{\tilde{\mf G}}_{\gamma,k}, \mf q_{\gamma,k}^{\tilde{\mf G}}$ for $\tilde{\mf G}$, respectively $\mf p_k, \mf q_k$ for $\mf F$, as in (\ref{eq: pGgk def}) and (\ref{eq: qGgk def}), respectively \cite[Theorem 2.6(i)]{MSY}. The following proposition is a summary of the smooth case of \cite[Theorem 2.6]{MSY}, which comes from \cite[Remark 2 after Corollary 3.2]{Deninger-Murre}.

\begin{prop} \label{prop: MSY smooth}
    The $\mf p_k, \mf q_k \in \Corr^0_U(\Prym^\circ, \Prym^\circ)$ are projectors, such that the homological realization of $(\Prym^\circ, \mf p_k, 0) \to (\Prym^\circ, \Delta_{\Prym^\circ}, 0)$ is the perverse truncation
    $$\ptau_{\le k+\dim U} \pi_{U*} \BC_{\Prym^\circ} \to \pi_{U*} \BC_{\Prym^\circ}$$
    and the homological realization of $(\Prym^\circ, \Delta_{\Prym^\circ}, 0) \to (\Prym^\circ, \mf q_k, 0)$ is
    $$\pi_{U*} \BC_{\Prym^\circ} \to \ptau_{\ge k +\dim U} \pi_{U*} \BC_{\Prym^\circ}.$$
\end{prop}

\begin{prop} \label{prop: (ii) smooth}
    Theorem \ref{thm: main sec 2}(ii) holds for the smooth Prym fibration $\pi_U: \Prym(Y/U) \to U$.
\end{prop}

\begin{proof}
    Note that components of $I(D(\Prym))$ are indexed by $\gamma \in K$ and that $\ker(\J^0(C)[m] \to \pi_0(\Prym))$ acts trivially on $\pi_{U*} \BC_{\Prym}$ (see \cite[Lemma 7.2]{HP}), and so the isotypic components of $\pi_{U*} \BC_{\Prym}$ are indexed by characters of $\pi_0(\Prym) = K^\vee$, i.e., by elements of $K$. As noted in \cite[Proof of Proposition 1.3]{MS}, the character $\J^0(C)[m] \to \pi_0(\Prym) \xrightarrow{\gamma} \BC^*$ can be described in terms of the Weil pairing as $\kappa := \langle \gamma, - \rangle$; we use $(\pi_{U*} \BC_{\Prym})_\gamma$ and $(\pi_{U*} \BC_{\Prym})_\kappa$ interchangeably to describe this component.
    
    We consider only $\mf p_{\gamma, k}^{\tilde{\mf G}}$; the proof for $\mf q_{\gamma, k}^{\tilde{\mf G}}$ is identical. By definition and the compatibility of the cycle class map with the operations on Chow groups (see \S \ref{subsubsec: constructible sheaves}) we have
    $$cl(\mf p_{\gamma, k}^{\tilde{\mf G}}) = \sum_{i \le k} cl(\tilde{\mf G}_{\gamma,i}) \circ cl(\tilde{\mf G}^{-1}_{\gamma, 2d-i})  = \bigg(\sum_{i \le k} cl(\tilde{\mf G}_{\gamma,\beta,i}) \circ cl(\tilde{\mf G}^{-1}_{\gamma, \alpha, 2d-i})\bigg)_{\alpha, \beta \in K^\vee}$$
    $$= \bigg(\sum_{i \le k} \beta(\gamma) cl((g \times 1)^* \mf F_i) \circ \alpha(\gamma^{-1}) cl((1 \times g)^* \mf F^{-1}_{2d-i}) \bigg)_{\alpha, \beta \in K^\vee}\in \oplus_{\alpha, \beta \in K^\vee} H^{BM}_*(\Prym^\circ \times_U \Prym^\circ)$$
    by Propositions \ref{prop: components of smooth chP} and \ref{prop: components of smooth chP-1}, where here we use the isomorphisms $(f_\alpha \times f_\beta): \Prym^\circ \times_U \Prym^\circ \xrightarrow{\sim} \Prym^\alpha \times_U \Prym^\beta$ to identify the sum with $H_*^{BM}(\Prym \times_U \Prym)$
    $$ = \bigg( \sum_{i \le k} \alpha^{-1}\beta(\gamma) \cdot cl( \mf F_i \circ [\Gamma_g]) \circ cl([\Gamma_g]^T \circ  \mf F^{-1}_{2d-i}) \bigg)_{\alpha, \beta \in K^\vee}$$
    by Lemma \ref{lem: composition with graphs}(a)
    $$ = \bigg(\sum_{i \le k} \alpha^{-1}\beta(\gamma) \cdot cl( \mf F_i \circ [\Gamma_g] \circ [\Gamma_g]^T \circ  \mf F^{-1}_{2d-i})\bigg)_{\alpha, \beta \in K^\vee}$$
    $$ = \frac{1}{\# K} \bigg(\sum_{i \le k} \alpha^{-1}\beta(\gamma) \cdot cl( \mf F_i \circ  \mf F^{-1}_{2d-i})\bigg)_{\alpha, \beta \in K^\vee} = \frac{1}{\# K} \bigg( \alpha^{-1}\beta(\gamma) \cdot cl( \mf p_k) \bigg)_{\alpha, \beta \in K^\vee}$$
    by Lemma \ref{lem: composition with graphs}(b). Then as in Example \ref{ex: Fourier for finite G} we note that the corresponding morphism of $\pi_{U*} \BC_{\Prym}$ is trivial on all $K$-isotypic components other than $(\pi_{U*} \BC_{\Prym})_\gamma$, and under the isomorphism
    $$h := (\alpha(\gamma) \cdot)_{\alpha \in K^\vee}: \pi_{U*} \BC_{\Prym^\circ} \xrightarrow{\sim} (\pi_{U*} \BC_{\Prym})_\gamma \subset \oplus_{\alpha \in K^\vee} \pi_{U*} \BC_{\Prym^\circ}$$
    we see that
    $$\pi_{U*} \BC_{\Prym^\circ} \xrightarrow{h} (\pi_{U*} \BC_{\Prym})_\gamma \subset \pi_{U*} \BC_{\Prym} \xrightarrow{cl(\mf p^{\tilde{\mf G}}_{\gamma, k})} \pi_{U*} \BC_{\Prym} \cong \oplus_{\beta \in K^\vee} \pi_{U*} \BC_{\Prym^\circ}$$
    is given by
    $$\frac{1}{\# K} \bigg(\sum_{\alpha \in K^\vee} \alpha(\gamma) \cdot \alpha^{-1} \beta(\gamma) \cdot cl(\mf p_k)\bigg)_{\beta \in K^\vee} = (\beta(\gamma) \cdot cl(\mf p_k))_{\beta \in K^\vee} = h \circ cl(\mf p_k).$$
    In other words, we see that the morphisms induced by $\mf p^{\tilde{\mf G}}_{\gamma, k}$ and $\mf p_k$ agree under the isomorphism $h: \pi_{U*} \BC_{\Prym^\circ} \xrightarrow{\sim} (\pi_{U*} \BC_{\Prym})_\kappa $, and thus, by Proposition \ref{prop: MSY smooth}, the image of $cl(\mf p_{\gamma, k}^{\tilde{\mf G}})$ is exactly
    $$\ptau_{\le k + \dim U} (\pi_{U*} \BC_{\Prym})_\kappa.$$
\end{proof}

\begin{lem} \label{lem: composition with graphs}
    (a) Let $\mc X, \mc Y, \mc Z$ be smooth nice quotients, all of which are proper over a smooth base variety $B$. Given $\eta \in \Corr_B^*(\mc X,\mc Y)$ and a morphism $f: \mc Y \to \mc Z$ over $B$ with graph $[\Gamma_f] \in \Corr^*_B(\mc Y, \mc Z)$, we have
    $$(1 \times f)_* \eta = [\Gamma_f] \circ \eta, \,\, \textrm{and the transpose} \,\, (f \times 1)_* \eta^T = \eta^T \circ [\Gamma_f]^T.$$
    Similarly, given $\eta \in \Corr_B^*(\mc Z, \mc X)$, we have
    $$(f \times 1)^* \eta = \eta \circ [\Gamma_f], \,\, \textrm{and the transpose} \,\, (1 \times f)^* \eta^T = [\Gamma_f]^T \circ \eta^T.$$

    (b) Let $[\Gamma_g] \in \Corr^0_U(D(\Prym), (\Prym^\circ)^\vee)$ be the graph of the morphism $g$ defined in (\ref{eq: def of g}). We have
    $$[\Gamma_g] \circ [\Gamma_g]^T = \frac{1}{\# K} [\Delta_{(\Prym^\circ)^\vee}] \in \Corr^0_U((\Prym^\circ)^\vee, (\Prym^\circ)^\vee).$$
\end{lem}

\begin{proof}
    Part (a) follows as in the scheme case (see \cite[Proposition 16.1.1(c)]{Fulton}, which is stated in the case that $B$ is trivial but holds in this generality). The point to note is that the diagonal $\delta_{\mc Y}$ for the first statement, respectively $\delta_{\mc Z}$ for the second, is a representable regular local immersion in the sense of \cite[\S 4.1]{Kresch}, and thus there is a well-defined Gysin pullback satisfying the standard compatibilities with flat pullback and projective pushforward.

    For part (b), consider the stack of Cartesian squares
    \[
    \begin{tikzcd}
        D(\Prym) \arrow[r] \arrow[d, "g \times 1 \times g"] & D(\Prym) \times D(\Prym) \arrow[d, "(g \times 1) \times (1 \times g)"] \\
        (\Prym^\circ)^\vee \times_U D(\Prym) \times_U (\Prym^\circ)^\vee \arrow[d, "p_2"] \arrow[r] & (\Prym^\circ)^\vee \times_U D(\Prym) \times D(\Prym) \times_U (\Prym^\circ)^\vee \arrow[d, "p_{23}"] \\
        D(\Prym) \arrow[r, "\Delta_{D(\Prym)}"] & D(\Prym) \times D(\Prym)
    \end{tikzcd}
    \]
    We start by observing that
    $$\Delta_{D(\Prym)}^!([\Gamma_g]^T \times [\Gamma_g]) = \Delta^!_{D(\Prym)}(g \times 1 \times 1 \times g)_* [D(\Prym) \times D(\Prym)] $$
    $$= (g \times 1 \times g)_* \Delta_{D(\Prym)}^![D(\Prym) \times D(\Prym)]$$
    by the generalization of \cite[Theorem 6.2(a)]{Fulton} (see \cite[\S 4.1]{Kresch}), noting that $\Delta_{D(\Prym)}$ is a regular local immersion
    $$=(g \times 1 \times g)_* [D(\Prym)]$$
    by the corresponding statement on schematic approximations (noting that $\Delta_{D(\Prym)}$ is representable). Thus, we see that
    $$[\Gamma_g] \circ [\Gamma_g]^T = p_{13*} (g \times 1 \times g)_* [D(\Prym)] = \Delta_{(\Prym^\circ)^\vee *} g_*[D(\Prym)]$$
    $$= \frac{1}{\# K} [\Delta_{(\Prym^\circ)^\vee}] \in \Corr^0_U((\Prym^\circ)^\vee, (\Prym^\circ)^\vee),$$
    as $D(\Prym)$, respectively $(\Prym^\circ)^\vee$, is a $\J^0(C)[m]$-quotient, respectively $\J^0(C)[m]/K$-quotient, of $\Prym$.
\end{proof}

\subsection{Comparison of Poincar\'e sheaves} \label{subsec: FHHO Pryms in families}
In this section, we prove an identity (Proposition \ref{prop: smooth and singular Poincare}) comparing the Poincar\'e sheaf of the compactified Prym fibration of a family of nodal curves with the Poincar\'e sheaf of the Prym fibration of a family of normalizations. We start with the corresponding identity for the compactified Jacobian of a single nodal curve \cite[Theorem A]{FHHO}, restrict to the compactified Prym, show that the identity holds in families, and descend. Because our arguments will involve the moduli of semistable sheaves, in this subsection we do distinguish between derived and underived functors.

\subsubsection{Setup} \label{subsubsec: FHHO setup}
Let $U$ be a smooth variety, and $X \to U$ a family of projective curves, each with $k$ nodes, equipped with a finite flat map $c: X \to C \times U$ over $U$. Let $\tilde X \to U$ be a family of smooth projective curves with a morphism $\nu: \tilde X \to X$, where $\nu_u: \tilde X_u \to X_u$ is the normalization map for each $u \in U$. We assume that $\det \pi_{U*} \omega_{\tilde X/U} \cong \det \pi_{U*} \omega_{X/U} \cong \ms O_U$. We also assume that there are sections $\{\sigma_i: U \to X\}_{1 \le i \le k}$ parametrizing the $k$ nodes of each fiber of $X \to U$, and sections $\{\sigma_i^1, \sigma_i^2: U \to \tilde X\}_{1 \le i \le k}$ parametrizing the two preimages in $\tilde X$ of each node. Finally, suppose also that there is some section $s: U \to \tilde X$ disjoint from the $\{\sigma_i^1, \sigma_i^2\}$. (In Proposition \ref{prop: U exists} we will see that for each $\gamma \in \Gamma$, we can find such a $U$ \'etale over $B_\gamma$.)

We start by defining a morphism
$$a: \J^0(\tilde X/U) \xrightarrow{\sim} \J^{-k}(\tilde X/U) \to \oJ^0(X/U), [\ms L] \in \J^0(\tilde X_u) \mapsto [\ms L(-ks(u))] \mapsto [\nu_{u*} \ms L(-ks(u))].$$
Note that $a$ is $\J^0(C)$-equivariant by the projection formula, and thus gives a morphism
$$a_P: [\J^0(\tilde X/U)/\J^0(C)] \to [\oJ^0(X/U)/\J^0(C)].$$

Next, let $\tilde{\mc P}_{\tilde X/J}$ be the Poincar\'e line bundle on $\tilde X \times_U \J^0(\tilde X/U)$ normalized along $s$ (i.e., such that $(s \times 1)^* \tilde{\mc P}_{\tilde X/J} \cong \ms O_{J^0}$).
\begin{defn} \label{def: PMod}
    Consider the $\BP^1$-bundle
    $$\BP_i := \BP((\sigma_i^1 \times 1)^* \tilde{\mc P}_{\tilde X/J} \oplus (\sigma_i^2 \times 1)^* \tilde{\mc P}_{\tilde X/J})$$
    over $\J^0(\tilde X/U)$.
    We define varieties
    $$\PMod := \BP_1 \times_{\J^0(\tilde X/U)} \cdots \times_{\J^0(\tilde X/U)} \BP_k, \,\,\,\PMod_P := \Prym(\tilde X/U) \times_{\J^0(\tilde X/U)} \PMod$$
    as well as line bundles
    $$\mc V := \ms O_{\BP_1}(-1) \boxtimes \cdots \boxtimes \ms O_{\BP_k}(-1) \in \Pic(\PMod), \,\,\, \mc V_P := \mc V|_{\PMod_P}.$$
    (Note that $\PMod$ is independent of the choice of normalization of $\tilde{\mc P}_{\tilde X/J}$, but $\mc V$ is not.)
\end{defn}

By definition there are projections
$$p: \PMod \to \J^0(\tilde X/U), \,\, p_P: \PMod_P \to \Prym(\tilde X/U).$$
There is also a morphism
$$f: \PMod \to \oJ^0(X/U), ([\ms L] \in \J^0(\tilde X_u), (\ms L_{\sigma_i^1(u)} \oplus \ms L_{\sigma_i^2(u)} \twoheadrightarrow V_i)_i) \mapsto \ker(\nu_{u*} \ms L \to \oplus_i V_i)$$
where here we consider the sum of maps
$$\nu_{u*} \ms L \to \nu_{u*}(\ms L \otimes \ms O_{\sigma_i^1(u) \cup \sigma_i^2(u)}) = (\ms L_{\sigma_i^1(u)} \oplus \ms L_{\sigma_i^2(u)} ) \otimes \ms O_{\sigma_i(u)} \twoheadrightarrow V_i \otimes \ms O_{\sigma_i(u)}.$$
(Note that the kernel is still torsion-free of rank 1, and of degree $= \deg \nu_{u*} \ms L - k = 0$.) As described in \cite[Proposition 2.2]{Bhosle}, the map $f$ is a fiberwise resolution of singularities, and a point of $\oJ^0(\tilde X/U)$ has exactly $2^i$ preimages when it fails to be locally free at exactly $i$ nodes. We also note that $f$ restricts to a morphism
$$f_P: \PMod_P \to \oPrym(X/U)$$
since
$$\det(c_*(\ker(\nu_{u*} \ms L \to \oplus_i V_i))) = \det( c_* \nu_{u*} \ms L)(-\sum_i c(\sigma_i(u)))$$
$$= \det(c_* \nu_{u*} \ms O_{\tilde X_u})(-\sum_i c(\sigma_i(u))) = \det(c_*\ms O_{ X_u}).$$

\begin{prop} \label{prop: PMod = Sigma}
    The morphism $(f, p): \PMod \to \oJ^0(X/U) \times_U \J^0(\tilde X/U)$ is a closed embedding whose image is the closure of the graph of the pullback morphism $\nu^*: \J^0(X/U) \to \J^0(\tilde X/U)$. The same holds for $(f_P, p_P): \PMod_P \to \oPrym(X/U) \times_U\Prym(\tilde X/U)$ and $\nu^*: \Prym(X/U) \to \Prym(\tilde X/U)$.
\end{prop}

\begin{proof}
    Recall that $f: \PMod \to \oJ^0(X/U)$ is an isomorphism over $\J^0(X/U) \subset \oJ^0(X/U)$, with inverse given by
    $$\ms L \in \J^0(X_u) \mapsto (\nu_u^* \ms L, \textrm{coker}(\ms L \hookrightarrow \nu_{u*} \nu_u^* \ms L)) \in \PMod_u.$$
    Since $f^{-1}(\J^0(X/U))$ is dense in $\PMod$ and $p,f$ are proper, we see that the image of $(f,p)$ is the closure of $(f,p)(f^{-1}(\J^0(X/U)) = \Gamma_{\nu^*}$. It follows from \cite[Lemma 3.7(3)]{FHHO} that $(f,p)$ is injective on points. Then since $\overline{\Gamma}_{\nu^*}$ is smooth \cite[Lemma A.3.1]{Yun}, it follows from Zariski's main theorem that $(f,p): \PMod \to \overline{\Gamma}_{\nu^*}$ is an isomorphism. The same argument holds for $(f_P, p_P)$.
\end{proof}

Recall that we write $\mc P_{J/J}$ for the Arinkin Poincar\'e sheaf on $\oJ^0(X/U) \times_U \oJ^0(X/U)$ and $\mc P$ for the Poincar\'e sheaf on $[\oJ^0(X/U)/\J^0(C)] \times_U \oPrym(X/U)$. We will also write $\tilde{\mc P}_{J/J}$ for the Poincar\'e bundle on $\J^0(\tilde X/U) \times_U \J^0(\tilde X/U)$ and $\tilde{\mc P}$ for the restriction and descent to $[\J^0(\tilde X/U)/\J^0(C)] \times_U \Prym(\tilde X/U)$ constructed as in Proposition \ref{prop: GS}. As there, we define
$$\mc P^{-1} := R\mc Hom(\mc P^T, \ms O[h]), \,\, \tilde{\mc P}^{-1} := R\mc Hom(\tilde{\mc P}^T, \ms O[h-k])$$
(since by our assumptions on $X, \tilde X \to U$ we have $\ms A \cong \ms O_U$ in each case, and $\dim \Prym(\tilde X_u) = \dim \oPrym(X_u)-k$).

The goal of \S \ref{subsec: FHHO Pryms in families} is to prove the following:
\begin{prop} \label{prop: smooth and singular Poincare}
    With the definitions of this section, and letting $q_1$ be the projection onto the first factor of $\PMod_P \times_U [\J^0(\tilde X/U)/\J^0(C)]$, there is an isomorphism
    $$(1 \times a_P)^* \mc P^{-1} \cong (f_P \times 1)_*(q_1^*(\omega_{\PMod_P} \otimes \mc V_P^\vee) \otimes (p_P \times 1)^* \tilde{\mc P}^{-1})[k]$$
    of $h$-shifted sheaves on $\oPrym(X/U) \times_U [\J^0(\tilde X/U)/\J^0(C)]$.
\end{prop}

\begin{rmk}
    The Grothendieck dual of a very similar identity (over a trivial base $U$) is proved in \cite[Propositions 5.16-7]{FHHO}. However, there the $\Gamma$-equivariant structure of the Poincar\'e sheaf on $\oPrym(X_u) \times \oPrym(X_u)$ is defined in order to make the equality hold. In particular, it is not a priori clear that the resulting sheaf is $\mc P$ as defined in Proposition \ref{prop: GS} rather than, say, a tensor product of $\mc P$ with the pullback of a nontrivial line bundle from $B\Gamma$.
\end{rmk}

\subsubsection{Descending in families} \label{subsubsec: FHHO descent}
We will now prove Proposition \ref{prop: smooth and singular Poincare} in a series of lemmas. The starting point is the following:

\begin{lem} \label{lem: FHHO}
    With notation and definitions as before, for any $u \in U$ there is an isomorphism
    \begin{equation} \label{eq: lem FHHO}
      (1 \times a_u)^* \mc P_{J_u/J_u} \cong (f_u \times 1)_*(q_1^* \mc V_u \otimes (p_u \times 1)^* \tilde{\mc P}_{J_u/J_u})  
    \end{equation}
    of sheaves on $\oJ^0(X_u) \times \J^0(\tilde X_u)$. Moreover, the right and left side of the equation do not change if all functors are taken to be derived.
\end{lem}

\begin{proof}
    The first part is \cite[Theorem 5.3]{FHHO}. For the second part, on the right side of the equation, since $\mc V_u$ and $\tilde{\mc P}_{J_u/J_u}$ are line bundles, it suffices to note that $R(f_u \times 1)_* = (f_u \times 1)_*$ since $f_u \times 1$ is finite. On the left side of the equation, recall from \cite[Theorem A(2)]{Arinkin} that $\mc P_{J_u/J_u}$ is flat over the second factor $\oJ^0(X_u)$, and so $L(1 \times a_u)^* = (1 \times a_u)^*$ by \cite[Tag 0C0V]{stacks-project}.
\end{proof}

\begin{lem} \label{lem: FHHO families to point}
    Consider sheaves
    $$\mc F := (1 \times a)^* \mc P_{J/J}$$
    $$\mc G := (f \times 1)_*(q_1^* \mc V \otimes (p \times 1)^* \tilde{\mc P}_{J/J})$$
    on $\oJ^0(X/U) \times_U \J^0(\tilde X/U)$.
    For each $u \in U$, let $\iota_u: \oJ^0(X_u) \times \J^0(\tilde X_u) \hookrightarrow \oJ^0(X/U) \times_U \J^0(\tilde X/U)$ be the inclusion. Then $\iota_u^*\mc F, \iota_u^*\mc G$ are isomorphic to the sheaves of Lemma \ref{lem: FHHO}.
\end{lem}

\begin{proof}
    The statement for $\mc F$ is clear once we recall from \cite{Arinkin} that $\iota_u^* \mc P_{J/J} = \mc P_{J_u/J_u}$. (From the description of each in terms of pushforward of determinant of cohomology it is clear that they agree on $\oJ^0 \times \J^0 \cup \J^0 \times \oJ^0$, i.e., away from codimension 2. Recall that $\mc P_{J/J}$ and $\mc P_{J_u/J_u}$ are maximal Cohen-Macaulay; by \cite[Lemma 2.3]{Arinkin} the same is true of $\iota_u^* \mc P_{J/J}$, and thus this agrees with $\mc P_{J_u/J_u}$ by \cite[Lemma 2.2]{Arinkin}.)
    
    As for $\mc G$, we start by noting that $L\iota_u^* (f \times 1)_* = (f_{u} \times 1)_* L\iota_u^*$ by Lemma \ref{lem: Tor indep}(b) and finiteness of $f$. Thus
    $$\iota_u^* \mc G = \mc H^0((f \times 1)_* L\iota_u^*(q_1^* \mc V \otimes (p \times 1)^*\tilde{\mc P}_{J/J}))$$
    $$=\mc H^0((f \times 1)_*(q_1^* \mc V_u \otimes (p_{u} \times 1)^* \iota_u^*\tilde{\mc P}_{J/J}))$$
    since $L\iota_u^* = \iota_u^*$ on a line bundle. Finally, we note that the $\mc H^0$ is superfluous, and that $\iota_u^* \tilde{\mc P}_{J/J} = \tilde{\mc P}_{J_u/J_u}$.
\end{proof}

The next step is to prove that $\mc F$ and $\mc G$ are isomorphic, using a suitable generalization of the ``seesaw lemma" for line bundles. As in \cite[Proposition 4.4]{Bai}, the approach is to note that $\mc F, \mc G$ induce the same morphism to a (separated) moduli space of rank 1 torsionfree sheaves and therefore differ at most by a line bundle, which can be shown to be trivial. However, instead of using an analytic moduli space of simple sheaves, we use the moduli spaces of relatively stable sheaves constructed in \cite{Simpson-reps1}. (See \cite[Remark 4.5]{Bai} for further discussion.)

\begin{lem} \label{lem: FHHO relative J iso}
    We have $\mc F \cong \mc G$ for the sheaves $\mc F$ and $\mc G$ on $\oJ^0(X/U) \times_U \J^0(\tilde X/U)$ of Lemma \ref{lem: FHHO families to point}.
\end{lem}

\begin{proof}
    Our first goal is to show that $\mc F$ and $\mc G$ are relatively Gieseker stable (for any polarization) in the sense of \cite[\S 1]{Simpson-reps1}, i.e., flat over $U$ with $\iota_u^* \mc F \cong \iota_u^* \mc G$ Gieseker stable for each $u \in U$. Since $\mc P_{J/J}$ is flat over each factor \cite[Theorem A(2)]{Arinkin}, we see that $\mc F = (1 \times a)^* \mc P_{J/J}$ is flat over the second factor $\oJ^0(X/U)$, thus over $U$. As for $\mc G$, we note that $q_1^* \mc V \otimes (p \times 1)^* \tilde{\mc P}_{J/J}$ is a line bundle over $\PMod$, which is flat over $U$. Then since $(f_P \times 1)$ is finite (thus affine), the pushforward $\mc G$ remains flat over $U$. Finally, since by Lemma \ref{lem: FHHO families to point} the fiber $\iota_u^* \mc F \cong \iota_u^* \mc G$ is isomorphic to the right side of (\ref{eq: lem FHHO}), we see that this must be torsion-free (being the pushforward of a line bundle from an integral variety) of rank 1. Thus this is Gieseker stable for any polarization.

    Thus, choosing any relatively ample line bundle $H$ on $\oJ^0(X/U) \times_U \oJ^0(X/U) \to U$, both $\mc F$ and $\mc G$ give rise to morphisms $\phi_{\mc F}, \phi_{\mc G}: U \to M$, where $M$ is the coarse moduli space of relatively Gieseker semistable sheaves on $\oJ^0(X/U) \times_U \oJ^0(X/U)$ with Hilbert polynomial matching that of $\iota_u^* \mc F \cong \iota_u^* \mc G$ constructed, as constructed in \cite[Theorem 1.21]{Simpson-reps1}. Importantly, $M$ is projective \cite[Theorem 1.21(2)]{Simpson-reps1}, thus separated; then $\phi_{\mc F}$ and $\phi_{\mc G}$ agree on some closed subscheme on $U$ \cite[Tag 01KM]{stacks-project}, which contains every closed point (Lemma \ref{lem: FHHO families to point}), and thus must be all of $U$, since $U$ is reduced. Furthermore, since $\phi_{\mc F} = \phi_{\mc G}$ lands in the open locus $M^s \subset M$ of relatively stable sheaves, by \cite[Theorem 1.21(4)]{Simpson-reps1} we see that there is an \'etale cover $a: U' \to U$ such that $a^* \mc F \cong a^* \mc G$. We conclude that $\pi_{U'*} \mc Hom(a^* \mc F, a^* \mc G) \cong \ms O_{U'}$ \cite[Lemma 4.6.3]{Huybrechts-Lehn} and the evaluation map
    \begin{equation*}
        a^* \mc F \otimes \pi_{U'}^* \pi_{U'*} \mc Hom(a^* \mc F, a^* \mc G) \to a^* \mc F \otimes \mc Hom(a^* \mc F, a^* \mc G) \to a^* \mc G
    \end{equation*}
    is an isomorphism. These properties descend to $U$, once we note the following: since $a^* = La^*$ we have
    $$\mc Hom(a^* \mc F, a^* \mc G) = \mc Hom_{U'}(La^* \mc F, La^* \mc G) = \mc H^0(R\mc Hom(La^* \mc F, La^* \mc G))$$
    $$= \mc H^0(La^*R\mc Hom(\mc F, \mc G)) = a^* \mc Hom(\mc F, \mc G)$$
    and so
    $$\pi_{U'*}\mc Hom(a^* \mc F, a^* \mc G) = \pi_{U'*} a^* \mc Hom(\mc F, \mc G) = a^* \pi_{U*} \mc Hom(\mc F, \mc G)$$
    with the last equality following from applying $\mc H^0$ to \cite[Tag 08IB]{stacks-project}. Thus, we conclude that $\pi_{U*} \mc Hom(\mc F, \mc G)$ is a line bundle and that there is an isomorphism
    \begin{equation} \label{eq: stable ev}
        \mc F \otimes \pi_U^* \pi_{U*} \mc Hom(\mc F, \mc G) \xrightarrow{\sim} \mc G.
    \end{equation}
    
    Finally, to identify the line bundle $\pi_U^* \pi_{U*} \mc Hom(\mc F, \mc G)$, we pull back along the section
    $$t: U \to \oJ^0(X/U) \times_U \J^0(\tilde X/U), u \mapsto (\ms O_{X_u}, \ms O_{\tilde X_u}).$$
    Note that $\mc P_{J/J}$ is trivial along $\{\ms O_{X_u}\} \times_U \oJ^0(X/U)$, so that we have
    \begin{equation} \label{eq: FHHO equality F}
    t^*(\mc F \otimes \pi_U^* \pi_{U*} \mc Hom(\mc F, \mc G)) = t^* \mc F \otimes \pi_{U*} \mc Hom(\mc F, \mc G) \cong  \pi_{U*} \mc Hom(\mc F, \mc G).  
    \end{equation}
    On the other hand, recall that $f$ is an isomorphism over $\J^0(X/U)$, and thus
    $$t^* \mc G = t^*(q_1^* \mc V \otimes (p \times 1)^* \tilde{\mc P}_{J/J}).$$
    Now, recall that $\tilde{\mc P}_{J/J}$ is trivial along $\J^0(\tilde X/U) \times_U \{\ms O_{\tilde X_u}\}$, and so $t^*(p \times 1)^* \tilde{\mc P}_{J/J}$ is trivial. For the same reason, we also have $\BP_i|_{\ms O_{\tilde X_u}} \cong \BP^1 \times U$ (see Definition \ref{def: PMod}) and the composition of $t$ with projection $U \xrightarrow{q_1 \circ t} \PMod \to \BP_i$ is given by $u \mapsto ([1:1], u)$, so that the pullback of $\ms O_{\BP_i}(-1)$, thus also $t^* q_1^* \mc V$, is trivial.
    We conclude that $t^* \mc G \cong \ms O_U$. Combining this with (\ref{eq: stable ev}) and (\ref{eq: FHHO equality F}), we see that $\pi_{U*} \mc Hom(\mc F, \mc G) = \ms O_U$ and $\mc F \cong \mc G$.
\end{proof}

Next, let $i: \oPrym(X/U) \hookrightarrow \oJ^0(X/U)$, respectively $\tilde i: \Prym(\tilde X/U) \hookrightarrow \J^0(\tilde X/U)$, be the inclusions. We write $\mc P_{P/J} := (i \times 1)^* \mc P_{J/J}$ and $\tilde{\mc P}_{P/J} := (\tilde i \times 1)^* \tilde{\mc P}_{J/J}$.

\begin{lem}
    We have an isomorphism
    \begin{equation} \label{eq: FHHO pre descent}
    (1 \times a)^* \mc Hom(\mc P_{P/J}, \ms O_{\oPrym \times \oJ^0}) \cong (f_P \times 1)_*(q_1^* (\omega_{\PMod_P/U} \otimes \mc V^\vee_P) \otimes (p_P \times 1)^* \tilde{\mc P}_{P/J}^\vee)
    \end{equation}
    of sheaves on $\oPrym(X/U) \times_U \J^0(\tilde X/U)$, where both sides are unchanged if any subset of the functors is taken to be derived.
\end{lem}

\begin{proof}
    We start by applying $(i \times 1)^*$ to the isomorphism $\mc F \cong \mc G$ of Lemma \ref{lem: FHHO relative J iso}. Noting that $i$ and the corresponding map $i_{\PMod}: \PMod_P \hookrightarrow \PMod$ (a flat base change of $\tilde i$) are both flat base changes of $\{\ms O_C\} \hookrightarrow \J^0(C)$ (see Remark \ref{rmk: i is reg}), we have $L(i \times 1)^* R(f \times 1)_* = R(f_P \times 1)_* L(i_{\PMod} \times 1)^*$ by Lemma \ref{lem: Tor indep}(b), where as usual, since $f, f_P$ are finite the derived pushforwards are underived. Thus
    $$(i \times 1)^* \mc G = \mc H^0(L(i \times 1)^* (f \times 1)_*(q_1^* \mc V \otimes (p \times 1)^* \tilde{\mc P}_{J/J})$$
    $$ = \mc H^0((f_P \times 1)_* L(i_{\PMod} \times 1)^*(q_1^* \mc V \otimes (p \times 1)^* \tilde{\mc P}_{J/J})) = (f_P \times 1)_* (q_1^* \mc V_P \otimes (p_P \times 1)^* \tilde{\mc P}_{P/J})$$
    where in the last equality we use that derived and underived pullback of a line bundle agree. So from Lemma \ref{lem: FHHO relative J iso} we get
    \begin{equation} \label{eq: FHHO restriction}
      (f_P \times 1)_* (q_1^* \mc V_P \otimes (p_P \times 1)^* \tilde{\mc P}_{P/J}) = (i \times 1)^* \mc G = (i \times 1)^* \mc F = (1 \times a)^* \mc P_{P/J}.  
    \end{equation}
    The needed identity will follow by taking the dual $R\mc Hom(-, \omega_{\oPrym \times_U \J^0}^\bullet)$ of each side of \ref{eq: FHHO restriction}) and tensoring by $\pi_U^* \omega_U^\vee[-e]$ for $e := \dim \oPrym(X/U) \times_U \J^0(\tilde X/U)$.

    On the left side, note that
    $$R\mc Hom((i \times 1)^* \mc G, \omega_{\oPrym \times \J^0}^\bullet) = R\mc Hom(R(f_P \times 1)_*(q_1^* \mc V_P \otimes (p_P \times 1)^* \tilde{\mc P}_{P/J}), \omega_{\oPrym \times_U \J^0}^\bullet)$$
    $$ = R(f_P \times 1)_* R\mc Hom(q_1^* \mc V_P \otimes (p_P \times 1)^* \tilde{\mc P}_{P/J}, \omega_{\PMod_P \times_U \J^0}^\bullet)$$
    by Grothendieck duality
    $$ = R(f_P \times 1)_* (q_1^* \mc V^\vee_P \otimes R\mc Hom((p_P \times 1)^* \tilde{\mc P}_{P/J}, \ms O_{\PMod_P \times \J^0}) \otimes \omega_{\PMod_P \times_U \J^0}[e])$$
    since $\PMod_P \times_U \J^0(\tilde X/U)$ is smooth
    $$=R(f_P \times 1)_* (q_1^* \mc V^\vee_P \otimes (p_P \times 1)^* \tilde{\mc P}_{P/J}^\vee \otimes \omega_{\PMod_P \times_U \J^0}[e])$$
    $$=R(f_P \times 1)_* (q_1^*(\omega_{\PMod/U} \otimes \mc V^\vee_P) \otimes (p_P \times 1)^* \tilde{\mc P}_{P/J}^\vee) \otimes \pi_U^* \omega_U[e]$$
    using the projection formula and the fact that $\omega_{\PMod_P \times_U \J^0} \cong q_1^*\omega_{\PMod_P/U} \otimes \pi_U^* \omega_U \otimes q_2^* \omega_{\J^0/U}$, where the third term is trivial by assumption on $U$
    $$=(f_P \times 1)_* (q_1^*(\omega_{\PMod/U} \otimes \mc V^\vee_P) \otimes (p_P \times 1)^* \tilde{\mc P}_{P/J}^\vee) \otimes \pi_U^* \omega_U[e]$$
    since $f_P$ is finite.

    On the right side, we have
    $$R\mc Hom((1 \times a)^* \mc P_{P/J}, \omega^\bullet_{\oPrym \times_U \J^0}) = R\mc Hom((1 \times a)^* \mc P_{P/J}, \ms O) \otimes \omega_{\oPrym \times_U \J^0}[e]$$
    since $\oPrym(X/U) \times_U \J^0(\tilde X/U)$ is Gorenstein (Lemma \ref{lem: oPrym is flat})
    $$ =  R\mc Hom(L(1 \times a)^* \mc P_{P/J}, \ms O) \otimes \omega_{\oPrym \times_U \J^0}[e]$$
    by \cite[Tag 0C0V]{stacks-project}, since $\mc P_{P/J}$ and $\oPrym(X/U) \times_U \J^0(\tilde X/U)$ are flat over the second factor by \cite[Lemma 4.3(b)]{GS} and Lemma \ref{lem: oPrym is flat}, respectively
    $$ = L(1 \times a)^* R\mc Hom(\mc P_{P/J}, \ms O) \otimes \omega_{\oPrym \times_U \J^0}[e]$$
    $$ = L(1 \times a)^* R\mc Hom(\mc P_{P/J}, \ms O) \otimes \pi_U^*\omega_U[e]$$
    since by assumption $\omega_{\oPrym/U}$ and $ \omega_{\J^0/U}$ are both trivial (see Lemma \ref{lem: oPrym is flat})
    $$ = L(1 \times a)^* \mc Hom(\mc P_{P/J}, \ms O) \otimes \pi_U^*\omega_U[e]$$
    since $\mc P_{P/J}$ is maximal Cohen-Macaulay by \cite[Lemma 2.3(b)]{Arinkin}
    $$ = (1 \times a)^* \mc Hom(\mc P_{P/J}, \ms O) \otimes \pi_U^*\omega_U[e]$$
    because by the previous paragraph, the result of applying the duality functor to (\ref{eq: FHHO restriction}) is an $e$-shifted sheaf.
\end{proof}

It remains to descend the isomorphism of (\ref{eq: FHHO pre descent}) from $\oPrym(X/U) \times_U \J^0(\tilde X/U)$ to $\oPrym(X/U) \times_U [\J^0(\tilde X/U)/\J^0(C)]$. 

\begin{proof}[Proof of Proposition \ref{prop: smooth and singular Poincare}]
    Let $\tilde \varpi$, respectively $\varpi$, represent all base changes of the $\J^0(C)$-quotient map $\J^0(\tilde X/U) \to [\J^0(\tilde X/U)/\J^0(C)]$, respectively $\oJ^0(X/U) \to [\oJ^0(X/U)/\J^0(C)]$. We note that 
    $$\tilde \varpi^* (1 \times a_P)^* \mc P^{-1} := \tilde \varpi^* (1 \times a_P)^* R\mc Hom(\mc P^T, \ms O)[h] = (1 \times a)^* \varpi^* R\mc Hom(\mc P^T, \ms O)[h]$$
    $$=(1 \times a)^* R\mc Hom(\varpi^*\mc P^T, \ms O)[h]$$
    since $\varpi$ is flat
    $$=(1 \times a)^* R\mc Hom(\mc P_{P/J}, \ms O)[h]$$
    by definition of $\mc P$, so that we recover the left side of (\ref{eq: FHHO pre descent}), shifted by $h$. Similarly,
    $$\tilde \varpi^*(f_P \times 1)_*(q_1^*(\omega_{\PMod_P} \otimes \mc V_P^\vee) \otimes (p_P \times 1)^* \tilde{\mc P}^{-1})$$
    $$:= \tilde \varpi^*(f_P \times 1)_*(q_1^*(\omega_{\PMod_P} \otimes \mc V_P^\vee) \otimes (p_P \times 1)^* (\tilde{\mc P}^T)^\vee)[h-k]$$
    $$= (f_P \times 1)_*\tilde{\varpi}^*(q_1^*(\omega_{\PMod_P} \otimes \mc V_P^\vee) \otimes (p_P \times 1)^* (\tilde{\mc P}^T)^\vee)[h-k]$$
    because $\varpi$ is flat and $f_P$ finite
    $$=(f_P \times 1)_*(q_1^*(\omega_{\PMod_P} \otimes \mc V_P^\vee) \otimes (p_P \times 1)^* (\tilde{\varpi}^*\tilde{\mc P}^T)^\vee)[h-k]$$
    $$=(f_P \times 1)_*(q_1^*(\omega_{\PMod_P} \otimes \mc V_P^\vee) \otimes (p_P \times 1)^* \tilde{\mc P}_{P/J}^\vee)[h-k],$$
    matching the (shifted) right side of (\ref{eq: FHHO pre descent}). Thus, it suffices to check that there is at most one way to descend the sheaf (\ref{eq: FHHO pre descent}) to $\oPrym(X/U) \times_U [\J^0(\tilde X/U)/\J^0(C)]$.

    To ease notation, write
    $$\ms L := q_1^*(\omega_{\PMod_P/U} \otimes \mc V^\vee_P) \otimes (p_P \times 1)^* \tilde{\mc P}_{P/J}^\vee \in \Pic(\PMod \times_U \J^0(\tilde X/U)).$$ Consider the sheaf $\mc End(p_{23}^*(f_P \times 1)_* \ms L)$ on $\J^0(C) \times \oPrym(X/U) \times_U \J^0(\tilde X/U)$. As usual, since $p_{23}$ is flat, we have $p_{23}^* (f_P \times 1)_* \ms L = (1 \times f_P \times 1)_* p_{23}^* \ms L$. Now, since $(1 \times f_P \times 1)$ is finite, therefore in particular affine, note that the adjunction
    $$(1 \times f_P \times 1)^*(1 \times f_P \times 1)_* p_{23}^* \ms L \to p_{23}^* \ms L$$
    is surjective; moreover, the kernel is torsion since $(1 \times f_P \times 1)$ is an isomorphism over a dense open subset. Since $p_{23}^* \ms L$ is a line bundle and thus torsion-free, we conclude that the map
    $$\mc Hom(p_{23}^* \ms L, p_{23}^* \ms L) \to \mc Hom((1 \times f_P \times 1)^*(1 \times f_P \times 1)_* p_{23}^* \ms L, p_{23}^* \ms L)$$
    is an isomorphism, and so
    $$\mc End(p_{23}^* (f_P \times 1)_* \ms L) = \mc Hom((1 \times f_P \times 1)_* p_{23}^* \ms L, (1 \times f_P \times 1)_* p_{23}^* \ms L)$$
    $$= (1 \times f_P \times 1)_* \mc Hom((1 \times f_P \times 1)^*(1 \times f_P \times 1)_* p_{23}^* \ms L, p_{23}^* \ms L)$$
    $$= (1 \times f_P \times 1)_* \mc Hom(p_{23}^* \ms L, p_{23}^* \ms L) = (1 \times f_P \times 1)_* \ms O_{\J \times \PMod_P \times_U \J}.$$
    Taking global sections, we see that
    $$\mathrm{End}(p_{23}^*(f_P \times 1)_* \ms L) = H^0(\J^0(C) \times \PMod_P \times_U \J^0(\tilde X/U), \ms O) = H^0(\PMod_P \times_U \J^0(\tilde X/U), \ms O)$$
    and in particular, an isomorphism of $p_{23}^*(f_P \times 1)_* \ms L$ is uniquely determined by its value on $\{\ms O_C\} \times \oPrym(X/U) \times_U \J^0(\tilde X/U)$. Since a $\J^0(C)$-equivariant structure on the sheaf (\ref{eq: FHHO pre descent}) is given by an isomorphism of sheaves $p_{23}^* (f_P \times 1)_* \ms L$ and $\sigma^* (f_P \times 1)_* \ms L$ (where $\sigma$ is the action map) satisfying the cocycle condition, which implies in particular that the isomorphism $(p_{23}^*(f_P \times 1)_* \ms L)|_{\{\ms O_C\}} \cong (\sigma^*(f_P \times 1)_* \ms L)|_{\{\ms O_C\}}$ is given by the identification of each side with $(f_P \times 1)_* \ms L$ under pullback, we conclude that a $\J^0(C)$-equivariant structure must be unique.
\end{proof}

Finally, as a last example of the techniques of this section, we return to an assertion made in the proof of Proposition \ref{prop: GS}:

\begin{prop} \label{prop: GS descent}
    The restriction of the Poincar\'e sheaf $\mc P_{J/J}$ to $\oJ^0(X/B) \times_B \oPrym(X/B)$ descends uniquely to $[\oJ^0(X/B)/\J^0(C)] \times_B \oPrym(X/B)$. 
\end{prop}

\begin{proof}
    First, for each $L \in \J^0(C)$, acting on the first factor of $\oJ^0(X/B) \times_B \oPrym(X/B)$ by tensor product with the pullback (call this action $\sigma$), we have $\sigma(L,-)^* \mc P_{J/J} \cong \mc P_{J/J} \otimes p_2^*(\mc P_{J/J}|_{\{c^*L\} \times J})$ by \cite[Lemma 6.5]{Arinkin}. We note that $\mc P_{J/J}|_{\{c^*L\} \times \oPrym}$ is trivial, by \cite[Proposition 4.4]{GS} applied to $\mc M = p_{13}^* \mc L$ and $\mc N = \ms O_{X \times \oJ \times \oJ}$, since after restricting to $X \times_B \oJ^0(X/B) \times_B \oPrym(X/B)$ one has $\det c_*(\mc M) \cong \det c_*(\mc N)$ by definition of $\oPrym$. Thus the pullbacks of $\mc P_{J/J}|_{\oJ \times \oPrym} =: \mc P_{J/P}$ under the action $\sigma$ and projection $p_{23}:  \J^0(C) \times \oJ^0(X/B) \times_B \oPrym(X/B) \to \oJ^0(X/B) \times_B \oPrym(X/B)$ are isomorphic over each point $L \in \J^0(C)$.

    Now, consider the projection $p_{13}: \J^0(C) \times \oJ^0(X/B) \times_B \oPrym(X/B) \to \J^0(C) \times \oPrym(X/B)$. We note that $\sigma^* \mc P_{J/P}, p_{23}^* \mc P_{J/P}$ are relatively flat with respect to $p_{13}$ since $\mc P_{J/J}$ is flat over each factor. Moreover, for each $x \in \J^0(C) \times \oPrym(X/B)$, the fibers $\sigma^* \mc P_{J/P}|_x \cong p_{23}^* \mc P_{J/P}|_x$ are torsion-free (because maximal Cohen-Macaulay \cite[Theorem A(2)]{Arinkin}) of rank 1 on the irreducible variety $\oJ^0(X_{\pi(x)})$ \cite[Theorem 9]{AIK}, thus stable with respect to any polarization. The argument in the proof of Lemma \ref{lem: FHHO relative J iso} then implies that $\sigma^* \mc P_{J/P}$ and $p_{23}^* \mc P_{J/P}$ differ by the line bundle $p_{13}^* p_{13*} \mc Hom(\sigma^* \mc P_{J/P}, p_{23}^* \mc P_{J/P})$.

    To identify this line bundle, we pull back along the section $t$ of $p_{13}$ given by the section $B \to \oJ^0(X/B), b \mapsto \{\ms O_{X_b}\}$. Since this is a section of $\J^0(X/B) \subset \oJ^0(X/B)$, we note that $t^* \sigma^* \mc P_{J/P}$ and $t^*p_{23}^* \mc P_{J/P} = p_2^*\mc P_{J/P}|_{\{\ms O\} \times P} \cong \ms O_{\J^0(C) \times \oPrym}$ are line bundles, computable by the determinant of cohomology formula \cite[(1.1)]{Arinkin}. In order to calculate $t^* \sigma^* \mc P_{J/P}$, we use \cite[Proposition 1.5]{Huishi}, whose proof also holds when $X \to C \times B$ is a flat finite map but not an \'etale double cover (although in general one must replace $\mc D_{q_{23}}(q_{12}^* K|_{D_Y})$ by $\mc D_{q_{23}}(q_{12}^* K|_{D_Y})^{\otimes (n-1)}$, where $n$ is the degree). Applying this over the base $\J^0(C) \times \oPrym(X/B)$ with $K$ on $C \times \J^0(C) \times \oPrym(X/B)$ pulled back from a universal bundle on $C \times \J^0(C)$, $\mc M$ pulled back from a universal bundle on $X \times_B \oPrym(X/B)$, and $\mc N$ trivial, since $\det c_*(\mc M) \cong \det c_*(\mc N)$, we conclude that $t^*\sigma^*\mc P_{J/P} = \mc P|_{\J^0(C) \times \oPrym}$ is trivial. In other words, since the line bundles $t^* \sigma^* \mc P_{J/P} \cong t^*p_{23}^* \mc P_{J/P}$ differ by the line bundle $t^* p_{13}^* p_{13*} \mc Hom(\sigma^* \mc P_{J/P}, p_{23}^* \mc P_{J/P}) = p_{13*}\mc Hom(\sigma^* \mc P_{J/P}, p_{23}^* \mc P_{J/P})$, we conclude that this is trivial, hence in fact $\sigma^* \mc P_{J/P} \cong p_{23}^* \mc P_{J/P}$.

    Now we are ready to define the equivariant structure on $\mc P_{J/P}$. Since from the previous paragraph we have $p_{13*} \mc Hom(\sigma^* \mc P_{J/P}, p_{23}^* \mc P_{J/P}) \cong \ms O_{\J \times \oPrym}$, in particular
    $$\Hom(\sigma^* \mc P_{J/P}, p_{23}^* \mc P_{J/P}) \cong H^0(\J^0(C) \times \oPrym(X/B), \ms O) = H^0(\oPrym(X/B), \ms O)$$
    and so a morphism is uniquely determined by its value on the section $s: \oPrym(X/B) \to \J^0(C) \times \oJ^0(X/B) \times_B \oPrym(X/B)$ given by $\{\ms O_C\} \times \{\ms O_{X_b}\} \times_B \textrm{id}$. Let $\phi: \sigma^* \mc P_{J/P} \to p_{23}^* \mc P_{J/P}$ be the unique morphism restricting to the identity on
    $$s^* \sigma^* \mc P_{J/P} = s^* p_{23}^* \mc P_{J/P} = \mc P_{J/P}|_{\ms O \times P} \cong \ms O_{\oPrym},$$
    as required by the cocycle condition. Note that $\phi$ is an isomorphism, as it does not have zeroes on $s$.

    Finally, we claim that $\phi$ must satisfy the cocycle condition, i.e., equality of the two isomorphisms $(m \times 1)^* \phi$ and $p_{234}^* \phi \circ (1 \times \sigma)^* \phi$ on the sheaf isomorphic to $p_{34}^* \mc P_{P/J}$ for the projection $p_{34}: \J^0(C) \times \J^0(C) \times \oJ^0(X/B) \times_B \oPrym(X/B): \oJ^0(X/B) \times_B \oPrym(X/B)$. Indeed, repeating the argument above, we see that
    $$\mathrm{End}(p_{34}^* \mc P_{P/J}) \cong H^0(\J^0(C) \times \J^0(C) \times \oPrym(X/B), \ms O) = H^0(\oPrym(X/B), \ms O)$$
    and so it suffices that the two isomorphisms agree on the section $\{\ms O_C\} \times \{\ms O_C\} \times \{\ms O_{X_b}\} \times_B \textrm{id}$, which is true by choice of $\phi$.
\end{proof}

\subsection{Proof of Theorem \ref{thm: main sec 2}(ii)} \label{subsec: proof of realization}
We are now ready to describe the homological realization of the projectors $\pg_{\gamma,k}, \qg_{\gamma,k}$. In \S \ref{subsubsec: breaking into chars} we see that $\mf G_\gamma, \mf G_\gamma^{-1}$ only interact with the isotypic component $(\pi_* \BC_M)_\kappa$, where $\gamma \in \Gamma$ and $\kappa \in \Gamma^\vee$ are paired as in Definition \ref{def: Weil pairing}. Then in \S \ref{subsubsec: reducing to U} we explain, one element $\gamma \in \Gamma$ at a time, how to reduce the problem to a suitable \'etale neighborhood of $B_\gamma$ on which to apply the Poincar\'e sheaf comparison result of Proposition \ref{prop: smooth and singular Poincare}, reducing the problem to the smooth case already treated in Proposition \ref{prop: (ii) smooth}.

\subsubsection{Matching isotypic components to components of the inertia stack} \label{subsubsec: breaking into chars}
Given the corresponding result in the smooth case (Proposition \ref{prop: (ii) smooth}) and the comparison from Proposition \ref{prop: smooth and singular Poincare}, it will quickly follow that $\mf G_\gamma$, respectively $\mf G_\gamma^{-1}$, only induce nontrivial morphisms to, respectively from, $(\pi_* \BC_M)_\kappa \subset \pi_* \BC_M$.

\begin{lem} \label{lem: theorem of square}
    Write $M^\circ := \Prym(X/B) \subset M$, and let $\mu: M \times_B M^\circ \to M$ be the multiplication map. Then
    $$(1 \times \mu)^* \mc P = p_{12}^* \mc P \otimes p_{13}^* \mc P|_{M^\vee \times_B M^\circ} \in \Coh(M^\vee \times_B M \times_B M^\circ).$$
\end{lem}

\begin{proof}
    Writing $\mc P_{J/J}$ for the Arinkin Poincar\'e sheaf on $\oJ^0(X/B) \times_B \oJ^0(X/B)$, the corresponding identity holds on $\oJ^0(X/B) \times_B \oJ^0(X/B) \times_B \J^0(X/B)$ by \cite[Lemma 6.5]{Arinkin}. Restricting to $\oJ^0(X/B) \times_B \oPrym(X/B) \times_B \Prym(X/B)$, we obtain an identity
    $$(1 \times \mu)^* \mc P_{J/M} = p_{12}^* \mc P_{J/M} \otimes p_{13}^* \mc P_{J/M^\circ} \in \Coh(\oJ^0(X/B) \times_B M \times_B M^\circ)$$
    (noting that $\mc P_{J/J^\circ}$ is a line bundle). As in the proof of Proposition \ref{prop: GS descent}, we note that this sheaf descends uniquely to $M^\vee \times_B M \times_B M^\circ$ and thus that the two possible descents $(1 \times \mu)^* \mc P$ and $p_{12}^* \mc P \otimes p_{13}^* \mc P|_{M^\vee \times_B M^\circ}$ are isomorphic.
\end{proof}

\begin{prop} \label{prop: Ggamma in the kappa cpt}
    For any $\gamma \in \Gamma$, we have
    $$cl(\mf G_\gamma) \in H_*^{BM}(M^{\gamma, \vee} \times_B M)_\kappa, \,\,\, cl(\mf G_\gamma^{-1}) \in H^{BM}_*(M \times_B M^{\gamma, \vee})_{\kappa^{-1}},$$
    where $\Gamma$ acts on $M^{\gamma, \vee} \times_B M$,  respectively $M \times_B M^{\gamma, \vee}$, via its action on $M$.
\end{prop}

\begin{proof}
    For any $\beta \in \Gamma$, let $\sigma_\beta: M \to M$ be the action map. By Lemma \ref{lem: theorem of square} we have
    \begin{equation} \label{eq: pullback of P}
        \sigma_\beta^* \mc P = \mc P \otimes p_1^*\mc P|_{M^\vee \times \{\beta\}}, \,\,\, \sigma_\beta^* \mc P^{-1} = \mc P^{-1} \otimes p_2^*\mc P|_{M^\vee \times \{\beta\}}^\vee
    \end{equation}
    and thus
    $$\beta \cdot cl(\mf G_\gamma) = cl(\tch(p_1^*\mc P|_{M^\vee \times \{\beta\}})_\gamma \cap \mf G_\gamma), \,\,\, \beta \cdot cl(\mf G_\gamma^{-1}) = cl(\tch(p_2^*\mc P|_{M^\vee \times \{\beta\}}^\vee)_\gamma \cap \mf G_\gamma^{-1}).$$
    Since $\mc P|_{M^\vee \times \beta} \in \Pic(M^\vee)$ is a line bundle, the decomposition into $\gamma$-eigensheaves is trivial, and by Corollary \ref{cor: eigenvalue of P} and Proposition \ref{prop: smooth and singular Poincare} (applied, say, over a single point $b \in B_\gamma$), we see that the eigenvalue of $\gamma$ is $\kappa(\beta)$. As in the proof of Lemma \ref{lem: Palpha top trivial}, we see that $\mc P_{M^\vee \times \{\beta\}}$ is topologically trivial, and thus
    $$\beta \cdot cl(\mf G_\gamma) = \tch^{top}(p_1^*\mc P|_{M^\vee \times \{\beta\}})_\gamma \cap cl(\mf G_\gamma) = \kappa(\beta) \cdot cl(\mf G_\gamma), \,\,\, \beta \cdot cl(\mf G_\gamma^{-1}) = \kappa(\beta)^{-1} \cdot cl(\mf G_\gamma^{-1}).$$
\end{proof}

\begin{cor} \label{cor: Ggamma triv on other pieces}
    Under the homological realization isomorphism (\ref{eq: BM = sheaf Hom}) the morphism
    $$cl(\mf G_{\gamma, i}) \in \Hom_B(\pi_* \BC_{M^{\gamma, \vee}}, \pi_* \BC_M[2(i-h_\gamma)])$$
    factors through the (shifted) inclusion $(\pi_* \BC_M)_\kappa \subset \pi_* \BC_M$, and
    $$cl(\mf G_{\gamma, i}^{-1}) \in \Hom_B(\pi_* \BC_M,\pi_* \BC_{M^{\gamma, \vee}}[2(i-2h+h_\gamma)])$$
    factors through the projection $\pi_* \BC_M \twoheadrightarrow (\pi_* \BC_M)_\kappa$.
\end{cor}

\begin{proof}
    It follows from Proposition \ref{prop: Ggamma in the kappa cpt} that each degree component $\mf G_{\gamma,i}$, respectively $\mf G_{\gamma,i}^{-1}$, is in the $\kappa$, respectively $\kappa^{-1}$, component of the relevant $H^{BM}_*$. Since $\Gamma$ acts trivially on $\pi_* \BC_{M^{\gamma, \vee}}$, note that a $\kappa$-component morphism must have image in $(\pi_* \BC_M[*])_\kappa$ (for $\mf G_{\gamma, i}$), and in the opposite direction, the image of a $\kappa'$-component subsheaf by a $\kappa^{-1}$ morphism cannot be $\Gamma$-invariant unless $\kappa' = \kappa$ (for $\mf G_{\gamma, i}^{-1}$).
\end{proof}

It follows from the above that for each $\gamma \ne \beta \in \Gamma$, we have $cl(\mf G_{\gamma,i}^{-1}) \circ cl(\mf G_{\beta,j}) = 0$. For future use, we record the fact that this is already true on the level of Chow groups:

\begin{prop} \label{prop: FV for different chars}
    For any $i,j \in \BN$ and $\gamma \ne \beta \in \Gamma$, we have
    $$\mf G_{\gamma,i}^{-1} \circ \mf G_{\beta,j} = 0 \in \Corr_B^{i+j-2h}(M^{\beta, \vee}, M^{\gamma, \vee}).$$
\end{prop}

\begin{proof}
    For any $\alpha \in \Gamma$ acting on $M$ via an isomorphism $\sigma_\alpha$, we note that
    $$\mf G_{\gamma,i}^{-1} \circ \mf G_{\beta,j} = (\sigma_\alpha \times 1)^* \mf G_{\gamma,i}^{-1} \circ (1 \times \sigma_\alpha)^* \mf G_{\beta,j}.$$
    By definition and (\ref{eq: pullback of P}) we see that
    $$(1 \times \sigma_\alpha)^* \mf G_\beta = \tch(p_1^* \mc P|_{M^\vee \times \{\alpha\}})_\beta \cap \mf G_\beta, \,\,\, (\sigma_\alpha \times 1)^* \mf G_\gamma^{-1} = \tch(p_2^* \mc P|^\vee_{M^\vee \times \{\alpha\}})_\gamma \cap \mf G_\gamma^{-1}.$$
    In particular, splitting into degree pieces
    $$(1 \times \sigma_\alpha)^* \mf G_{\beta,j} = ((1 \times \sigma_\alpha)^* \mf G_{\beta})_j = \sum_{a \ge 0} \tch^a(p_1^* \mc P|_{M^\vee \times \{\alpha\}})_\beta \cap \mf G_{\beta, j-a}$$
    and similarly for $\mf G_{\gamma,i}^{-1}$.
    Then by the projection formula for the representable morphism $p_{13}: M^{\beta, \vee} \times_B M \times_B M^{\gamma, \vee} \to M^{\beta, \vee} \times_B M^{\gamma, \vee}$ we have
    $$\mf G_{\gamma,i}^{-1} \circ \mf G_{\beta,j} = (\sigma_\alpha \times 1)^* \mf G_{\gamma,i}^{-1} \circ (1 \times \sigma_\alpha)^* \mf G_{\beta,j}$$
    $$=\sum_{a,b \ge 0} \tch^a(p_2^* \mc P|^\vee_{M^\vee \times \{\alpha\}})_\gamma \cap \tch^b(p_1^* \mc P|_{M^\vee \times \{\alpha\}})_\beta \cap (\mf G_{\gamma, i-a}^{-1} \circ \mf G_{\beta,j-b}).$$
    
    By induction, we may suppose that for all pairs $(i',j') < (i,j)$ in lexicographic order we have $\mf G_{\gamma,i'}^{-1} \circ \mf G_{\beta, j'} = 0$. Then
    \begin{equation} \label{eq: GV diff cpts}
        \mf G_{\gamma,i}^{-1} \circ \mf G_{\beta, j} = \tch^0(p_2^* \mc P|^\vee_{M^\vee \times \{\alpha\}})_\gamma \cap \tch^0(p_1^* \mc P|_{M^\vee \times \{\alpha\}})_\beta \cap (\mf G_{\gamma, i}^{-1} \circ \mf G_{\beta,j}).
    \end{equation}
    As in the proof of Proposition \ref{prop: Ggamma in the kappa cpt}, we note that
    $$\tch^0( \mc P|^\vee_{M^\vee \times \{\alpha\}})_\gamma  = \langle \gamma, \alpha \rangle^{-1}, \,\,\, \tch^0( \mc P|_{M^\vee \times \{\alpha\}})_\beta  = \langle \beta, \alpha \rangle.$$
    Since $\gamma \ne \beta$ and the pairing of Definition \ref{def: Weil pairing} is nondegenerate, we may choose $\alpha$ with $\langle \gamma^{-1}\beta, \alpha \rangle \ne 1$. Then (\ref{eq: GV diff cpts}) implies that $\mf G_{\gamma, i}^{-1} \circ \mf G_{\beta,j}=0$.
\end{proof}

\subsubsection{Reduction to the smooth case} \label{subsubsec: reducing to U}
Let $X \to B$ be a family of integral curves with a morphism to $C$ inducing a good compactified Prym fibration $M = \oPrym(X/B) \to B$. Choose an element $\gamma \in \Gamma$.

\begin{prop} \label{prop: U exists}
    There is an \'etale morphism $U \to B_\gamma$ which can be equipped with a family $\tilde X \to U$ normalizing $X_U \to U$ and sections $\{\sigma_i, \sigma_i^1, \sigma_i^2, s: 1 \le i \le k\}$, satisfying all the properties described at the beginning of \S \ref{subsubsec: FHHO setup}.
\end{prop}

\begin{proof}
    By Definition \ref{defn: good fibration}(3), there is a dense open subset $U' \subset B_\gamma$ over which the curves $X_{U'} \to U'$ are nodal. Restricting further if necessary, by upper semicontinuity of the $\delta$-invariant, we may assume that the number of nodes is some constant $k$. Then there exists a flat family $\tilde X_{U'} \to U'$ of normalizations \cite[Theorem 4.2]{normalization}. Restricting further if needed, we may assume that the line bundles $\det \pi_{U'*} \omega_{\tilde X/U'}, \det \pi_{U'*} \omega_{X/U'}$ are trivial. (Note that since the singularities of $X_{U'} \to U'$ are nodal, this is a family of Gorenstein curves \cite[Tag 0E37]{stacks-project}, and thus $\omega_{X'/U'}$ is a line bundle compatible with base change \cite[Tag 0E6R]{stacks-project}, and by Grauert's theorem $\pi_{U'*} \omega_{X'/U'}$ is a vector bundle.) We will take $U$ to be an \'etale cover of $U'$ (which is smooth by Definition \ref{defn: good fibration}(1)).
    
    Note that the discriminant locus $D$ of $X_{U'} \to U'$ is finite \'etale over $U'$ of order $k$; after passing to a $(k!)$-fold finite \'etale cover $U'' \to U'$, we can trivialize the cover $D \to U'$, i.e., define sections $\sigma_i$. Then $\nu^{-1}(\sigma_i)$ gives a double \'etale cover of $U''$ for each $i$, and trivializing these $k$ covers via some finite \'etale $U''' \to U''$, we obtain sections $\{\sigma_i^1, \sigma_i^2\}$. Finally, after taking an \'etale cover $U$ of $U'''$, we can define a section of the smooth morphism $\tilde X_{U'''} \setminus \cup_i (\sigma_i^1 \cup \sigma_i^2) \to U'''$.
\end{proof}

\begin{prop} \label{prop: a is restriction to Mgamma}
    Let $U$ be as in Proposition \ref{prop: U exists}. The morphism $a_P: [\J^0(\tilde X/U)/\J^0(C)] \to [\oJ^0(X/U)/\J^0(C)]$ is a closed embedding whose image is the fixed locus $[\oJ^0(X/U)^\gamma/\J^0(C)]$.
\end{prop}

\begin{proof}
    We start with the claim that $a_P$ (or equivalently $a: \J^0(\tilde X/U) \to \oJ^0(X/U)$) is a closed embedding. Consider the closed subvariety $\oJ^0(X/U)_{|k|} \subset \oJ^0(X/U)$ consisting of sheaves $\mc F$ which fail to be locally free at each node of $X_u$. (This is the image of the closed subvariety of $\PMod = \BP_1 \times_{\J^0(\tilde X/U)} \cdots \times \BP_k$ consisting of points such that the coordinate in each $\BP_i$ belongs to one of the two distinguished sections of $\BP_i \to \J^0(\tilde X/U)$.) It follows from \cite[Lemma 3.7(3-4)]{FHHO} that this is the image of $a$. Now, note that for $\ms L \in \J^0(\tilde X_u)$, the adjunction map $\nu_u^* \nu_{u*} \ms L \to \ms L$ is surjective since $\nu_u$ is finite, thus affine, and the kernel is the torsion subsheaf of $\nu_u^* \nu_{u*} \ms L$. In particular, we can define an inverse morphism
    $$b: \oJ^0(X/U)_{|k|} \to \J^0(\tilde X/U), \mc F \mapsto (\nu_u^* \mc F/T_{\nu_u^*\mc F})(ks(u))$$
    where $T_{\nu_u^* \mc F}$ is the torsion subsheaf. (More precisely, since $a$ and $b$ are inverse on closed points, $\oJ^0(X/U)_{|k|}$ is connected over each connected component of $U$, and $\J^0(\tilde X/U)$ is a smooth variety, it follows from Zariski's main theorem that $b$ is an isomorphism.)

    Now, recall from Remark \ref{rmk: endoscopic locus defs} that the image of $a$ is contained in the fixed locus $\oJ^0(X/U)^\gamma$. To show the other containment, for each subset $I \subset \{1, \ldots, k\}$, we consider the locally closed subvariety $\oJ^0(X/U)_I \subset \oJ^0(X/U)$ consisting of sheaves which fail to be locally free exactly at the nodes $\{\sigma_i: i \in I\}$. If $\nu_I: X_{u,I} \to X_u$ is the partial normalization resolving the nodes $\{\sigma_i: i \in I\}$, then by the arguments of the previous paragraph, pushforward gives an isomorphism $\nu_{I*}: \J^{-\# I}(X_{u,I}) \to \oJ^0(X_u)_I$; in particular $\dim \oJ^0(X_u)_I = \dim \J^0(X_{u,I}) = \dim \J^0(X_u) - \#I$. Note that $\nu_{I*}$ is $\J^0(C)$-equivariant by the projection formula, and that $\gamma \in \J^0(C)$ acts on $\J^{-\# I}(X_{u,I})$ without fixed points if the pullback $\nu_I^* a_u^* \gamma \in \J^0(X_{u,I})$ is nontrivial, and otherwise trivially. Thus $\oJ^0(X_u)^\gamma$ is of the form $\coprod_{I \in \mc S} \nu_{I*} \J^{-\# I}(X_{u,I})$ for some subset $\{|k|\} \subset \mc S \subset 2^{2^{|k|}}$.

    We now claim that smoothness of $M$ (Definition \ref{defn: good fibration}(1)) implies that $\mc S = \{|k|\}$. Indeed, if $I \subsetneq |k|$ were minimal such that $I \in \mc S$, then we note that for each $u \in U$ there would be a proper birational morphism
    $$\PMod_{X_{u,I}} \to \oJ^0(X_{u,I}) \cong \oJ^{-\# I}(X_{u,I}) \xrightarrow{\nu_{I*}} \oJ^0(X_u)^\gamma$$
    onto an irreducible component of $\oJ^0(X_u)^\gamma$ with some nontrivially disconnected fibers, and then by Zariski's main theorem $\oJ^0(X_u)^\gamma$ would not be smooth. On the other hand, as in Remark \ref{rmk: etale cover}, smoothness of $M$ implies smoothness of $\oJ^0(X/B)$, thus of the fixed locus $\oJ^0(X/B)^\gamma$, thus of the \'etale cover $\oJ^0(X_U/U)^\gamma$, and thus of a general fiber $\oJ^0(X_u)^\gamma$. Thus $\oJ^0(X_u)^\gamma = \nu_{u*} \J^{-k}(\tilde X_u) = a_u(\J^0(\tilde X_u))$.
\end{proof}

\begin{cor} \label{cor: h-hgamma}
    If $M \to B$ is a good Prym fibration, we have $h-h_\gamma = b-b_\gamma$.
\end{cor}

\begin{proof}
    By Proposition \ref{prop: a is restriction to Mgamma} we have $h_\gamma = \dim \Prym(\tilde X_u)$ for any $u \in U$. Then
    $$h := \dim \oPrym(X_u) = \dim \PMod_u = \dim \Prym(\tilde X_u) + k = h_\gamma + k,$$
    where by definition $k = \delta(B_\gamma)$. Now recall from Remark \ref{rmk: Severi/Ngo} that $\delta(B_\gamma) = b-b_\gamma$.
\end{proof}

\begin{prop} \label{prop: tau B tau U}
    Let $\alpha: U \to B_\gamma$ be as in Proposition \ref{prop: U exists}. We write $\mc P_U^{-1}$, respectively $\mc P^{-1}$, for the Poincar\'e sheaves defined in Proposition \ref{prop: GS} on $\oPrym(X/U) \times_U [\oPrym(X/U)/\Gamma]$, respectively $M \times_B M^\vee$. Then
    \begin{equation} \label{eq: prop restrict}
        \alpha^! \tau(\mc P^{-1})_\gamma = (1 \times \alpha)^*\ttd(-T_{1 \times \iota_{\gamma}})_\gamma \cap \tau((1 \times \iota_{\gamma,U})^*\mc P^{-1}_U)_\gamma
    \end{equation}
    in $\CH_*(\oPrym(X/U) \times_U [\oPrym(X/U)^\gamma/\Gamma]),$
    where $\iota_{\gamma,U}: [\oPrym(X/U)^\gamma/\Gamma] \hookrightarrow [\oPrym(X/U)/\Gamma]$ is the inclusion.
\end{prop}

\begin{proof}
    Consider the composition
    $$\phi: \oPrym(X/U) \times_U [\Prym(X/U)^\gamma/\Gamma] \xhookrightarrow{1 \times \iota_{\gamma,U}} \oPrym(X/U) \times_U [\oPrym(X/U)/\Gamma]$$
    $$= (M \times_B M^\vee) \times_B U \xrightarrow{1 \times \alpha} M \times_B M^\vee$$
    or equivalently,
    $$\phi: \oPrym(X/U) \times_U [\Prym(X/U)^\gamma/\Gamma] = (M \times_B M^{\gamma,\vee}) \times_B U \xrightarrow{1 \times \alpha} M \times_B M^{\gamma,\vee} \xhookrightarrow{1 \times \iota_\gamma} M \times_B M^\vee.$$
    We claim first that $\mc P_U^{-1} = L(1 \times \alpha)^* \mc P^{-1}$: indeed, recall that $\mc P_{J/J}$ is maximal Cohen-Macaulay \cite[Theorem A(2), Lemma 2.1(1)]{Arinkin}; then since $i \times i: \oPrym(X/B) \times_B \oPrym(X/B) \hookrightarrow \oJ^0(X/B) \times_B \oJ^0(X/B)$ is a flat base change of the regular embedding $\{\ms O_C\} \times \{\ms O_C\} \hookrightarrow \J^0(C) \times \J^0(C)$, we conclude that $(i \times i)^* \mc P_{J/J}$ is maximal Cohen-Macaulay \cite[Lemma 2.3(2)]{Arinkin}; finally, since $\mc P$ pulls back to $(i \times i)^* \mc P_{J/J}$ under the \'etale map $M \times_B M \to M \times_B M^\vee$, we conclude that $\mc P$ is also maximal Cohen-Macaulay, and thus the shifted sheaf $\mc P^{-1}$ (the Grothendieck dual, up to shift and the line bundle $\omega_B$) is as well. The same argument applies to $\mc P_U^{-1}$. In particular, since $\mc P_U^{-1}$ and $L(1 \times \alpha)^* \mc P^{-1}$ are maximal Cohen-Macaulay extensions (the second assertion being \cite[Lemma 2.3]{Arinkin}) of the same line bundle defined away from codimension 2, by definition of $\mc P, \mc P_U$, we conclude from \cite[Lemma 2.2]{Arinkin} that $\mc P_U^{-1} \cong L(1 \times \alpha)^* \mc P^{-1}$. Note as well that $\iota_{\gamma,U}$ is a representable embedding of smooth quotients, and so by another application of \cite[Lemma 2.3(1)]{Arinkin} we see that $(1 \times \iota_{\gamma,U})^* \mc P_U^{-1} = L(1 \times \iota_{\gamma,U})^* \mc P_U^{-1}$.

    With these preliminaries, we rewrite the right side of (\ref{eq: prop restrict}) using
    $$\tau((1 \times \iota_{\gamma,U})^*\mc P_U^{-1})_\gamma = \tau(L(1 \times \iota_{\gamma,U})^*\mc P_U^{-1})_\gamma = \tau(L\phi^* \mc P^{-1})_\gamma = \tau(L(1 \times \alpha)^* L(1 \times \iota_\gamma)^* \mc P^{-1})_\gamma$$
    $$= \tau^{EG} \circ \rho_\gamma(L(1 \times \alpha)^* L(1 \times \iota_\gamma)^* \mc P^{-1})$$
    since $\oPrym(X/U) \times_U [\oPrym(X/U)^\gamma/\Gamma]$ is already $\gamma$-fixed, so $\iota_{\gamma *}^{-1}$ is trivial, and projecting to $\pi_g$ is not necessary by Proposition \ref{prop: EG factors through m1} and (\ref{eq: rho components})
    $$= \tau^{EG}((1 \times \alpha)^* \rho_\gamma(L(1 \times \iota_\gamma)^* \mc P^{-1}))$$
    as $(1 \times \alpha)$ is $\Gamma$-equivariant
    $$=(1 \times \alpha)^* \tau^{EG}(\rho_\gamma(L(1 \times \iota_\gamma)^* \mc P^{-1}))$$
    by \cite[Theorem 3.1(d)(1)]{EG-RRequivChow}, since $U \to B_\gamma$ is \'etale
    \begin{equation} \label{eq: restricting 1}
        \implies\tau((1 \times \iota_{\gamma,U})^*\mc P_U^{-1})_\gamma =(1 \times \alpha)^* \tau(L(1 \times \iota_\gamma)^* \mc P^{-1})_\gamma
    \end{equation}
    by the same argument with $\iota_{\gamma *}^{-1}, \pi_\gamma$ as before.

    On the left side, consider the representable regular embedding $(1 \times \iota_\gamma): M \times_BM^{\gamma, \vee} \hookrightarrow M \times_B M^\vee$. By Theorem \ref{thm: tau properties}(e) applied with $\mc Z := M \times M^{\gamma, \vee}$ and $\mc W := M \times M^\vee$, we see that
    % $$\tau(L(1 \times \iota_\gamma)^* \mc P^{-1})_\gamma = \ttd(T_{1 \times \iota_\gamma})_\gamma \cap \tau(\mc P^{-1})_\gamma$$
    \begin{equation*} \label{eq: reduce to U 1}
      \tau(\mc P^{-1})_\gamma = \ttd(-T_{1 \times \iota_\gamma})_\gamma \cap \tau(L(1 \times \iota_\gamma)^*\mc P^{-1})_\gamma. 
    \end{equation*}
    Then
    $$\alpha^!\tau(\mc P^{-1})_\gamma := (1 \times \alpha)^* \tau(\mc P^{-1})_\gamma = (1 \times \alpha)^* \ttd(-T_{1 \times \iota_\gamma})_\gamma \cap (1 \times \alpha)^* \tau(L(1 \times \iota_\gamma)^* \mc P^{-1})_\gamma.$$
    By (\ref{eq: restricting 1}), this completes the proof.
\end{proof}

\begin{prop} \label{prop: comparison on V}
    There is an analytic open subset $V \subset U$ mapping homeomorphically onto its image $V \subset B_\gamma$ such that
    $$cl(\mf G^{-1}_\gamma)|_V = (-1)^k cl(\tilde{\mf G}^{-1}_{\gamma})|_V \circ cl(\Sigma)|_V \in H_*^{BM}(\oPrym(X_V/V) \times_V [\oPrym(X_V/V)^\gamma/\Gamma]),$$
    where $\Sigma := (f_P \times p_P)_*[\PMod_P] \in \Corr_U^{-k}(\oPrym(X/U), \Prym(\tilde X/U))$.
\end{prop}

\begin{proof}
    Combining Propositions \ref{prop: smooth and singular Poincare} and \ref{prop: tau B tau U}, and using the identification of $a_P: [\J^0(\tilde X/U)/\J^0(C)] \hookrightarrow [\oJ^0(X/U)/\J^0(C)]$ with $\iota_{\gamma,U}: M^{\gamma,\vee}_U \hookrightarrow M^\vee_U$ from Proposition \ref{prop: a is restriction to Mgamma}, we have
    $$(-1)^k\alpha^! \tau(\mc P^{-1})_\gamma = (1 \times \alpha)^*\ttd(-T_{1 \times \iota_\gamma})_\gamma \cap \tau((f_P \times 1)_*(q_1^*(\omega_{\PMod_P} \otimes \mc V_P^\vee) \otimes (p_P \times 1)^* \tilde{\mc P}^{-1}))_\gamma$$
    $$= (1 \times \alpha)^*\ttd(-T_{1 \times \iota_\gamma})_\gamma \cap I(f_P \times 1)_{\gamma *}\tau((q_1^*(\omega_{\PMod_P} \otimes \mc V_P^\vee) \otimes (p_P \times 1)^* \tilde{\mc P}^{-1}))_\gamma$$
    $$= (1 \times \alpha)^*\ttd(-T_{1 \times \iota_\gamma})_\gamma \cap I(f_P \times 1)_{\gamma *}(\tch(q_1^*(\omega_{\PMod_P} \otimes \mc V_P^\vee))_\gamma \cap \tau((p_P \times 1)^* \tilde{\mc P}^{-1})_\gamma)$$
    $$= (1 \times \alpha)^*\ttd(-T_{1 \times \iota_\gamma})_\gamma \cap I(f_P \times 1)_{\gamma *}(\tch(q_1^*(\omega_{\PMod_P} \otimes \mc V_P^\vee))_\gamma \cap \ttd(T_{p_P \times 1})_\gamma \cap I(p_P \times 1)^!_\gamma\tau(\tilde{\mc P}^{-1})_\gamma)$$
    $$\in \CH_*(\oPrym(X/U) \times_U [\oPrym(X/U)^\gamma/\Gamma])$$
    by applications of Theorem \ref{thm: tau properties}(c), (a), (d), respectively.

    Next, we apply the cycle class map and restrict to an open subset $V \subset U$ as in Lemma \ref{lem: ch and td cancel}, shrinking if necessary to ensure that $V$ maps homeomorphically onto its image under the \'etale map $U \to B_\gamma$. Because of the compatibilities recalled in \S \ref{subsubsec: constructible sheaves}, as well as the fact that proper pushforward in Borel-Moore homology commutes with restriction to an open subset (for schemes, thus immediately for representable morphisms), we see that
    $$(-1)^k cl(\tau(\mc P^{-1})_\gamma)|_V = \ttd^{top}(-T_{1 \times \iota_\gamma}|_V)_\gamma \cap I(f_P \times 1)_{\gamma *}I(p_P \times 1)^*_\gamma cl(\tau(\tilde{\mc P}^{-1})_\gamma)|_V$$
    $$\in H^{BM}_*(\oPrym(X_V/V) \times_V [\oPrym(X_V/V)^\gamma/\Gamma]).$$
    Multiplying both sides by $\ttd^{top}(T_{1 \times \iota_\gamma}|_V-\pi_V^* T_V)$ and using the projection formula, we obtain
    \begin{equation} \label{eq: restrict to V 1}
        \ttd^{top}(T_{1 \times \iota_\gamma}|_V-\pi_V^* T_V)_\gamma \cap cl(\tau(\mc P^{-1})_\gamma)|_V = (-1)^k I(f_P \times 1)_{\gamma *} I(p_P \times 1)^*_\gamma cl(\tilde{\mf G}^{-1}_{\gamma})|_V.
    \end{equation}
    Since $V \subset B_\gamma$ is a (Euclidean) open subset, the left side of (\ref{eq: restrict to V 1}) is
    \begin{equation*} \label{eq: restrict to V LHS}
        \ttd^{top}(T_{1 \times \iota_\gamma}|_V-\pi_V^* T_V)_\gamma \cap cl(\tau(\mc P^{-1})_\gamma)|_V = cl(\ttd(T_{1 \times \iota_\gamma} - \pi_{B_\gamma}^* T_{B_\gamma})_\gamma \cap \tau(\mc P^{-1})_\gamma)|_V = \mf G^{-1}_\gamma|_V.
    \end{equation*}
    On the right side, note that by Lemma \ref{lem: composition with graphs}(a) we have
    $$I(f_P \times 1)_{\gamma *} I(p_P \times 1)^*_\gamma cl(\tilde{\mf G}^{-1}_\gamma)|_V = cl(\tilde{\mf G}^{-1}_\gamma \circ [\Gamma_{p_{P}}] \circ [\Gamma_{f_P}^T])|_V$$
    $$= cl(\tilde{\mf G}^{-1}_\gamma \circ (f_P \times 1)_*[\Gamma_{p_P}])|_V =: cl(\tilde{\mf G}^{-1}_\gamma \circ \Sigma)|_V.$$
    Finally, as before, because of the compatibilities of the class map and restriction to the open subset $V$ with pushforwards and pullbacks, we have
    $$cl(\tilde{\mf G}^{-1}_\gamma \circ \Sigma)|_V = (cl(\tilde{\mf G}^{-1}_\gamma) \circ cl(\Sigma))|_V = cl(\tilde{\mf G}^{-1}_\gamma)|_V \circ cl(\Sigma)|_V.$$
    Thus (\ref{eq: restrict to V 1}) gives the needed equality.
\end{proof}

\begin{lem} \label{lem: ch and td cancel}
    In the situation of Proposition \ref{prop: comparison on V}, for each $u \in U$ there is an analytic open subset $u \in V \subset U$ on which
    $$\tch^{top}(q_1^*(\omega_{\PMod_P} \otimes \mc V_P^\vee))_\gamma \cap \ttd^{top}(T_{p_P \times 1})_\gamma \in H^*(\PMod_P \times_U [\Prym(\tilde X/U)/\Gamma])$$
    pulls back to $1 \in H^*(\PMod_P \times_U [\Prym(\tilde X/U)/\Gamma] \times_U V)$.
\end{lem}

\begin{proof}
    First, note that since $q_1^*(\omega_{\PMod_P} \otimes \mc V^\vee_P)$ and $T_{p_P \times 1} = q_1^* T_{p_P}$ are pulled back from the $\gamma$-fixed factor $\PMod_P$, in fact we have
    $$\tch(q_1^*(\omega_{\PMod_P} \otimes \mc V_P^\vee))_\gamma \cap \ttd(T_{p_P \times 1})_\gamma = q_1^*(\ch(\omega_{\PMod_P} \otimes \mc V^\vee_P) \cap \td(T_{p_P})).$$
    Since the cycle class map is compatible with pullback, it therefore suffices to show that there is $V$ such that
    $$(\ch^{top}(\omega_{\PMod_P} \otimes \mc V^\vee_P) \cap \td^{top}(T_{p_P}))|_V = 1 \in H^*(\PMod_P|_V).$$
    
    By Ehresmann's theorem, we may choose some contractible open $V \ni u$ such that the morphism $\PMod_P \to  U$ is topologically trivial; then restriction $H^*(\PMod_P|_V) \to H^*(\PMod_{P,u})$ is an isomorphism, and so it suffices to work over the point $u$.

    Note first that $\PMod_{P,u} \cong \Prym(\tilde X_u) \times (\BP^1)^k$, as each $\BP_i$ is the projectivization of the sum of topologically trivial line bundles $(\sigma_i^j \times 1)^* \tilde{\mc P}_{\tilde X_u/J}$ (which via the Abel-Jacobi map are restrictions of the universal bundle on $\Pic^0(\J^0(\tilde X_u)) \times \J^0(\tilde X_u)$).
    Thus $T_{p_{P,u}} \cong T_{(\BP^1)^k} = a_1^* T_{\BP^1} \oplus \cdots \oplus a_k^* T_{\BP^1}$ and
    $$\td^{top}(T_{p_{P,u}}) \cong \otimes_i \td^{top}(a_i^* T_{\BP^1}) = \prod_{i=1}^k a_i^*(1 + H_i),$$
    where $a_i: \PMod_{P,u} \to \BP^1$ is the projection onto the $i^{\mathrm{th}}$ factor and $H_i$ the corresponding hyperplane class. On the other hand, note that $c_1^{top}(\mc V_{P,u}^{\otimes 2}) = c_1^{top}(\omega_{\PMod_{P,u}})$ by the relative Euler sequence (see \cite[Lemma 3.8]{FHHO}) and the fact that each factor $\BP_i$ is the projectivization of a topologically trivial vector bundle. Since 
    $$c_1^{top}(\omega_{\PMod_{P,u}}) = \sum_{i=1}^k a_i^*(-2H_i)$$
    we conclude that (always in $H^*_\BC$)
    $$\ch^{top}(\omega_{\PMod_{P,u}} \otimes \mc V_{P,u}^\vee) = \exp\left(\frac{1}{2} \sum_{i=1}^k -2a^*_i(H_i)\right) = \prod_{i=1}^k (1 - a_i^* H_i).$$
    Thus
    $$\ch^{top}(\omega_{\PMod_{P,u}} \otimes \mc V_{P,u}^\vee) \cap \td^{top}(T_{p_{P,u}}) = \prod_{i=1}^k(1- a_i^* H_i) \cap \prod_{i=1}^k(1 + a_i^* H_i)$$
    $$= \prod_{i=1}^k a_i^*(1 - H_i^2) = 1 \in H^*(\PMod_{P,u}).$$
\end{proof}

\begin{cor} \label{cor: Gi vs Fi}
    For each $i \in \BZ$, recalling the homological realization isomorphism (\ref{eq: BM = sheaf Hom}) and the degree conventions of Definition \ref{def: G and p}, we have
    $$cl(\mf G^{-1}_{\gamma,i+h-h_\gamma})|_V = \pm cl(\tilde{\mf G}^{-1}_{\gamma,i})_V \circ cl(\Sigma)_V \in \Hom(\pi_{V*} \BC_{M_V}, \pi_{V*}\BC_{M^{\gamma, \vee}_V}[2(i-h)])$$
    and
    $$cl(\Sigma)|_V \circ cl(\mf G_{\gamma,i})|_V = \pm cl(\tilde{\mf G}_{\gamma, i-h+h_\gamma})|_V \in \Hom(\pi_{V*} \BC_{M^{\gamma, \vee}_V}, \pi_{V*} \BC_{\Prym(\tilde X_V/V)}[2(i-h)]).$$
\end{cor}

\begin{proof}
    The first part is Proposition \ref{prop: comparison on V} taken one degree at a time. (Recall that $\Sigma = [\PMod_P]$ is of pure dimension $\dim \oPrym(X/U)$.)
    % Sigma in Corr(-k), G^{-1}_j in Corr(j-2h + h_\gamma), tilde G^{-1}_j in Corr(j - h_\gamma)
    % i - h_\gamma - k = i - h_\gamma - h + h_\gamma = i -h
    % = j - 2h + h_\gamma

    For the second part, recall from Proposition \ref{prop: G and G-1 inverse} that
    \begin{equation*} 
        \mf G^{-1} \circ \mf G = [\Delta_{IM}] \implies \mf G^{-1}_\gamma \circ \mf G_\gamma = [\Delta_{M^{\gamma, \vee}}].
    \end{equation*}
    It follows that
    \begin{equation} \label{eq: G PMod}
       cl(\tilde{\mf G}_\gamma)|_V = cl(\tilde{\mf G}_\gamma)|_V \circ cl(\mf G^{-1}_\gamma)|_V \circ cl(\mf G_\gamma)|_V
    \end{equation}
    $$ = \pm cl(\tilde{\mf G}_\gamma)|_V  \circ cl(\tilde{\mf G}^{-1}_\gamma)|_V \circ cl(\Sigma)|_V\circ cl(\mf G_\gamma)|_V $$
    by the first part. We can rewrite the last expression as follows: first, by Corollary \ref{cor: Ggamma triv on other pieces}, we see that
    $$cl(\mf G_{\gamma,i})|_V: \pi_{V*} \BC_{M^{\gamma, \vee}_V} \to (\pi_{V*} \BC_{M_V})_\kappa[2(i-h_\gamma)] \subset \pi_{V*} \BC_{M_V}[2(i-h_\gamma)].$$
    Next, recall from Proposition \ref{prop: PMod = Sigma} that $ \Sigma \in \Corr_U^{h_\gamma - h}(\oPrym(X/U), \Prym(\tilde X/U))$ can be defined as the closure of the graph of the $\Gamma$-equivariant pullback map $\nu^*: \Prym(X/U) \to \Prym(\tilde X/U)$, it is stable under the diagonal action of $\Gamma$ and so preserves isotypic components, i.e., we have
    $$cl(\Sigma)|_V: (\pi_{V*} \BC_{M_V})_\kappa[2(i-h_\gamma)] \to (\pi_{V*} \BC_{\Prym(\tilde X_V/V)})_\kappa[2(i-h)].$$
    Finally, recall from Proposition \ref{prop: (ii) smooth} that
    $$\tilde{\mf G}_\gamma \circ \tilde{\mf G}_\gamma^{-1} \in \Corr^0_U(\Prym(\tilde X/U), \Prym(\tilde X/U))$$
    induces the projection
    $$\pi_{U*} \BC_{\Prym(\tilde X/U)} \to (\pi_{U*} \BC_{\Prym(\tilde X/U)})_\kappa$$
    and in particular acts trivially on the $\kappa$ component. Thus, as a composition we have
    $$cl(\tilde{\mf G}_\gamma)|_V  \circ cl(\tilde{\mf G}^{-1}_\gamma)|_V \circ cl(\Sigma)|_V\circ cl(\mf G_\gamma)|_V =  cl(\Sigma)|_V\circ cl(\mf G_\gamma)|_V.$$
    The statement then follows by taking degree pieces of (\ref{eq: G PMod}).
    
    % Sigma in Corr(-k), G_j in Corr(j-h_\gamma), tilde G_j in Corr(j - h_\gamma)
    % -h + h_\gamma + i-h_\gamma = j - h_\gamma
\end{proof}

\begin{proof}[Proof of Theorem \ref{thm: main sec 2}(ii)]
    Fix $\gamma \in \Gamma$ and $k \in \BN$. We only consider $\mf p^{\mf G}_k$; the proof for $\mf q^{\mf G}_k$ is identical. Recall from Corollary \ref{cor: Ggamma triv on other pieces} that $\mf p^{\mf G}_k$ is trivial on all isotypic components $(\pi_* \BC_M)_{\kappa'}$ with $\kappa' \ne \kappa$, so it suffices to understand the image of $(\pi_* \BC_M)_\kappa$. Exactly as in the proof of \cite[Theorem 2.6(ii)]{MSY}, by the condition that $(\pi_* \BC_M)_\kappa$ has full support on $B_\gamma$, it suffices to show that $\mf p_k^{\mf G}$ has image $\ptau_{\le k + \dim B} (\pi_* \BC_M)_\kappa$ after restricting to an analytic open subset $V \subset B_\gamma$.

    Over $V$, it follows from Corollary \ref{cor: Gi vs Fi} and the definition of $\mf p_k^{\mf G}$ that we have
    $$cl(\Sigma)|_V \circ cl(\mf p_k^{\mf G})|_V = \pm cl(\mf p_{k-d_\gamma}^{\tilde{\mf G}})|_V \circ cl(\Sigma)|_V \in \Hom(\pi_{V*} \BC_{M_V}, \pi_{V*} \BC_{\Prym(\tilde X_V/V)}[-2d_\gamma]).$$
    Since $\Sigma$ is identified with Ng\^o's endoscopic correspondence by Proposition \ref{prop: PMod = Sigma}, it follows from \cite[Lemma 3.4.1]{Yun}, as in \cite[Proof of Theorem 3.10]{MS}, that $cl(\Sigma)|_V$
    induces an isomorphism of $\kappa$ components
    $$(\pi_{V*} \BC_{M_V})_\kappa \xrightarrow{\sim} (\pi_{V*} \BC_{\Prym(\tilde X_V/V)})_\kappa[-2d_\gamma].$$
    Combining this with Proposition \ref{prop: (ii) smooth}, we conclude that
    $$cl(\mf p_k^{\mf G})|_V = \pm cl(\Sigma)|_V^{-1} \circ cl(\mf p_{k-d_\gamma}^{\tilde{\mf G}})|_V \circ cl(\Sigma)|_V$$
    is given by
    $$(\pi_{V*} \BC_{M_V})_\kappa \xrightarrow{\sim} (\pi_{V*} \BC_{\Prym(\tilde X_V/V)})_\kappa[-2d_\gamma] \rightarrow (\ptau_{\le k-d_\gamma + b_\gamma}(\pi_{V*} \BC_{\Prym(\tilde X_V/V)})_\kappa)[-2d_\gamma] $$
    $$= \ptau_{\le k+d_\gamma + b_\gamma}((\pi_{V*} \BC_{\Prym(\tilde X_V/V)})_\kappa[-2d_\gamma]) \xrightarrow{\sim} \ptau_{\le k+d_\gamma + b_\gamma}(\pi_{V*} \BC_{M_V})_\kappa.$$
    It now suffices to note that $d_\gamma + b_\gamma = b$ by Corollary \ref{cor: h-hgamma}.
\end{proof}

\section{Projectors and multiplicativity} \label{sec: projectors and mult}
This section of the paper more closely follows \cite{MSY}. In \S \ref{subsec: supp K} we start by establishing an Arinkin-style support inequality for the restriction of convolution kernels to the various endoscopic loci, generalizing \cite[\S 3.3]{MSY}. Via the Adams operations argument of \cite[\S 3.5]{MSY} we then obtain the needed form of the ``Fourier vanishing" condition \cite[Definition 2.5]{MSY}, in \S \ref{subsec: FV}. The proofs of Theorem \ref{thm: main sec 2}(i) and (iii) are then essentially formal; these are treated in \S \ref{subsec: pf of projectors} and \ref{subsec: pf of multiplicativity}, respectively. We return to the convention that all functors are derived unless otherwise stated.

\subsection{Arinkin support inequalities} \label{subsec: supp K}
This section is dedicated to the proofs of two similar support inequalities, Propositions \ref{prop: Fn supp ineq} and \ref{prop: supp K endoscopic}, for the complexes defined in (\ref{eq: Fn def}) and (\ref{eq: K def}), respectively. Let $M = \oPrym(X/B) \to B$ be a good compactified Prym fibration, and write $M_J := \oJ^0(X/B)$ for the corresponding compactified Jacobian variety. Let $i: M \hookrightarrow M_J$ be the inclusion and $p: M_J \to M^\vee = [M_J/\J^0(C)]$ the quotient. We use $\mc P$, respectively $\mc P_J$, to denote the Poincar\'e sheaf on $M^\vee \times_B M$, respectively on $M_J \times_B M_J$, and similarly for the Fourier-Mukai inverses $\mc P^{-1}, \mc P^{-1}_J$.

Let $j: M^\vee \times_B M \hookrightarrow M^\vee \times M$ be the inclusion, and similarly $j_J: M_J \times_B M_J \hookrightarrow M_J \times M_J$. In order to prove the Fourier vanishing result (Proposition \ref{prop: FV}), we will use the complex
\begin{equation} \label{eq: Fn def}
    \mc F^n :=\mc P^{-1} \circ (j_* \mc P)^{\otimes n} \in D^b\Coh(M^\vee \times M^\vee).
\end{equation}
(Here, composition is defined by the formula (\ref{eq: composition on Db}); since $M^\vee \times M^\vee$ is smooth, note that $(j_*\mc P)$ is perfect, thus $(j_* \mc P)^{\otimes n}$ is again perfect, and all operations preserve the bounded derived category.) Similarly, we define
$$\mc F^n_J := \mc P^{-1}_J \circ (j_{J*} \mc P_J)^{\otimes n} \in D^b\Coh(M_J \times M_J).$$

\begin{lem} \label{lem: Kn KnJ}
    We have
    $$(p \times 1)^*\mc F^n = (1 \times p)_*\mc F^n_J \in D^b\Coh(M_J \times M^\vee).$$
\end{lem}

\begin{proof}
    We have
    $$(p \times 1)^*(\mc P^{-1} \circ (j_* \mc P)^{\otimes n}) := (p \times 1)^* p_{13*} (p_{12}^*(j_* \mc P)^{\otimes n} \otimes p_{23}^* \mc P^{-1})$$
    $$ = p_{13*}(p \times 1 \times 1)^* (p_{12}^*(j_* \mc P)^{\otimes n} \otimes p_{23}^* \mc P^{-1})$$
    since $(p \times 1)$ is flat
    $$ = p_{13*}((p_{12}^*(p \times 1)^*j_* \mc P)^{\otimes n} \otimes p_{23}^* \mc P^{-1}) = p_{13*}((p_{12}^* j_*' (p \times 1)^* \mc P)^{\otimes n} \otimes p_{23}^* \mc P^{-1})$$
    again since $(p \times 1)$ is flat (where here $j'$ is the inclusion $M_J \times_B M \hookrightarrow M_J \times M$)
    $$ = p_{13*}((p_{12}^* j_*' (1 \times i)^* \mc P_J)^{\otimes n} \otimes p_{23}^* \mc P^{-1})$$
    by definition of $\mc P$ (noting that there is no difference between derived and underived pullbacks because $p$ is flat, $i$ is a regular embedding, and $\mc P_J$ is maximal Cohen-Macaulay)
    $$ = p_{13*}((p_{12}^*(1 \times i)^* j_{J*} \mc P_J)^{\otimes n} \otimes p_{23}^* \mc P^{-1})$$
    by Lemma \ref{lem: Tor indep}(b), since $j', j_J$ are both flat base changes of the regular embedding $\Delta_B: B \hookrightarrow B \times B$
    $$= \mc P^{-1} \circ (1 \times i)^*(j_{J*} \mc P_J)^{\otimes n} = \mc P^{-1} \circ (\ms O_{\Gamma_i^T} \circ (j_{J*} \mc P_J)^{\otimes n})$$
    \begin{equation} \label{eq: Kn J}
        =(\mc P^{-1} \circ \ms O_{\Gamma_i^T}) \circ (j_{J*} \mc P_J)^{\otimes n} = (\ms O_{\Gamma_p} \circ \mc P^{-1}_J) \circ (j_{J*} \mc P_J)^{\otimes n}
    \end{equation}
    $$= \ms O_{\Gamma_p} \circ (\mc P^{-1}_J \circ (j_{J*} \mc P_J)^{\otimes n}) = (1 \times p)_* (\mc P^{-1}_J \circ (j_{J*} \mc P_J)^{\otimes n}).$$

    Note that the second equality of (\ref{eq: Kn J}) holds because
    \begin{equation} \label{eq: P-1 i}
        \mc P^{-1} \circ \ms O_{\Gamma_i^T} = \mc P^{-1} \circ \ms O_{\Gamma_i^T} \circ \mc P_J \circ \mc P_J^{-1} = \mc P^{-1} \circ \mc P \circ \ms O_{\Gamma_p} \circ \mc P_J^{-1} = \ms O_{\Gamma_p} \circ \mc P^{-1}_J
    \end{equation}
    where for the middle equality we again used that $(1 \times i)^* \mc P_J = (p \times 1)^* \mc P$.
\end{proof}

\begin{prop} \label{prop: Fn supp ineq}
    For any $\gamma, \beta \in \Gamma$ and $n \ge 1$, we have
    $$\dim(\Supp(\mc F^n) \cap M^{\gamma, \vee} \times M^{\beta, \vee}) \le b-h+h_\gamma+h_\beta.$$
\end{prop}

\begin{proof}
    Let $M^\gamma_J$ be the fixed locus of $\gamma$ acting on $M_J$, i.e., $p^{-1}(M^{\gamma,\vee})$. Then
    $$\dim(\Supp(\mc F^n) \cap M^{\gamma, \vee} \times M^{\beta,\vee}) = \dim(\Supp((p \times 1)^* \mc F^n) \cap M^\gamma_J \times M^{\beta, \vee}) -g$$
    $$= \dim(\Supp((1 \times p)_*\mc F^n_J) \cap M^\gamma_J \times M^{\beta,\vee}) -g$$
    by Lemma \ref{lem: Kn KnJ}
    \begin{equation} \label{eq: Fn support}
     \implies \dim(\Supp(\mc F^n) \cap M^{\gamma, \vee} \times M^{\beta,\vee})  \le \dim(\Supp(\mc F^n_J) \cap M^\gamma_J \times M^{\beta}_J) -g. 
    \end{equation}

    Now, by \cite[Proof of Proposition 3.3]{MSY}, we have
    $$\Supp(\mc F^n_J) \subset \{(\mc G_1, \mc G_2) \in \oJ^0(X_u)^{\times 2}: (\mc G_1|_{X_u^{\sm}})^{\otimes n} \cong \mc G_2|_{X_u^{\sm}}\} =: Z \subset M_J \times_B M_J \subset M_J \times M_J.$$
    We note that, for each $u \in B$, the restriction $Z_u \cap (M^\gamma_{J,u} \times_B M^\beta_{J,u}) \subset M^\gamma_{J,u} \times_B M^\beta_{J,u}$ is a countable union of closed subsets of codimension $g(\tilde X_u)$, as multiplication by $\J^0(X_u)$ preserves the fixed loci $\oJ^0(X_u)^\gamma, \oJ^0(X_u)^\beta$ and so the argument of \cite[Proposition 7.4]{Arinkin} applies. Then
    $$\dim (\Supp(\mc F^n_J)|_u \cap (M^\gamma_{J,u} \times M^\beta_{J,u})) \le h_\gamma + g + h_\beta+g - g(\tilde X_u).$$
    Recall from Remark \ref{rmk: Severi/Ngo} that by the Severi inequality
    $$\dim \overline{\{u \in B: g(X_u) - g(\tilde X_u) \ge n\}} \le b - n$$
    and so we have (using that $g(X_u) = \dim M_J - b = h+g$)
    $$\dim (\Supp \mc F^n \cap M_J^{\gamma} \times M^{\beta}_J) \le b + 2g +h_\gamma + h_\beta - (h+g).$$
    Combining this with (\ref{eq: Fn support}), we obtain the needed result.
\end{proof}

Our second complex of interest (for the multiplicativity result in \S \ref{subsec: pf of multiplicativity}) will be
\begin{equation} \label{eq: K def}
    \mc K:= \mc P^{-1} \circ \ms O_{\Delta_{M}^s} \circ (\mc P \boxtimes \mc P) \in D^b\Coh(M^\vee \times_B M^\vee \times_B M^\vee)
\end{equation}
and its counterpart
$$\mc K_J := \mc P^{-1}_J \circ \ms O_{\Delta_{M_J}^s} \circ (\mc P_J \boxtimes \mc P_J) \in D^b\Coh(M_J \times_B M_J \times_B M_J),$$
where compositions are as defined in \S \ref{subsubsec: coherent sheaves} and the small diagonal $\ms O_{\Delta_{M}^s}$ is the pushforward of $\ms O_M$ under the diagonal map $M \to M \times_B M \times_B M$, and similarly for $\ms O_{\Delta_{M_J}^s}$.

\begin{prop} \label{prop: K on M and MJ}
    We have
    $$(p \times p \times 1)^* \mc K = (1 \times 1 \times p)_* \mc K_J \in D^b\Coh(M_J \times_B M_J \times_B M^\vee).$$
\end{prop}

\begin{proof}
    We have
    $$(p \times p \times 1)^* \mc K = \mc K \circ (\ms O_{\Gamma_p} \boxtimes \ms O_{\Gamma_p})$$
    $$:= \mc P^{-1} \circ \ms O_{\Delta^s_{M}} \circ (\mc P \boxtimes \mc P) \circ (\ms O_{\Gamma_p} \boxtimes \ms O_{\Gamma_p}) = \mc P^{-1} \circ \ms O_{\Delta^s_{M}} \circ (\ms O_{\Gamma_i^T} \boxtimes \ms O_{\Gamma_i^T}) \circ (\mc P_J \boxtimes \mc P_J)$$
    since $(p \times 1)^* \mc P = (1 \times i)^* \mc P_J$
    $$ = \mc P^{-1} \circ (i \times i \times 1)_*\ms O_{\Delta^s_{M}} \circ (\mc P_J \boxtimes \mc P_J)$$
    $$= \mc P^{-1} \circ (1 \times 1 \times i)^*\ms O_{\Delta^s_{M_J}} \circ (\mc P_J \boxtimes \mc P_J)$$
    by Lemma \ref{lem: Deltasm base change}
    $$= \mc P^{-1} \circ \ms O_{\Gamma_i^T} \circ \ms O_{\Delta^s_{M_J}}  \circ (\mc P_J \boxtimes \mc P_J)$$
    $$= \ms O_{\Gamma_p} \circ \mc P^{-1}_J \circ \ms O_{\Delta^s_{M_J}}  \circ (\mc P_J \boxtimes \mc P_J) =: \ms O_{\Gamma_p} \circ \mc K_J = (1 \times 1 \times p)_* \mc K_J$$
    where for the first equality we use (\ref{eq: P-1 i}).
\end{proof}

\begin{lem} \label{lem: Deltasm base change}
    We have
    $$(i \times i \times 1)_*\ms O_{\Delta^s_{M}} = (1 \times 1 \times i)^*\ms O_{\Delta^s_{M_J}} \in D^b\Coh(M_J \times_B M_J \times_B M).$$
\end{lem}

\begin{proof}
    Consider the Cartesian square
    \[
    \begin{tikzcd}
        M \arrow[d, "i"] \arrow[r, "i \times i \times 1"] & M_J \times_B M_J \times_B M \arrow[d, "1 \times 1 \times i"] \\
        M_J \arrow[r, "1 \times 1 \times 1"] & M_J \times_B M_J \times_B M_J
    \end{tikzcd}
    \]
    Since $i: M \hookrightarrow M_J$ is an embedding of regular varieties, this and the flat base change $1 \times 1 \times i$ are regular embeddings of the same codimension $g$, and so by Lemma \ref{lem: Tor indep}(b) the square is $\Tor$-independent and
    $$(i \times i \times 1)_* \ms O_{\Delta^s_{M}} = (i \times i \times 1)_* i^* \ms O_{M_J}$$
    $$=(1 \times 1 \times i)^* (1 \times 1 \times 1)_* \ms O_{M_J} = (1 \times 1 \times i)^* \ms O_{\Delta^s_{M_J}}.$$
\end{proof}

\begin{prop} \label{prop: supp K endoscopic}
    For any $\gamma, \beta, \alpha \in \Gamma$, we have
    $$\dim(\Supp \mc K \cap M^{\gamma, \vee} \times_B M^{\beta, \vee} \times_B M^{\alpha, \vee}) \le b-h+h_\gamma + h_\beta + h_\alpha.$$
\end{prop}

\begin{proof}
    The proof is analogous to that of Proposition \ref{prop: Fn supp ineq}. By \cite[Proof of Proposition 3.2]{MSY}, over each point $u \in B$, the closed subset $\Supp \mc K_J|_u \subset \oJ^0(X_u)^{\times 3}$ is contained in the locus
    $$Z :=\{(\mc F_1, \mc F_2, \mc F_3): \mc F_1|_{X_u^{\sm}} \otimes \mc F_2|_{X_u^{\sm}} = \mc F_3|_{X_u^{\sm}} \} \subset \oJ^0(X_u)^{\times 3}.$$
    As before, $Z \cap \oJ^0(X_u)^\gamma \times \oJ^0(X_u)^\beta \times \oJ^0(X_u)^\alpha$ is a countable union of closed subsets of codimension $ g(\tilde X_u)$. Then by the Severi inequality
    $$\dim (\Supp \mc K_J \cap M_J^{\gamma} \times_B M^{\beta}_J \times_B M^{\alpha}_J) \le b + 3g+ h_\gamma + h_\beta + h_\alpha - (h+g).$$
    Finally, by Proposition \ref{prop: K on M and MJ} we have
    $$\dim(\Supp \mc K \cap M^{\gamma, \vee} \times_B M^{\beta, \vee} \times_B M^{\alpha, \vee})  \le \dim(\Supp \mc K_J \cap M_J^{\gamma} \times_B M^{\beta}_J \times_B M^{\alpha}_J) - 2g.$$
\end{proof}

\subsection{Fourier vanishing} \label{subsec: FV}
In this section, we combine the bound of Proposition \ref{prop: Fn supp ineq} with the Adams operations argument in \cite[\S 3.5]{MSY} to establish an equivalent of the Fourier vanishing condition of \cite[Definition 2.5]{MSY} for each $\gamma \in \Gamma$ (Proposition \ref{prop: FV}). 

Fix $\gamma \in \Gamma$, and pick an approximation scheme $u: M^{\gamma, \vee}_\Gamma \to M^{\gamma, \vee}$ as defined in \S \ref{subsubsec: EG def of Chow}. (Instead of verifying that the usual compatibilities among $\tau$, localized Chern classes, and Adams operations hold also in the case of nice quotients, we find it more convenient to work on the approximation scheme.)

Recall from \cite[\S 18.2]{Fulton} that from the smooth embedding of varieties $M_\Gamma^{\gamma, \vee} \times_B M \hookrightarrow M^{\gamma, \vee}_\Gamma \times M$ we obtain
a localized Chern character
$$\ch^{loc} : K_0(M_\Gamma^{\gamma, \vee} \times_B M) \to \CH^*(M_\Gamma^{\gamma, \vee} \times_B M \rightarrow M^{\gamma, \vee}_\Gamma \times M)$$
such that by definition
$$\tau(-) := \td(T_{M^{\gamma, \vee}_\Gamma \times M}) \cap \ch^{loc}(-) \cap [M^{\gamma, \vee}_\Gamma \times M]: K_0(M_\Gamma^{\gamma, \vee} \times_B M) \to \CH_*(M_\Gamma^{\gamma, \vee} \times_B M).$$
We recall from \cite{Gillet-Soule} that there are Adams operations $\psi^n$ on the relative $K$-theory group
$$K_0^{M^{\gamma, \vee}_\Gamma \times_B M}(M^{\gamma, \vee}_\Gamma \times M) = K_0(M^{\gamma, \vee}_\Gamma \times_B M)$$
(the equality being \cite[Lemma 1.9]{Gillet-Soule})
defined inductively via
$$\psi^k - \psi^{k-1} \cdot  \lambda^1 + \cdots + (-1)^{k-1} \psi^1 \cdot \lambda^{k-1} + (-1)^kk\lambda^k = 0,$$
where $\lambda_i$ is given by the $i^{\mathrm{th}}$ exterior power of a perfect complex on $M^{\gamma, \vee}_\Gamma \times M$ supported on $M^{\gamma, \vee}_\Gamma \times_B M$.
We will use the identity \cite[Theorem 3.1]{Kurano-Roberts}
\begin{equation} \label{eq: Kurano-Roberts}
    \ch^{loc, i} \circ \psi^k = k^i \ch^{loc,i}: K_0(M^{\gamma, \vee}_\Gamma \times_B M) \to \CH^i(M_\Gamma^{\gamma, \vee} \times_B M \rightarrow M^{\gamma, \vee}_\Gamma \times M).
\end{equation}

Let $j_u: M^{\gamma,\vee}_\Gamma \times_B M \hookrightarrow M^{\gamma,\vee}_\Gamma \times M$ be the inclusion, and similarly $j_u^T: M \times_B M^{\gamma,\vee}_\Gamma \hookrightarrow M \times M^{\gamma,\vee}_\Gamma$ We now recast Proposition \ref{prop: Fn supp ineq} in the following form:

\begin{prop} \label{prop: psi support}
    The class
    $$(j_{u*}^T(1 \times u)^*(\rho_\gamma(\mc P^{-1}_\gamma))) \circ (j_{u*}\psi^n((u \times 1)^*(\rho_\gamma(\mc P_\gamma)) \in K_0(M^{\gamma, \vee}_\Gamma \times M_\Gamma^{\gamma, \vee})$$
    has support of dimension $\le 2 \dim u + b-h+2h_\gamma$, where
    $\mc F_\gamma := \iota_{\gamma *}^{-1} \circ \pi_\gamma(\mc F)$ is the localization defined in (\ref{eq: pi}) and (\ref{eq: iota}).
\end{prop}

Note that by pushing forward along $j_u, j_u^T$ we eliminate the problem that the diagonal $\delta_M: M^{\gamma,\vee}_\Gamma \times_B M \times_B M^{\gamma,\vee}_\Gamma \to M^{\gamma,\vee}_\Gamma \times_B M \times M \times_B M^{\gamma,\vee}_\Gamma$ may not be an l.c.i. morphism, as $M \times_B B_\gamma$ may not be smooth. In general, the (not necessarily proper) pushforward $p_{13*}$ may not be defined on $K_0$, but since we work with a class on $M^{\gamma,\vee}_\Gamma \times M \times M^{\gamma,\vee}_\Gamma$ supported on $M^{\gamma,\vee}_\Gamma \times_B M \times_B M^{\gamma,\vee}_\Gamma$, which is proper over $M^{\gamma,\vee}_\Gamma \times M^{\gamma,\vee}_\Gamma$, the pushforward is well-defined in this case.

\begin{proof}
    We have
    \begin{multline*}
        \Supp (j_{u*}^T(1 \times u)^*(\rho_\gamma(\mc P^{-1}_\gamma)) \circ (j_{u*}\psi^n(u \times 1)^*(\rho_\gamma(\mc P_\gamma)))) \\ \subset \Supp (j_{u*}^T(1 \times u)^*(\rho_\gamma(\mc P^{-1}_\gamma)) \circ (j_{u*}(u \times 1)^*(\rho_\gamma(\mc P_\gamma)))^{\otimes n})
    \end{multline*}
    since $j_{u*} \psi^n(\mc F)$ was defined inductively as a sum of products of exterior powers of $j_{u*} \mc F$ of degrees adding up to $n$, each of which is an $S_n$-isotypic component (and thus direct summand) of $(j_{u*} \mc F)^{\otimes n}$
    $$= \Supp ((1 \times u)^*(j_*^T(\rho_\gamma(\mc P^{-1}_\gamma))) \circ (u \times 1)^*(j_*(\rho_\gamma(\mc P_\gamma)))^{\otimes n})$$
    since $u$ is flat
    $$= (u \times u)^{-1} (\Supp (j^T_{\gamma *}(\rho_\gamma(\mc P^{-1}_\gamma)) \circ (j_{\gamma *}(\rho_\gamma(\mc P_\gamma)))^{\otimes n}))$$
    using the usual compatibility of flat pullback with proper pushforward and Gysin morphisms \cite[Lemmas 3.1, 3.2]{Anderson-Payne}, and using $j_\gamma: M^{\gamma,\vee} \times_B M \hookrightarrow M^{\gamma,\vee} \times M$ for the inclusion 
    $$ = (u \times u)^{-1} (\Supp (\rho_\gamma(j_{\gamma *}^T\mc P^{-1}_\gamma) \circ \rho_\gamma((j_{\gamma *}\mc P_\gamma)^{\otimes n})))$$
    since $j_\gamma$ is $\Gamma$-equivariant
    $$= (u \times u)^{-1} (\Supp ((j_{\gamma *}^T\mc P^{-1}_\gamma) \circ (j_{\gamma *}\mc P_\gamma)^{\otimes n}))$$
    since $\rho_\gamma(\mc F) \circ \rho_\gamma(\mc G) = \rho_{(\gamma, \gamma)}(\mc F \circ \mc G)$ and $\rho_{(\gamma, \gamma)}$ is an isomorphism of $K_0(M^{\gamma, \vee} \times M^{\gamma, \vee})$ compatible with $(\Gamma \times \Gamma)$-equivariant pushforward, thus respecting supports
    $$\subset (u \times u)^{-1}(\Supp(\mc P^{-1} \circ (j_* \mc P)^{\otimes n}) \cap (M^{\gamma, \vee} \times M^{\gamma, \vee}))$$
    by Lemma \ref{lem: gamma loc vs intersect}. Thus it follows from Proposition \ref{prop: Fn supp ineq} that
    $$\dim \Supp (j_{u*}^T(1 \times u)^*(\rho_\gamma(\mc P^{-1}_\gamma)) \circ (j_{u*}\psi^n(u \times 1)^*(\rho_\gamma(\mc P_\gamma)))) \le 2\dim u + b-h+2h_\gamma.$$
\end{proof}

\begin{lem} \label{lem: gamma loc vs intersect}
    We have
    $$\Supp((j_{\gamma *}^T\mc P^{-1}_\gamma) \circ (j_{\gamma *} \mc P_\gamma)^{\otimes n}) \subset \Supp(\mc P^{-1} \circ (j_*\mc P)^{\otimes n}) \cap (M^{\gamma, \vee} \times M^{\gamma, \vee}) \subset M^\vee \times M^\vee.$$
\end{lem}

\begin{proof}
    First, by (\ref{eq: pushforward loc}), we have $j_{\gamma *} \mc P_\gamma = (j_*\mc P)_\gamma$. Since $M^\vee \times M$ is smooth, recall from Lemma \ref{lem: localization inverse} that
    $$(j_* \mc P)_\gamma = \lambda_{-1}(\mc N_{\iota_\gamma}^\vee)^{-1} \cdot \iota_\gamma^!(j_*\mc P) \in K_0(M^{\gamma, \vee} \times M),$$
    where $\iota_\gamma: M^{\gamma, \vee} \times M \hookrightarrow M^\vee \times M$ is the inclusion of the fixed locus. In particular, we conclude that
    $$(j_{\gamma *} \mc P_\gamma)^{\otimes n} = (\lambda_{-1}(\mc N_{\iota_\gamma}^\vee)^{-1} \cdot \iota_\gamma^!(j_*\mc P))^{\otimes n}$$
    $$= \lambda_{-1}(\mc N_{\iota_\gamma}^\vee)^{-n} \cdot \iota_\gamma^!((j_*\mc P)^{\otimes n}) = \lambda_{-1}(\mc N_{\iota_\gamma}^\vee)^{-n+1} \cdot ((j_*\mc P)^{\otimes n})_\gamma.$$
    Thus, with another application of (\ref{eq: pushforward loc}) we have
    $$(j_{\gamma *}^T\mc P^{-1}_\gamma) \circ (j_{\gamma *} \mc P_\gamma)^{\otimes n} = (j_{*}^T\mc P^{-1})_\gamma \circ (\lambda_{-1}(\mc N_{\iota_\gamma}^\vee)^{-n+1} \cdot ((j_*\mc P)^{\otimes n})_\gamma)$$
    $$= p_1^*\lambda_{-1}(\mc N_{M^{\gamma,\vee}/M^\vee}^\vee)^{-n+1} \cdot ((j_*^T\mc P^{-1})_\gamma \circ ((j_*\mc P)^{\otimes n})_\gamma)$$
    by the projection formula, and so
    \begin{equation} \label{eq: move loc}
        \Supp ((j_{\gamma *}^T\mc P^{-1}_\gamma) \circ (j_{\gamma *} \mc P_\gamma)^{\otimes n}) \subset \Supp((j^T_*\mc P^{-1})_\gamma \circ ((j_*\mc P)^{\otimes n})_\gamma).
    \end{equation}

    Next, we claim that
    \begin{equation} \label{eq: composition loc}
        (j_*^T\mc P^{-1})_\gamma \circ ((j_*\mc P)^{\otimes n})_\gamma = ((j_*^T\mc P^{-1}) \circ (j_*\mc P)^{\otimes n})_{(\gamma,\gamma)} \in K_0(M^{\gamma, \vee} \times M^{\gamma, \vee}).
    \end{equation}
    Indeed, we have
    $$((j_*^T\mc P^{-1}) \circ (j_*\mc P)^{\otimes n})_{(\gamma,\gamma)} := (p_{13*}\delta_M^! ((j_*\mc P)^{\otimes n} \boxtimes (j_*^T\mc P^{-1})))_{(\gamma, \gamma)}$$
    $$ = p_{13*}((\delta_M^! ((j_*\mc P)^{\otimes n} \boxtimes (j_*^T\mc P^{-1})))_{(\gamma, \gamma)})$$
    by (\ref{eq: pushforward loc}), applied to both the (proper) inclusion $M^{\gamma,\vee} \times_B M \times_B M^{\gamma,\vee} \hookrightarrow M^{\gamma,\vee} \times M \times M^{\gamma,\vee}$ and the (proper) projection map $M^{\gamma,\vee} \times_B M \times_B M^{\gamma,\vee} \to M^{\gamma,\vee} \times M^{\gamma,\vee}$
    $$ = p_{13*}\delta_M^! (((j_*\mc P)^{\otimes n} \boxtimes (j_*^T\mc P^{-1}))_{(\gamma, \gamma)})$$
    by the first equality in Lemma \ref{lem: pullback first step}
    $$ = p_{13*}\delta_M^! (((j_*\mc P)^{\otimes n})_\gamma \boxtimes (j_*^T\mc P^{-1})_\gamma) \,\,\,\,\, \mathrm{by}\,\, (\ref{eq: boxtimes loc})$$
    $$ =: (j_*^T\mc P^{-1})_\gamma \circ ((j_*\mc P)^{\otimes n})_\gamma.$$

    Finally, we note that
    $$(j_*^T\mc P^{-1}) \circ (j_* \mc P)^{\otimes n} = \mc P^{-1} \circ (j_* \mc P)^{\otimes n}$$
    by compatibility of proper pushforward with Gysin morphisms \cite[Lemma 3.1]{Anderson-Payne}, and so, combining with (\ref{eq: move loc}) and (\ref{eq: composition loc}), we obtain
    $$\Supp ((j_{\gamma *}^T\mc P^{-1}_\gamma) \circ (j_{\gamma *} \mc P_\gamma)^{\otimes n}) \subset \Supp(\mc P^{-1} \circ (j_* \mc P)^{\otimes n})_{(\gamma,\gamma)}$$
    $$ \subset \Supp(\mc P^{-1} \circ (j_*\mc P)^{\otimes n}) \cap (M^{\gamma, \vee} \times M^{\gamma,\vee})$$
    because in general, given $\mc X = [X/G]$, since $\iota_{g *}^{-1} \circ \pi_g$ commutes with $G$-equivariant proper pushforward by (\ref{eq: pushforward loc}), in particular $\iota_{g*}^{-1} \circ \pi_g(\mc F)$ is supported on $(\Supp \mc F)^g$.
\end{proof}

\begin{prop} \label{prop: sheaf to tau supp}
    We have
    \begin{multline*}
        \dim ((1 \times u)^* \mf G_\gamma^{-1} \circ (\ch^{loc}(\psi^n((u \times 1)^*(\rho_\gamma(\mc P_\gamma)))) \cap [M^{\gamma, \vee}_\Gamma \times M])) \\ \le \dim \Supp (j_{u*}^T(1 \times u)^*(\rho_\gamma(\mc P^{-1}_\gamma)) \circ j_{u*}\psi^n((u \times 1)^*(\rho_\gamma(\mc P_\gamma))))
    \end{multline*}
    where by dimension of a class of $\CH_*(M^{\gamma,\vee} \times_B M^{\gamma,\vee})$ we mean the highest $k$ such that the $\CH_k$ component is nonzero.
\end{prop}

\begin{proof}
    By the standard properties of the (schematic) $\tau$ functor \cite[Theorem 18.2, Example 18.3.1]{Fulton}, and again using that $p_{13}$ is proper on the support of the element of $K_0(M^{\gamma,\vee}_\Gamma \times M \times M^{\gamma,\vee}_\Gamma)$ under consideration, we have
    $$\tau(j_{u*}^T(1 \times u)^*(\rho_\gamma(\mc P^{-1}_\gamma)) \circ j_{u*}\psi^n((u \times 1)^*(\rho_\gamma(\mc P_\gamma))))$$
    $$= p_{13*}(\td(-p_2^* T_M) \cap \delta_M^!(j_{u*}\tau(\psi^n((u \times 1)^*(\rho_\gamma(\mc P_\gamma)))) \times j_{u*}^T\tau((1 \times u)^*\rho_\gamma(\mc P^{-1}_\gamma))))$$
    $$= p_{13*}(\td(-p_2^* T_M) \cap \Delta_M^!(\tau(\psi^n((u \times 1)^*(\rho_\gamma(\mc P_\gamma)))) \times \tau((1 \times u)^*\rho_\gamma(\mc P^{-1}_\gamma))))$$
    since $\delta_M$ is a flat base change of $\Delta_M$, and by compatibility of proper pushforward with Gysin morphisms and pushforward
    $$ = p_{13*}\Delta_M^!((\td(-p_2^* T_M) \cap \tau(\psi^n((u \times 1)^*(\rho_\gamma(\mc P_\gamma))))) \times \tau((1 \times u)^*\rho_\gamma(\mc P^{-1}_\gamma)))$$
    $$ = \tau((1 \times u)^* \rho_\gamma(\mc P^{-1})) \circ (\td(p_1^*T_{M^{\gamma,\vee}_\Gamma}) \cap \ch^{loc}(\psi^n((u \times 1)^* \rho_\gamma(\mc P_\gamma))) \cap [M^{\gamma,\vee}_\Gamma \times M])$$
    $$ = (1 \times u)^*\tau^{EG}( \rho_\gamma(\mc P^{-1})) \circ (\td(p_1^*T_{M^{\gamma,\vee}_\Gamma}) \cap \ch^{loc}(\psi^n((u \times 1)^* \rho_\gamma(\mc P_\gamma)))\cap [M^{\gamma,\vee}_\Gamma \times M])$$
    by definition of $\tau^{EG}$
    \begin{multline*}
        = \td(p_1^* T_{M^{\gamma,\vee}_\Gamma} + p_2^*\pi_{B_\gamma}^* T_{B_\gamma}) \cap p_2^* u^* \td^I(-T_{\iota_{M^\vee}})_\gamma \\ \cap  ((1 \times u)^*\mf G^{-1}_\gamma \circ (\ch^{loc}(\psi^n((u \times 1)^* \rho_\gamma(\mc P_\gamma)))\cap [M^{\gamma,\vee}_\Gamma \times M])).
    \end{multline*}
    In particular, we conclude that
    \begin{multline*}
        \dim \tau(j_{u*}^T(1 \times u)^*(\rho_\gamma(\mc P^{-1}_\gamma)) \circ j_{u*}\psi^n((u \times 1)^*(\rho_\gamma(\mc P_\gamma)))) \\ = \dim ((1 \times u)^*\mf G^{-1}_\gamma \circ (\ch^{loc}(\psi^n((u \times 1)^* \rho_\gamma(\mc P_\gamma)))\cap [M^{\gamma,\vee}_\Gamma \times M])).
    \end{multline*}
    Thus, it suffices to note that in general
    $$\dim \Supp \mc F \ge \dim \tau(\mc F)$$
    by compatibility of $\tau$ with proper pushforward.
\end{proof}

\begin{prop} \label{prop: FV}
    (Fourier vanishing) We have
    $$\mf G_{\beta,i}^{-1} \circ \mf G_{\gamma,j} = 0 \,\,\, \mathrm{if} \,\, \begin{cases}
        \beta \ne \gamma \\
        \beta = \gamma, i+j<2h.
    \end{cases}$$
\end{prop}

\begin{proof}
    The first part is Proposition \ref{prop: FV for different chars}. As for the second, we start by noting that by construction of $u$ we have
    $$\mf G_{\gamma,i}^{-1} \circ \mf G_{\gamma,j} = 0 \iff (u \times u)^*(\mf G_{\gamma,i}^{-1} \circ \mf G_{\gamma,j}) = ((1 \times u)^* \mf G_{\gamma,i}^{-1}) \circ ((u \times 1)^* \mf G_{\gamma,j}) = 0 $$
    (where the second equality is because Gysin pullback $\Delta_M^!$ and projective pushforward $p_{13*}$ commute with flat pullback). By definition and the projection formula, we have
    \begin{multline*}
        ((1 \times u)^* \mf G_{\gamma,i}^{-1}) \circ ((u \times 1)^* \mf G_{\gamma,j}) \\ = \sum_{a \ge 0} \td^a(\pi_{B_\gamma}^* T_{B_\gamma}) \cap \left( (1 \times u)^* \mf G_{\gamma, i}^{-1} \circ (\ch^{loc, j-a+b_\gamma}((u \times 1)^*(\rho(\mc P_\gamma))) \cap [M^{\gamma, \vee}_\Gamma \times M]) \right).
    \end{multline*}
    It therefore suffices to show that
    \begin{equation} \label{eq: FV with chloc}
        (1 \times u)^* \mf G_{\gamma,i}^{-1} \circ (\ch^{loc,k}((u \times 1)^*(\rho_\gamma(\mc P_\gamma))) \cap [M^{\gamma,\vee}_\Gamma \times M]) = 0
    \end{equation}
    for $i + k < 2h+b_\gamma$.

    Now, combining Propositions \ref{prop: psi support} and \ref{prop: sheaf to tau supp}, for each $n \ge 1$ and $a < b_\gamma - b + h = h_\gamma$ (Remark \ref{rmk: h-hgamma}) we have
    $$\left[(1 \times u)^* \mf G_\gamma^{-1} \circ (\ch^{loc}(\psi^n((u \times 1)^*(\rho_\gamma(\mc P_\gamma)))) \cap [M^{\gamma, \vee}_\Gamma \times M])\right]_{\codim = a} = 0$$
    $$ = \sum_{j \ge 0} (1 \times u)^* \mf G_{\gamma,a+h+b-j}^{-1} \circ \ch^{loc, j}(\psi^n((u \times 1)^*(\rho_\gamma(\mc P_\gamma)))) \cap [M^{\gamma, \vee}_\Gamma \times M]$$
    $$ = \sum_{j \ge 0} n^j (1 \times u)^* \mf G_{\gamma,a+h+b-j}^{-1} \circ \ch^{loc, j}((u \times 1)^*(\rho_\gamma(\mc P_\gamma))) \cap [M^{\gamma, \vee}_\Gamma \times M]$$
    by (\ref{eq: Kurano-Roberts}). Since this is true for all $n \ge 1$, we conclude that each term
    $$(1 \times u)^* \mf G_{\gamma,a+h+b-j}^{-1} \circ \ch^{loc, j}((u \times 1)^*(\rho_\gamma(\mc P_\gamma))) \cap [M_\Gamma^{\gamma, \vee} \times M] = 0$$
    In other words, this is true for $j \ge 0$ and $a <h_\gamma$ and so for
    $$i+k = a+h+b = a + h + (b_\gamma + h - h_\gamma) < 2h + b_\gamma.$$
\end{proof}

\subsection{Proof of Theorem \ref{thm: main sec 2}(i)} \label{subsec: pf of projectors}
As in \cite[\S 2.5.1]{MSY}, the Fourier vanishing condition (Proposition \ref{prop: FV}) immediately implies that the correspondences $\mf p_k^{\mf G}, \mf q_k^{\mf G}$ are projectors.

\begin{proof}[Proof of Theorem \ref{thm: main sec 2}(i)]
    We have
    $$\mf p_{\gamma,k}^{\mf G} = [\Delta_M] \circ \mf p_{\gamma, k}^{\mf G} = \sum_{i \ge 0} \mf G_i \circ \mf G^{-1}_{2h-i} \circ \mf p_{\gamma,k}^{\mf G}$$
    by Proposition \ref{prop: G and G-1 inverse} and the fact that $[\Delta_M] \in \Corr^0_B(M,M)$
    $$ = \sum_\beta \sum_{i \ge 0} \mf G_{\beta,i} \circ \mf G_{\beta,2h-i}^{-1} \circ \sum_{j \le k} \mf G_{\gamma,j} \circ \mf G_{\gamma, 2h-j}^{-1}$$
    $$= \sum_{i \ge 0} \mf G_{\gamma,i} \circ \mf G_{\gamma,2h-i}^{-1} \circ\sum_{j \le k} \mf G_{\gamma,j} \circ \mf G_{\gamma, 2h-j}^{-1}$$
    by Proposition \ref{prop: FV for different chars}
    $$= \sum_{i \le k} \mf G_{\gamma,i} \circ \mf G_{\gamma,2h-i}^{-1} \circ\sum_{j \le k} \mf G_{\gamma,j} \circ \mf G_{\gamma, 2h-j}^{-1} = \mf p_{\gamma,k}^{\mf G} \circ \mf p_{\gamma,k}^{\mf G}$$
    by Proposition \ref{prop: FV}. Moreover, for any $\beta \ne \gamma$ we note that
    $$\mf p_{\beta,j}^{\mf G} \circ \mf p^{\mf G}_{\gamma,k} = 0 = \mf p^{\mf G}_{\gamma,k} \circ \mf p_{\beta,j}^{\mf G} $$
    by Proposition \ref{prop: FV for different chars}, and so any sum $\sum_\gamma \mf p_{\gamma, k_\gamma}^{\mf G} = \pg_{(k_\gamma)_\gamma}$ is a projector. It follows immediately that the difference
    $$[\Delta_M] - \sum_\gamma \mf p_{\gamma,k_\gamma}^{\mf G} = \sum_\gamma \mf q_{\gamma, k_\gamma+1}^{\mf G} = \qg_{(k_\gamma+1)_\gamma}$$
    is a projector orthogonal to $\pg_{(k_\gamma)_\gamma}$.
\end{proof}

\subsection{Proof of Theorem \ref{thm: main sec 2}(iii)} \label{subsec: pf of multiplicativity}
Recall the convolution kernel
$$\mc K := \mc P^{-1} \circ \ms O_{\Delta^s_M} \circ (\mc P \boxtimes \mc P) \in D^b\Coh(M^\vee \times_B M^\vee \times_B M^\vee)$$
from (\ref{eq: K def}). We define
$$\mf C := \td^I(p_1^*(\pi_{IB}^* T_{IB} - T_{IM^\vee}) +p_2^*(\pi_{IB}^* T_{IB} - T_{IM^\vee}) -p_3^*(\pi_{IB}^* T_{IB} - T_{\iota_M^\vee})) \cap \tau(\mc K)$$
$$\in \CH_*(I(M^\vee \times_B M^\vee \times_B M^\vee)) = \CH_*(IM^\vee \times_B IM^\vee \times_B IM^\vee).$$
The Todd class convention is chosen so that the following is true:

\begin{prop} \label{prop: C identity}
    We have
    $$\mf G \circ \mf C \circ (\mf G^{-1} \times \mf G^{-1}) = [\Delta_M^s] \in \CH_*(M \times_B M \times_B M).$$
\end{prop}

\begin{proof}
    First, from the definition of $\mc K$, we obtain
    \begin{equation} \label{eq: C before tau}
        \mc P \circ \mc K \circ (\mc P^{-1} \boxtimes \mc P^{-1}) = \ms O_{\Delta^s_M}
    \end{equation}
    since $\mc P \circ \mc P^{-1} = \ms O_{\Delta_M}$, composition is associative, and $(\mc A \boxtimes \mc B) \circ (\mc C \boxtimes \mc D) = (\mc A \circ \mc C) \boxtimes (\mc B \circ \mc D)$. As usual, we apply $\tau$ to both sides of (\ref{eq: C before tau}), using the properties established in Theorem \ref{thm: tau properties}. (We will use these properties without comment, and whenever we are applying Corollary \ref{cor: part e}, we take $\mc Z_i$ to be $\mc Y_i$ with all $\times_B$ replaced by $\times$.) 

    On the left, we start by calculating
    $$\tau(\mc K \circ (\mc P^{-1} \boxtimes \mc P^{-1})) := \tau(p_{125*} \delta_{M^\vee_1}^* \delta_{M^\vee_2}^*(\mc P^{-1} \boxtimes \mc P^{-1} \boxtimes \mc K))$$
    $$ = Ip_{125*} \tau(\delta_{M^\vee_1}^* \delta_{M^\vee_2}^*(\mc P^{-1} \boxtimes \mc P^{-1} \boxtimes \mc K))$$
    $$= Ip_{125*}(\ttd(-p_3^*T_{M^\vee}) \cap I\Delta_{M^\vee_1}^!\tau(\delta_{M^\vee_2}^*(\mc P^{-1} \boxtimes \mc P^{-1} \boxtimes \mc K)))$$
    $$= Ip_{125*}(\ttd(-p_3^*T_{M^\vee}) \cap I\Delta_{M^\vee_1}^!(\ttd(-p_5^* T_{M^\vee}) \cap I\Delta^!_{M^\vee_2}(\tau(\mc P^{-1} \boxtimes \mc P^{-1} \boxtimes \mc K))))$$
    $$= Ip_{125*}(\ttd(-p_3^*T_{M^\vee}-p_4^* T_{M^\vee}) \cap I\Delta_{M^\vee_1}^! I\Delta^!_{M^\vee_2}(\tau(\mc P^{-1}) \times \tau(\mc P^{-1}) \times \tau(\mc K)))$$
    \begin{multline*}
        = Ip_{125*}I\Delta_{M^\vee_1}^! I\Delta^!_{M^\vee_2}((\td^I(-\pi_{IB}^* T_{IB} + p_2^* T_{\iota_{M^\vee}}) \cap  \tau(\mc P^{-1})) \times (\td^I(-\pi_{IB}^* T_{IB} + p_2^* T_{\iota_{M^\vee}}) \cap  \tau(\mc P^{-1})) \\ \times
    (\td^I(p_1^*(\pi_{IB}^* T_{IB} - T_{IM^\vee}) + p_2^*(\pi_{IB}^* T_{IB} - T_{IM^\vee})) \cap \tau(\mc K)))
    \end{multline*}
    $$ = Ip_{125*}I\Delta_{M^\vee_1}^! I\Delta^!_{M^\vee_2}(\mf G^{-1} \times \mf G^{-1} \times (\td^I(p_3^*(\pi_{IB}^* T_{IB} - T_{\iota_M^\vee})) \cap \mf C))$$
    $$ = \td^I(p_3^*(\pi_{IB}^* T_{IB} - T_{\iota_M^\vee})) \cap Ip_{125*}I\Delta_{M^\vee_1}^! I\Delta^!_{M^\vee_2}(\mf G^{-1} \times \mf G^{-1} \times \mf C)$$
    \begin{equation} \label{eq: C 1}
        \implies \tau(\mc K \circ (\mc P^{-1} \boxtimes \mc P^{-1})) = \td^I(p_3^*(\pi_{IB}^* T_{IB} - T_{\iota_M^\vee})) \cap (\mf C \circ (\mf G^{-1} \times \mf G^{-1})).
    \end{equation}
    
    Then
    $$\tau(\mc P \circ \mc K \circ (\mc P^{-1} \boxtimes \mc P^{-1})) := \tau(p_{124*} \delta_{M^\vee}^*((\mc K \circ (\mc P^{-1} \boxtimes \mc P^{-1})) \boxtimes \mc P))$$
    $$ = Ip_{124*} \tau(\delta_{M^\vee}^*((\mc K \circ (\mc P^{-1} \boxtimes \mc P^{-1})) \boxtimes \mc P))$$
    $$ = Ip_{124*}(\ttd(-p_3^* T_{M^\vee}) \cap I\Delta_{M^\vee}^!\tau((\mc K \circ (\mc P^{-1} \boxtimes \mc P^{-1})) \boxtimes \mc P))$$
    $$ = Ip_{124*}(\ttd(-p_3^* T_{M^\vee}) \cap I\Delta_{M^\vee}^!(\tau(\mc K \circ (\mc P^{-1} \boxtimes \mc P^{-1})) \times \tau(\mc P)))$$
    \begin{multline*}
        = Ip_{124*} I\Delta_{M^\vee}^!((\td^I(-p_3^* \pi_{IB}^* T_{IB} + p_3^* T_{\iota_{M^\vee}}) \cap  \tau(\mc K \circ (\mc P^{-1} \boxtimes \mc P^{-1}))) \\ \times (\td^I(\pi_{IB}^* T_{IB} - p_1^* T_{IM^\vee}) \cap \tau(\mc P)))
    \end{multline*}
    $$ = (\td^I(p_2^* T_M) \cap \mf G) \circ (\td^I(-p_3^* \pi_{IB}^* T_{IB} + p_3^* T_{\iota_{M^\vee}}) \cap  \tau(\mc K \circ (\mc P^{-1} \boxtimes \mc P^{-1})))$$
    $$ = (\td^I(p_2^* T_M) \cap \mf G) \circ (\mf C \circ (\mf G^{-1} \times \mf G^{-1}))$$
    by (\ref{eq: C 1})
    $$ = \td^I(p_3^* T_M) \cap (\mf G \circ \mf C \circ (\mf G^{-1} \times \mf G^{-1}))$$
    by the projection formula. Thus, applying $\tau$ to (\ref{eq: C before tau}) we obtain
    $$\td^I(p_3^* T_M) \cap (\mf G \circ \mf C \circ (\mf G^{-1} \times \mf G^{-1})) = \tau(\ms O_{\Delta^s_M}) = \td(p_3^* T_M) \cap [\Delta_M^s]$$
    and as $\td^I(p_3^* T_M) = \td(p_3^* T_M)$ is invertible, we are done.
\end{proof}

\begin{proof}[Proof of Theorem \ref{thm: main sec 2}(iii)]
    By Proposition \ref{prop: C identity} we have
    $$\qg_{k+l+1} \circ [\Delta_M^s] \circ (\pg_k \times \pg_l) = \qg_{k+l+1} \circ  \mf G \circ \mf C \circ (\mf G^{-1} \times \mf G^{-1})\circ (\pg_k \times \pg_l)$$
    $$ = (\qg_{k+l+1} \circ  \mf G) \circ \mf C \circ ((\mf G^{-1} \circ \pg_k) \times (\mf G^{-1} \circ \pg_l)).$$
    Now, by Fourier vanishing (Proposition \ref{prop: FV}) we note that
    $$\mf G^{-1}_\gamma \circ \pg_k = \mf G^{-1}_\gamma \circ \pg_{\gamma,k} = \sum_{i\ge 0} \mf G_{\gamma,i}^{-1} \circ \sum_{j \le k} \mf G_{\gamma,j} \circ \mf G_{\gamma,2h-j}^{-1} = \sum_{i \ge 2h-k} \mf G_{\gamma,i}^{-1} \circ \pg_{\gamma,k} \in \Corr_B^{\ge h_\gamma-k}(M, M^{\gamma,\vee})$$
    and similarly for $\mf G^{-1}_\beta \circ \pg_k$ and
    $$\qg_{k+l+1} \circ  \mf G_\alpha =\qg_{\alpha, k+l+1} \circ \sum_{a\ge k+l+1} \mf G_{\alpha,i} \in \Corr_B^{\ge k+l+1 - h_\alpha}(M^{\alpha,\vee},M) .$$
    Note as well that
    $$\mf C_{(\gamma,\beta,\alpha)} \in \Corr_B^{\ge h_\alpha - h_\beta -h_\gamma}(M^{\gamma,\vee}, M^{\beta,\vee}; M^{\alpha,\vee})$$
    as by definition of $\mf C$ we have
    $$\dim \mf C_{(\gamma,\beta,\alpha)} = \dim \tau(\mc K)_{(\gamma,\beta,\alpha)} \le \dim (\Supp \mc K) \cap M^{\gamma,\vee} \times_B M^{\beta,\vee} \times_B M^{\alpha,\vee} \le b-h+h_\gamma+h_\beta+h_\alpha$$
    by Proposition \ref{prop: supp K endoscopic}.
    Putting this all together, we have
    $$\qg_{k+l+1} \circ [\Delta_M^s] \circ (\pg_k \times \pg_l) = \sum_{\gamma,\beta,\alpha \in \Gamma} \qg_{k+l+1} \circ  \mf G_\alpha \circ \mf C_{(\gamma,\beta,\alpha)} \circ ((\mf G^{-1}_\gamma \circ \pg_k) \times( \mf G^{-1}_\beta \circ \pg_l))$$
    $$\in \Corr_B^{\ge h_\gamma -k + h_\beta - l + h_\alpha - h_\beta - h_\gamma + k+l+1 - h_\alpha}(M,M;M) = \Corr_B^{\ge 1}(M,M;M).$$
    But since the composition
    $$\qg_{k+l+1} \circ [\Delta_M^s] \circ (\pg_k \times \pg_l) \in \Corr_B^0(M,M;M)$$
    we conclude this must be 0.
    
    Finally, since the homological realizations of $\pg_k, \pg_l, \qg_{k+l+1}$ are the perverse truncations described in Theorem \ref{thm: main sec 2}(ii) and the homological realization of $[\Delta_M^s]$ is the cup product, after taking cohomology we obtain the statement that
    $$P_kH^*(M) \cup P_lH^*(M) \xrightarrow{\cup} H^*(M) \twoheadrightarrow H^*(M)/P_{k+l}H^*(M)$$
    is trivial, i.e., the perverse filtration is multiplicative.
\end{proof}

\printbibliography
\end{document}